\documentclass[oneside,english]{amsart}
\usepackage[T1]{fontenc}
\usepackage{color}
\usepackage{amsthm}
\usepackage{amstext}
\usepackage{graphicx}
\usepackage{amssymb}
\usepackage{esint}
\usepackage[all]{xy}

\makeatletter
\numberwithin{equation}{section} 
\numberwithin{figure}{section} 
\theoremstyle{plain}
\newtheorem{thm}{Theorem}[section]
  \theoremstyle{plain}
  \newtheorem{prop}[thm]{Proposition}
  \theoremstyle{plain}
  \newtheorem{lem}[thm]{Lemma}
  \theoremstyle{remark}
  \newtheorem*{rem*}{Remark}
  \theoremstyle{remark}
  \newtheorem{rem}[thm]{Remark}
  \theoremstyle{definition}
  \newtheorem{defn}[thm]{Definition}

\def\ms{\text{-}}

\makeatother

\usepackage{babel}

\begin{document}

\title{\textup{Determinantal Quintics and Mirror Symmetry of Reye Congruences}}

\author{Shinobu Hosono and Hiromichi Takagi}
\begin{abstract}
We study a certain family of determinantal quintic hypersurfaces in
$\mathbb{P}^{4}$ whose singularities are similar to the well-studied
Barth-Nieto quintic. Smooth Calabi-Yau threefolds with Hodge numbers
$(h^{1,1},h^{2,1})=(52,2)$ are obtained by taking crepant resolutions
of the singularities. It turns out that these smooth Calabi-Yau threefolds
are in a two dimensional mirror family to the complete intersection
Calabi-Yau threefolds in $\mathbb{P}^{4}\times\mathbb{P}^{4}$ which
have appeared in our previous study of Reye congruences in dimension
three. We compactify the two dimensional family over $\mathbb{P}^{2}$
and reproduce the mirror family to the Reye congruences. We also determine
the monodromy of the family over $\mathbb{P}^{2}$ completely. Our
calculation shows an example of the orbifold mirror construction with
a trivial orbifold group. 
\end{abstract}
\maketitle
{\small \tableofcontents{}}{\small \par}

\section{\textbf{\textup{Introduction}}}

Quintic hypersurfaces in the projective space $\mathbb{P}^{4}$ have
been invaluable testing grounds for the interesting mathematical ideas
coming from the string theory. This has been for long since the historical
discovery of an exact solution of N=2 superconformal field theory/string
theory and its profound relations to the quintic hypersurfaces \cite{Gepner}.
The idea of mirror symmetry of Calabi-Yau manifolds, for example,
has been verified first by translating a special involution in the
set of N=2 theories into some operation, now-called orbifold mirror
construction, in the algebraic geometry of quintic hypersurfaces \cite{Greene-Pressor},
\cite{essay}, and also the surprising applications of the mirror
symmetry to Gromov-Witten theory have been started from the Hodge
theoretical investigations of the mirror quintic hypersurfaces \cite{Candelas1}. 

In this paper, we will be concerned with certain special types of
quintic hypersurfaces in $\mathbb{P}^{4}$ which are called determinantal
quintics. The determinantal quintics are interesting not only from
the viewpoint of the mirror symmetry but also from the viewpoint of
the classical projective geometry. In fact these quintics have appeared
in our previous study of the so-called Reye congruences in dimension
three \cite{HoTa}, where a beautiful interplay between the mirror
symmetry and the classical projective geometry has been observed.
Historically, Reye congruences represent certain Enriques surfaces,
called nodal Enriques surfaces \cite{Cossec}, \cite{Tyurin},  and
their study goes back to the 19th century, where the term `congruence`
arose in relation to the geometry of the Grassmannian $G(2,4)$. They
naturally come with K3 surfaces which admit fixed point free involutions.
In dimension three, the corresponding Reye congruences turn out to
be Calabi-Yau manifolds with non-trivial fundamental groups \cite{Oliva},
and they also come with Calabi-Yau threefolds equipped with fixed
point free involutions which we call covering Calabi-Yau threefolds
of the Reye congruences.

In our previous work \cite{HoTa}, we have studied the mirror symmetry
of the three dimensional Reye congruences through the covering Calabi-Yau
threefolds using the methods in the toric geometry \cite{BatBo}.
In this paper, we will reconsider the mirror symmetry based on the
orbifold mirror construction and will observe that the projective
geometries of certain singular determinantal quintics come into play
in an interesting way. Also we find that, in our case, the so-called
orbifold group is a trivial group, $G_{orb}=\{{\rm id}\}$. The last
property naturally leads us to a problem that how is the mirror involution
in the corresponding N=2 string theory realized in such cases, although
we will not discuss the problem in this article.

\medskip{}

The construction of this paper is as follows: In the next section
we will summarize the geometries of the Reye congruences following
the previous work. There, after setting up the notation and the problems
in details, we describe the main results of this paper. In Section
3, we calculate the topological Euler numbers of certain (singular)
determinantal quintic hypersurfaces. In Section 4, we describe the
details about the calculations of some Euler numbers needed in Section
3. In Section 5, we will obtain the mirror family to the covering
Calabi-Yau threefolds of the Reye congruences. In Section 6, we will
determine completely the monodromy properties of the mirror family.
Taking the fixed point free involution into account, we construct
the mirror family to the Reye congruences. In Section 7, we will discuss
some geometry of the singular Hessian quintics. 

\vspace{1cm}

\begin{flushleft}
\textbf{Acknowledgements:} The authors would like to thank Prof. B.
van Geemen for his kind and helpful correspondence to their question
about the \'etale cohomology. They also would like to thank Prof. C. Vafa for
his correspondence. This work is supported in part by Grant-in Aid
Scientific Research (C 22540041, S.H.) and Grant-in Aid for Young
Scientists (B 30322150, H.T.).
\par\end{flushleft}

\bigskip{}
\bigskip{}
\bigskip{}
\bigskip{}
\bigskip{}
\bigskip{}
\bigskip{}
\bigskip{}
\bigskip{}

\section{\textbf{\textup{\label{sec:Some-background} Backgrounds and summary
of main results}}}

\subsection{\label{sub:summary-Reye-I}Three dimensional Reye congruences}

Let us consider the product $\mathbb{P}^{4}\times\mathbb{P}^{4}$
of the complex projective spaces with its bi-homogeneous coordinate
$([z],[w])$. We consider a generic complete intersection $\tilde{X}_{0}$
of five $(1,1)$ divisors in the product. In terms of the bi-homogeneous
coordinates, $\tilde{X}_{0}$ may be written by $f_{1}=\cdots=f_{5}=0$
in $\mathbb{P}^{4}\times\mathbb{P}^{4}$, where we set $f_{k}:=\,^{t}zA_{k}w$
with $5\times5$ matrices $A_{1},...,A_{5}$ over $\mathbb{C}$. When
$A_{k}$ are generic, $\tilde{X}_{0}$ defines a smooth Calabi-Yau
threefold with its Hodge numbers $h^{11}(\tilde{X}_{0})=2,h^{21}(\tilde{X}_{0})=52$.
Despite this simple descriptions$,$ $\tilde{X}_{0}$ has interesting
birational geometries which we summarize in the following diagram:
\begin{equation}
\begin{matrix}\xymatrix{\tilde{X}_{0}\ar[dr]^{\pi_{2}} &  & \ar[dl]_{p_{1}}\tilde{X}_{2}\ar[dr]^{p_{2}}\\
 & Z_{2} &  & \tilde{X}_{0}^{\sharp}\;,}
\end{matrix}\label{eq:HTdiagram1}\end{equation}
where $Z_{2}$ and $\tilde{X}_{0}^{\sharp}$ are determinantal quintics
defined by \[
Z_{2}=\left\{ \;[w]\in\mathbb{P}^{4}\mid{\rm det}(A_{1}wA_{2}w...A_{5}w)=0\right\} ,\]
and \[
\tilde{X}_{0}^{\sharp}=\left\{ [\lambda]\in\mathbb{P}_{\lambda}^{4}\biggl|\text{det}\bigl(\sum_{k=1}^{5}\lambda_{k}A_{k}\bigr)=0\right\} ,\]
respectively, and $\tilde{X}_{2}$ is defined by \[
\tilde{X}_{2}=\left\{ [w]\times[\lambda]\in\mathbb{P}^{4}\times\mathbb{P}_{\lambda}^{4}\vert\, A_{\lambda}w=0\right\} ,\text{ where }A_{\lambda}=\sum_{k=1}^{5}\lambda_{k}A_{k}.\]
$\mathbb{P}_{\lambda}^{4}$ is the projective space defined from the
$\mathbb{C}$-vector space spanned by the matrices $A_{k}(k=1,...,5)$.
The maps $\pi_{2}$ and $p_{i}\,(i=1,2)$ in the diagram (\ref{eq:HTdiagram1})
are defined by the natural projections; the projection to the second
factor $\mathbb{P}^{4}\times\mathbb{P}^{4}\rightarrow\mathbb{P}^{4}$
for $\pi_{2}$, and the projections from $\mathbb{P}^{4}\times\mathbb{P}_{\lambda}^{4}$
to the first and the second factors for $p_{i}\,(i=1,2)$, respectively.
As we can see in the definitions, both $Z_{2}$ and $\tilde{X}_{0}^{\sharp}$
are quintic hypersurfaces in the respective projective spaces, and
$\tilde{X}_{2}$ is a complete intersection of five $(1,1)$ divisors
in the product $\mathbb{P}^{4}\times\mathbb{P}_{\lambda}^{4}$. When
the matrices $A_{k}$ are generic, the both $Z_{2}$ and $\tilde{X}_{0}^{\sharp}$
determine generic determinantal varieties in $\mathbb{P}^{4}$ and
$\mathbb{P}_{\lambda}^{4}$, respectively. Generic determinantal varieties
are known to be singular along codimension three loci, where the matrices
have corank two (see \cite[Lemma 3.2]{HoTa} for example). In our
case, the degree of the singular loci is 50. Hence generic $Z_{2}$
and $\tilde{X}_{0}^{\sharp}$ are singular at 50 points, where the
rank of the relevant matrices decreases to three, and actually these
consist of 50 ordinary double points {[}\textit{ibid.} Proposition
3.3{]}. We also note that $\tilde{X}_{0}$ and $\tilde{X}_{2}$ are
birational but not isomorphic in general {[}\textit{ibid.} Sect.(3-2){]}. 

The geometries of $\tilde{X}_{0}$ and $\tilde{X}_{0}^{\sharp}$ in
the above diagram fit well to the classical projective duality, since
the projective dual $(\mathbb{P}^{4}\times\mathbb{P}^{4})^{*}$ to
the\textcolor{black}{{} Segre }variety $\mathbb{P}^{4}\times\mathbb{P}^{4}\hookrightarrow\mathbb{P}^{24}$
is naturally given by the determinantal variety in the dual projective
space $(\mathbb{P}^{24})^{*}$ and $\tilde{X}_{0}^{\sharp}$ is given
by a linear section of this determinantal variety. Based on this,
in {[}\textit{\textcolor{black}{ibid.}} Sect. (3-1){]} we have called
the determinantal quintic $\tilde{X}_{0}^{\sharp}$ as the \textit{Mukai
dual of $\tilde{X}_{0}$. }

The diagram (\ref{eq:HTdiagram1}) shows further interesting properties
if we require the matrices $A_{k}$ to be symmetric. When we identify
these symmetric matrices with quadrics in $\mathbb{P}^{4}$, the projective
space $\mathbb{P}_{\lambda}^{4}$ is nothing but the 4-dimensional
linear system of the quadrics spanned by $A_{k}$, which we denote
by $P=|A_{1},A_{2},...,A_{5}|$. In general, an $n$-dimensional linear
system of quadrics in $\mathbb{P}^{n}$ is called regular if it is
base point free and satisfies a further condition \cite{Cossec},
\cite{Tyurin}. In our present case, for a regular linear system of
quadrics $P=|A_{1},A_{2},...,A_{5}|$, we have a smooth Calabi-Yau
threefold $\tilde{X}=\tilde{X}_{0}$ which admits a fixed point free
involution; $\sigma:([z],[w])\mapsto([w],[z])$. Unlike the 2-dimensional
case, this involution preserves the holomorphic three form and we
obtain a Calabi-Yau threefold $X=\tilde{X}/\langle\sigma\rangle$,
which is called a Reye congruence in dimension three \cite{Oliva}.
$X$ has the Hodge numbers $h^{1,1}(X)=1,h^{2,1}(X)=26$ and degree
$35.$ Corresponding to the diagram (\ref{eq:HTdiagram1}), we have
\begin{equation}
\begin{matrix}\xymatrix{\tilde{X}\ar[dr]\ar[d]_{/\langle\sigma\rangle} &  & \ar[dl]U\ar[dr] & Y\ar[d]^{2:1}\\
X & \quad S\quad &  & \quad H\quad.}
\end{matrix}\label{eq:ReyeX-Y-diagram}\end{equation}
Here we have adopted the historical notations $S$ and $H$ for the
determinantal varieties of symmetric matrices; $S$ will be called
the Steinerian quintic and $H$ the Hessian quintic. These belong
to the special families of the previous determinantal quintics $Z_{2}$
and $\tilde{X}_{0}^{\sharp}$, i.e., the Steinerian quintic is defined
by the equations $\det(A_{w})=0$ with $A_{w}:=(A_{1}wA_{2}w...A_{5}w$)
and similarly for the Hessian quintic with $A_{\lambda}$. However,
for the generic regular linear system $P$, $A_{w}$ is not symmetric
while $A_{\lambda}$ is. Due to this, $S$ has generically $50$ ordinary
double points while $H$ is singular along a (smooth) curve of genus
$26$ and degree 20. In our previous work, guided by the calculations
from mirror symmetry, we have found: \medskip{}

\noindent\textbf{Theorem} (\cite[Theorem 3.14]{HoTa}) There exists
a double covering $Y$ of the Hessian quintic $H$ branched along
the singular locus, which is a smooth curve of genus 26 and degree
20. $Y$ is a smooth Calabi-Yau threefold with the Hodge numbers $h^{1,1}(Y)=1,h^{2,1}(Y)=26$
and degree $10$ with respect to $\mathcal{O}_{Y}(1)$.

\medskip{}

An explicit description of the covering $Y$ will be given in Definition$\,$\ref{def:def-Ysp}
and Remark$\,$\ref{rem:def-covering-Y}. 

We can observe an interesting projective duality behind the diagram
(\ref{eq:ReyeX-Y-diagram}). This time the projective dual we start
with is the dual $({\rm Sym}^{2}\mathbb{P}^{4})^{*}$ associated to
the embedding ${\rm Sym}^{2}\mathbb{P}^{4}\hookrightarrow\mathbb{P}^{14}$
by the Chow form. This duality has quite similar properties to that
of the Grassmannians under $G(2,n)\hookrightarrow\mathbb{P}^{\frac{1}{2}n(n-1)-1},$
which appeared in \cite{Rod}, \cite{BoCa}, \cite{Kuz}. Observing
this similarity, and also from the mirror symmetry, it has been conjectured
that the Calabi-Yau threefolds $X$ and $Y$ in the diagram have the
equivalent derived categories of coherent sheaves although they are
not birational (see \cite{Hori}, \cite{Jo} and references therein
for physical arguments on this). 

\medskip{}

Our main objective in this paper is to construct 'the mirror diagrams'
to the two diagrams (\ref{eq:HTdiagram1}) and (\ref{eq:ReyeX-Y-diagram}).
For this, we start with the orbifold mirror construction of $\tilde{X}_{0}$. 

\medskip{}

\global\long\def\Ssp{Z_{2}^{sp}}
\global\long\def\Usp{\tilde{X}_{2}^{sp}}
\global\long\def\Xsp{\tilde{X}_{0}^{sp}}

\global\long\def\Sspp{Z_{1}^{sp}}
\global\long\def\Uspp{\tilde{X}_{1}^{sp}}

\global\long\def\Hsp{\tilde{X}_{0}^{sp,\sharp}}

\global\long\def\cXsp{\tilde{X}_{0}^{*}}

\global\long\def\Xspi{\tilde{X}_{0}^{sp,(1)}}
\global\long\def\Xspii{\tilde{X}_{0}^{sp,(2)}}
\global\long\def\Xspiii{\tilde{X}_{0}^{sp,(3)}}

\subsection{Orbifold mirror construction of $\tilde{X}_{0}$}

Orbifold mirror constructions in general consist of the following
three main steps: Given a generic complete intersection Calabi-Yau
manifold (CICY) in a product of (weighted) projective spaces, we first
consider it in its deformation family. Then, secondly we try to find
a suitable special family of the generic deformation family. In general,
we encounter singularities in the generic members of the special family.
We may seek crepant resolutions of them at this point or defer them
to the next step since crepant resolutions may not exist at this point.
As the third step, we try to find a suitable finite group $G_{orb}$
which acts on generic fibers of the family and preserves holomorphic
three forms on them. $G_{orb}$ is required to have the property that
we have the mirror relations in the Hodge numbers when we take the
quotient (orbifold) of the generic fibers and after making crepant
resolutions of the singularities, if any. 

Apart from the hypersurfaces of Fermat type in the weighted projective
spaces \cite{Greene-Pressor}, \cite{BatPdual}, the existence of
the suitable special family and also $G_{orb}$ is based on case-by-case
studies for general CICY's (see \cite{BergHub} for Calabi-Yau hypersurfaces
of non-Fermat type).

In our case of the complete intersection $\tilde{X}_{0}$, we first
consider the following special (two dimensional) family of the complete
intersection:

\begin{align}
z_{1}w_{1}+a\, z_{2}w_{1}+b\, z_{1}w_{2}=0, & \quad z_{2}w_{2}+a\, z_{3}w_{2}+b\, z_{2}w_{3}=0,\nonumber \\
z_{3}w_{3}+a\, z_{4}w_{3}+b\, z_{3}w_{4}=0, & \quad z_{4}w_{4}+a\: z_{5}w_{4}+b\: z_{4}w_{5}=0,\label{eq:defeqsCICY}\\
z_{5}w_{5}+a\, z_{1}w_{5}+b\, z_{5}w_{1}=0,\nonumber \end{align}
where $a$ and $b$ are the parameters of the family. In what follows
in this paper, by $f_{k}$=$\,^{t}zA_{k}w\;(k=1,..,5)$ we represent
the above defining equations, i.e., we set \begin{align*}
A_{1} & =\left(\begin{smallmatrix}1 & b & 0 & 0 & 0\\
a & 0 & 0 & 0 & 0\\
0 & 0 & 0 & 0 & 0\\
0 & 0 & 0 & 0 & 0\\
0 & 0 & 0 & 0 & 0\end{smallmatrix}\right),\; A_{2}=\left(\begin{smallmatrix}0 & 0 & 0 & 0 & 0\\
0 & 1 & b & 0 & 0\\
0 & a & 0 & 0 & 0\\
0 & 0 & 0 & 0 & 0\\
0 & 0 & 0 & 0 & 0\end{smallmatrix}\right),\; A_{3}=\left(\begin{smallmatrix}0 & 0 & 0 & 0 & 0\\
0 & 0 & 0 & 0 & 0\\
0 & 0 & 1 & b & 0\\
0 & 0 & a & 0 & 0\\
0 & 0 & 0 & 0 & 0\end{smallmatrix}\right),\\
 & \qquad\; A_{4}=\left(\begin{smallmatrix}0 & 0 & 0 & 0 & 0\\
0 & 0 & 0 & 0 & 0\\
0 & 0 & 0 & 0 & 0\\
0 & 0 & 0 & 1 & b\\
0 & 0 & 0 & a & 0\end{smallmatrix}\right),\; A_{5}=\left(\begin{smallmatrix}0 & 0 & 0 & 0 & a\\
0 & 0 & 0 & 0 & 0\\
0 & 0 & 0 & 0 & 0\\
0 & 0 & 0 & 0 & 0\\
b & 0 & 0 & 0 & 1\end{smallmatrix}\right).\end{align*}
We consider the above family over $(\mathbb{C}^{*})^{2}$ by taking
$(a,b)\in(\mathbb{C}^{*})^{2}$, and denote by $\Xsp$ a general fiber
of this family. This special form of the defining equations has been
chosen so that period integrals of $\Xsp$ calculated in terms of
$f_{k}$ reproduce the period integrals from the toric mirror construction
\cite{BatBo}, \cite{HKTY}, see Sect.\ref{sec:Picard-Fuchs-equations}.
The validity of this choice will be confirmed by the mirror symmetry
among the Hodge numbers (see Theorem \ref{thm:mirror-Hodge}). 

We may consider the restriction $\Xsp|_{(\mathbb{C}^{*})^{8}}$ of
$\Xsp$ to $({\bf \mathbb{C}}^{*})^{4}\times(\mathbb{C}^{*})^{4}\subset\mathbb{P}^{4}\times\mathbb{P}^{4}$. 
\begin{prop}
\label{pro:disc-Xsp}The restriction $\Xsp|_{(\mathbb{C}^{*})^{8}}$
is smooth for generic $(a,b)\in(\mathbb{C}^{*})^{2}$ and becomes
singular when the following discriminant vanishes: \begin{equation}
dis(\Xsp|_{(\mathbb{C}^{*})^{8}})=\prod_{k,l=0}^{4}(\mu^{k}\, a+\mu^{l}\, b+1),\;\;(\mu^{5}=1,\mu\not=1).\label{eq:discXandU}\end{equation}
\end{prop}
\begin{proof}
The form of the discriminant follows from the Jacobian ideal by eliminating
the homogeneous coordinates of the projective spaces. To implement
the restriction to $({\bf \mathbb{C}}^{*})^{4}\times(\mathbb{C}^{*})^{4}\subset\mathbb{P}^{4}\times\mathbb{P}^{4}$,
we impose additional relations $z_{1}z_{2}z_{3}z_{4}z_{5}=1$ and
$w_{1}w_{2}w_{3}w_{4}w_{5}=1$ to the Jacobian ideal. The eliminations
may be done by \texttt{\textit{\textcolor{black}{Macauley2}}} \cite{M2}.
\end{proof}
In the sections \ref{sub:sing-loci-Zsp} and \ref{sub:sing-loci-Xsp},
we will derive the following property (see Proposition \ref{pro:Singular-tilde-Xosp-A1s}
for details):
\begin{prop}
\label{pro:Xosp-20-A1}For $(a,b)\in(\mathbb{C}^{*})^{2}$ with non-vanishing
discriminant (\ref{eq:discXandU}), the complete intersection $\Xsp$
is singular along 20 lines of $A_{1}$-singularity which intersect
at 20 points. Local geometries about the intersections are classified
into two types, which we call $(3A_{1},\mathcal{U}_{1})$ and $(2A_{1},\mathcal{U}_{2})$,
with the 20 points being split into 10 points for each. 
\end{prop}
We determine the singular loci above essentially by the Jacobian criterion,
however straightforward calculations do not work since the Jacobian
ideal turns out to be complicated. We avoid this complication by studying
the singular loci of the determinantal quintics which are naturally
associated to $\Xsp$ (see the next subsection). Detailed analysis
will be given in Sect.\ref{sec:A-Crepant-resolution}. There, the
type of the singularities and also the configuration of them will
be determined (see Fig.\ref{fig:FigResolution-S}). The configuration
of the singular loci, consisting of 20 lines of $A_{1}$-singularity,
is similar to the Barth-Nieto quintic studied in \cite{Barth-Nieto}
(see also \cite{HulekEtAl}). While the local geometry $(3A_{1},\mathcal{U}_{1})$
has the corresponding geometry in the Barth-Nieto quintic, the geometry
$(2A_{1},\mathcal{U}_{2})$ (and also $(\partial A_{1},\mathcal{U}_{3})$
which will be introduced in Sect.$\,$\ref{sec:A-Crepant-resolution})
is new in our case. For the resolution of the singularities, as in
\cite{Barth-Nieto} (see also \cite{HulekEtAl}), we start with the
blowing-up at the 10 points of $(3A_{1},\mathcal{U}_{1})$ singularity
and continue the blowing-up along the strict transforms of the lines
in the prescribed way in Sect.\ref{sec:A-Crepant-resolution}. Then
we finally obtain the following result: 

\bigskip{}
\noindent\textbf{Main Result 1.} (Theorem \ref{thm:Thm-crepant-e100},
Theorem \ref{thm:mirror-Hodge}) \textit{For $(a,b)\in(\mathbb{C}^{*})^{2}$
with non-vanishing discriminant (\ref{eq:discXandU}), there exists
a crepant resolution $\cXsp\rightarrow\Xsp$ with the Hodge numbers:
\[
h^{1,1}(\cXsp)=h^{2,1}(\tilde{X}_{0})=52,\;\; h^{2,1}(\cXsp)=h^{1,1}(\tilde{X}_{0})=2.\]
Namely, the resolution $\cXsp$ is a mirror Calabi-Yau threefold to
$\tilde{X}_{0}$. In particular, we have a trivial finite group $G_{orb}=\{id\}$
for the orbifold mirror construction. }

\vfill\bigskip{}
\pagebreak{}

\subsection{Special determinantal quintics $\Ssp$ and $\Hsp$ }

As in the diagram (\ref{eq:HTdiagram1}), we obtain two determinantal
quintics $\Ssp$ and $\Hsp$ from $\Xsp$, which can be arranged into
the following diagram: \begin{equation}
\begin{matrix}\xymatrix{\tilde{X}_{0}^{*}\ar[r] & \Xsp\ar[dr] &  & \ar[dl]\Usp\ar[dr]\\
 &  & \Ssp &  & \Hsp,}
\end{matrix}\label{eq:SHdiagam1}\end{equation}
where we define $\Usp$ as $\tilde{X}_{2}$ in (\ref{eq:HTdiagram1}).
The first determinantal quintic $\Ssp$ is defined by the map $\pi_{2}:\Xsp\rightarrow\Ssp$
associated with the projection to the second factor $\mathbb{P}^{4}\times\mathbb{P}^{4}\rightarrow\mathbb{P}^{4}$.
The defining equation $\det(A_{1}w\, A_{2}w...A_{5}w)=0$ is given
by the following quintic:\begin{equation}
\begin{alignedat}{1} & {\rm det}{\rm \left(\begin{matrix}w_{1}+bw_{2} & 0 & 0 & 0 & aw_{5}\\
aw_{1} & w_{2}+bw_{3} & 0 & 0 & 0\\
0 & aw_{2} & w_{3}+bw_{4} & 0 & 0\\
0 & 0 & aw_{3} & w_{4}+bw_{5} & 0\\
0 & 0 & 0 & aw_{4} & w_{5}+bw_{1}\end{matrix}\right)}\\
 & =a^{5}w_{1}w_{2}w_{3}w_{4}w_{5}\\
 & \quad+\bigl(w_{1}+bw_{2}\bigr)\bigl(w_{2}+bw_{3}\bigr)\bigr(w_{3}+bw_{4}\bigl)\bigr(w_{4}+bw_{5}\bigr)\bigl(w_{5}+bw_{1}\bigr).\end{alignedat}
\label{eq:defEqS}\end{equation}
Similarly, the second determinantal quintic $\Hsp$ is defined by
$\det(\sum_{k}\lambda_{k}A_{k})=0$ with\begin{equation}
\begin{alignedat}{1} & \begin{aligned}{\rm det}\end{aligned}
\left(\begin{matrix}\lambda_{1} & b\lambda_{1} & 0 & 0 & a\lambda_{5}\\
a\lambda_{1} & \lambda_{2} & b\lambda_{2} & 0 & 0\\
0 & a\lambda_{2} & \lambda_{3} & b\lambda_{3} & 0\\
0 & 0 & a\lambda_{3} & \lambda_{4} & b\lambda_{4}\\
b\lambda_{5} & 0 & 0 & b\lambda_{4} & \lambda_{5}\end{matrix}\right)\\
 & =\bigl(1+a^{5}+b^{5}\bigr)\lambda_{1}\lambda_{2}\lambda_{3}\lambda_{4}\lambda_{5}\\
 & \quad+\, a^{2}b^{2}\bigl(\lambda_{1}\lambda_{2}^{2}\lambda_{4}^{2}+\lambda_{2}\lambda_{3}^{2}\lambda_{5}^{2}+\lambda_{3}\lambda_{4}^{2}\lambda_{1}^{2}+\lambda_{4}\lambda_{5}^{2}\lambda_{2}^{2}+\lambda_{5}\lambda_{1}^{2}\lambda_{3}^{2}\bigr)\\
 & \quad-ab\bigl(\lambda_{1}\lambda_{2}\lambda_{3}\lambda_{4}^{2}+\lambda_{2}\lambda_{3}\lambda_{4}\lambda_{5}^{2}+\lambda_{3}\lambda_{4}\lambda_{5}\lambda_{1}^{2}+\lambda_{4}\lambda_{5}\lambda_{1}\lambda_{2}^{2}+\lambda_{5}\lambda_{1}\lambda_{2}\lambda_{3}^{2}\bigr).\end{alignedat}
\label{eq:def-eq-HessAB}\end{equation}

\begin{prop}
\label{pro:dis-X-S-H}The singular loci of $\Ssp$ are in $\mathbb{P}^{4}\setminus({\bf \mathbb{C}}^{*})^{4}$
for generic $(a,b)\in(\mathbb{C}^{*})^{2}$, i.e., the restriction
$\Ssp|_{(\mathbb{C}^{*})^{4}}$ of $\Ssp$ to the torus $({\bf \mathbb{C}}^{*})^{4}\subset\mathbb{P}^{4}$
is smooth. $\Ssp|_{(\mathbb{C}^{*})^{4}}$ becomes singular for $a,b$
on the discriminant $\{dis(\Ssp|_{(\mathbb{C}^{*})^{4}})=0\}\subset(\mathbb{C}^{*})^{2}$,
where \begin{equation}
dis(\Ssp|_{(\mathbb{C}^{*})^{4}})=a^{5}\prod_{k,l=0}^{4}(\mu^{k}\, a+\mu^{l}\, b+1)\;\;\;(\mu^{5}=1,\mu\not=1).\label{eq:dis-Z}\end{equation}
Similar restriction $\Hsp|_{(\mathbb{C}^{*})^{4}}$ of $\Hsp$ is
smooth for generic $(a,b)\in(\mathbb{C}^{*})^{2}$ and becomes singular
for the values on the discriminant $\{dis(\Hsp|_{(\mathbb{C}^{*})^{4}})=0\}\subset(\mathbb{C}^{*})^{2}$,
where \begin{equation}
dis(\Hsp|_{(\mathbb{C}^{*})^{4}})=\prod_{k,l=0}^{4}(\mu^{k}\, a+\mu^{l}\, b+1)\times\prod_{k=0}^{4}(a-\mu^{k}b)^{2}\;\;\;(\mu^{5}=1,\mu\not=1).\label{eq:discriminant-Hessian-quintic}\end{equation}
\end{prop}
\begin{proof}
As in Proposition \ref{pro:disc-Xsp}, we impose the restrictions
to $({\bf \mathbb{C}}^{*})^{4}$ by adding the equations $w_{1}w_{2}...w_{5}=1$
or $\lambda_{1}\lambda_{2}...\lambda_{5}=1$ to the Jacobian ideals.
By the eliminations, we obtain the claimed forms of the discriminants. 
\end{proof}
$\tilde{X}_{2}^{sp}$ is defined by special forms of five $(1,1)$-divisors
in $\mathbb{P}^{4}\times\mathbb{P}_{\lambda}^{4}.$ We can verify
that the restriction to the tori $(\mathbb{C}^{*})^{8}=(\mathbb{C}^{*})^{4}\times(\mathbb{C}^{*})^{4}$
is smooth and has the following form of the discriminant:\begin{equation}
dis(\tilde{X}_{2}\vert_{(\mathbb{C}^{*})^{8}})=\prod_{k,l=0}^{4}(\mu^{k}\, a+\mu^{l}\, b+1),\;\;(\mu^{5}=1,\mu\not=1).\label{eq:dis-U}\end{equation}

\begin{prop}
\label{pro:singular-loci-S-and-H}1) For $(a,b)\in(\mathbb{C}^{*})^{2}$
with non-vanishing discriminant (\ref{eq:dis-Z}), the determinantal
quintic $\Ssp$ is singular along 5 coordinate lines, each of them
is of $A_{2}$ type, and singular also along 10 lines of $A_{1}$
singularity. These lines intersect at 15 points. 2) For $(a,b)\in(\mathbb{C}^{*})^{2}$
with nonvanishing (\ref{eq:discriminant-Hessian-quintic}), the determinantal
quintic $\Hsp$ is singular along 5 coordinate lines, each of which
is of $A_{3}$ type, and singular also along 5 additional lines of
$A_{1}$ singularity. These lines intersect at 10 points.\end{prop}
\begin{proof}
We present the details of 1) in Sect.$\,$\ref{sec:A-Crepant-resolution}
and Fig.$\,$\ref{fig:FigResolution-S}. The property of 2) is obtained
in \ref{pro:singular-loci-Hsp} (see Fig.$\,\ref{fig:U-Hessian}).$
\end{proof}
The complete intersections $\Xsp$ and $\Usp$ give partial crepant
resolutions of $\Ssp$ and $\Hsp$, respectively. In fact, all the
singularities along the 15 lines in $\Ssp$ are (partially) resolved
to the singularities of $A_{1}$ type along the 20 lines in Proposition
\ref{pro:Xosp-20-A1}. The similar property also holds for the projection
$\mbox{\ensuremath{\Usp}}\rightarrow\Hsp$ (cf. Fig.$\,$\ref{fig:U-Hessian}).

Precisely, the crepant resolution $\cXsp$ of $\tilde{X}_{0}^{sp}$
(claimed in Main Result 1) is valid for $(a,b)\in(\mathbb{C}^{*})^{2}$
being away from the zero-loci of the discriminant (\ref{eq:discXandU})
in $(\mathbb{C}^{*})^{2}$. It is easy to see that $\tilde{X}_{0}^{sp}$
with two different values of $(a,b)$ and $(\mu^{k}a,\mu^{l}b)$ $(\mu^{5}=1)$
are isomorphic to each other by a simple coordinate change. Based
on this, we introduce the affine variables $x=-a^{5}$, $y=-b^{5}$
to have a smooth family over $(\mathbb{C}^{*})^{2}\ni(x,y)$, which
will be compactified to a family over $\mathbb{P}^{2}$ (see Sect.$\,\ref{sec:Picard-Fuchs-equations})$.
\smallskip{}

\noindent\textbf{Main Result 2.} (Propositions \ref{pro:Monodromy-Matrices},
\ref{pro:connection-matrix-C}, \ref{pro:Monodromy-relations-1)-3)})
\textit{Let ${\normalcolor {\normalcolor \tilde{\mathfrak{X}}^{*}}}$
be the family of Calabi-Yau manifolds $\cXsp$ over $\mathbb{P}^{2}$,
and consider the period integrals of the family. Then, the integral
and symplectic basis of the period integrals are generated by the
cohomology-valued hypergeometric series defined in \cite[Conj. 2.2]{CentralCh}
and \cite[Prop.1]{IIAmonod}. }

\smallskip{}
We remark that our crepant resolution $\tilde{X}_{0}^{*}\rightarrow X_{0}^{sp}$
is valid also for $a=b\in\mathbb{C}^{*}$ as far as we have non-vanishing
discriminant (\ref{eq:discXandU}). Hence we can consider the restriction
of the family $\tilde{\mathfrak{X}}^{*}$ over $\mathbb{P}^{2}$ to
a family over $\{x=y\}\cong\mathbb{P}^{1}$ and have the following
properties:\textit{ }\smallskip{}

\noindent\textbf{Main Result 3.}\textit{ }(Proposition \ref{pro:mirror-Reye-X})\textit{
Over the 'diagonal' $\{x=y\}\cong\mathbb{P}^{1}$, except $x=y=\frac{1}{32}$,
the family $\tilde{\mathfrak{X}}^{*}\rightarrow\mathbb{P}^{2}$ admits
a fiberwise fixed point free involution found in \cite[Prop. 2.9]{HoTa}.
By taking the fiberwise unramified quotient under this involution
over $\mathbb{P}^{1}$, we obtain the mirror family $\mathfrak{X}_{\mathbb{P}^{1}}^{*}$
of the Reye congruence $X$. In particular, the period integrals and
the monodromy matrices from Main Result 2 reproduce the previous results
obtained in \cite[Prop. 2.10, 3)]{HoTa}. }

\medskip{}
We summarize the geometries of the generic fiber $X^{*}$ of\textit{
$\mathfrak{X}_{\mathbb{P}^{1}}^{*}\rightarrow\mathbb{P}^{1}$} as
follows: 

\begin{equation}
\begin{matrix}\xymatrix{ & \tilde{X}_{0}^{*}\ar[r]\ar[d]^{/\mathbb{Z}_{2}} & \tilde{X}_{0}^{sp}\ar[dr] &  & \ar[dl]U_{sp}\ar[dr]\\
 & X^{*} &  & S_{sp} &  & H_{sp}\;,}
\end{matrix}\label{eq:diag-Reye-X-mirror}\end{equation}
where $S_{sp}$ and $H_{sp}$ are the special forms of the Steinerian
quintic and the Hessian quintic defined by (\ref{eq:defEqS}) and
(\ref{eq:def-eq-HessAB}) with $a=b$, respectively. 

\medskip{}
We close this section noting some properties of the Hessian quintic
$H_{sp}$. When $a=b$, the discriminant (\ref{eq:discriminant-Hessian-quintic})
of the Hessian $H_{sp}$ vanishes. In Sect.\ref{sec:Xs-and-Ys-discussions},
we will explain this (Proposition \ref{pro:singular-loci-Hsp}) by
observing that an elliptic normal quintic appears as a new component
of the singular loci of $\Hsp$ when $a=b$. There we will also discuss
that the Hessian quintic $H_{sp}$ admits a double covering $Y_{sp}$
ramified along its singular loci. From the mirror symmetry considerations
given in \cite{HoTa}, it is expected that there is a crepant resolution
$Y_{sp}^{*}$ of $Y_{sp}$ which gives a mirror Calabi-Yau threefold
$Y^{*}(=Y_{sp}^{*})$ to the $Y$ of the Reye congruence $X$. Namely
we expect that the pair $(X,Y)$ of Calabi-Yau manifolds associated
with the Reye congruence is mirrored to another pair $(X^{*},Y^{*})$
of the mirror Calabi-Yau manifolds. Here $Y^{*}$ can be either birational
to $X^{*}$ or a Fourier-Mukai partner to $X^{*}$. Both cases are
consistent with the homological mirror symmetry \cite{Ko}. The construction
of $Y_{sp}^{*}$ is left for future study. 

\vspace{2cm}

\section{\textbf{\textup{\label{sec:The-Euler-numbersSH}The Euler numbers
$e(\Ssp)$ and $e(\Hsp)$}}}

In this section, we determine the Euler numbers of the determinantal
quintics. Since these quintics are singular, we invoke to a topological
method. We assume that $(a,b)\in(\mathbb{C}^{*})^{2}$ is away from
the zero of the discriminant (\ref{eq:dis-Z}) and (\ref{eq:discriminant-Hessian-quintic}),
respectively, for the determinantal quintics $\Ssp$ and $\Hsp$.

\subsection{Euler number $e(\Ssp)$}

We compute the Euler numbers of the singular determinantal quintic
$\Ssp$ by considering the intersections of $\Ssp$ in (\ref{eq:defEqS})
with the following affine line $l\subset\mathbb{P}^{4}$ such that
$l\cup\{v_{0}\}=\mathbb{P}^{1}$ with $v_{0}=[0:0:0:0:1]\in\Ssp$:
\[
l:\;[w_{1}:w_{2}:w_{3}:w_{4}:w_{5}]=[x_{1}:x_{2}:x_{3}:x_{4}:t]\;\;\;(t\in{\bf {\bf \mathbb{C}}}).\]
Substituting the coordinates of this line into the defining equation
(\ref{eq:defEqS}), we obtain \[
f(t)=c_{2}t^{2}+c_{1}t+c_{0,}\]
where\begin{equation}
\begin{aligned}c_{2}=b(x_{1}+bx_{2})(x_{2}+bx_{3})(x_{3}+bx_{4}),\quad\\
c_{1}=a^{5}x_{1}x_{2}x_{3}x_{4}+\frac{1}{b}c_{2}(b^{2}x_{1}+x_{4}),\;\; c_{0}=x_{1}x_{4}c_{2}\;.\end{aligned}
\label{eq:ckS}\end{equation}
 The equation $f(t)=0$ determines the intersection of $l$ with $\Ssp$
as the fiber over each point $[x_{1}:x_{2}:x_{3}:x_{4}:0]$ associated
to the projection $\mathbb{P}^{4}\setminus\{v_{0}\}\rightarrow\mathbb{P}^{3}$.
Then, by counting the numbers of the solutions, we can calculate the
Euler number $e(\Ssp)$. The fiber over each point varies depending
on the values of $c_{2},c_{1,}c_{0}$, and the followings are two
extreme cases: 1) $c_{2}=c_{1}=c_{0}=0$, and 2) $c_{2}=c_{1}=0$
but $c_{0}\not=0$. We regard that the fiber over the former loci
is $\mathbb{P}^{1}=l\cup\{v_{0}\}$. The fiber over 2) is empty, however
we may regard it as the point at infinity $\{v_{0}\}$. We see that,
in the present case, 2) does not occur since $c_{2}=0$ implies $c_{0}=0$,
however, the following arguments are not restricted to such cases.
For other cases than 1) and 2), the numbers of solutions of the equation
$f(t)=0\,(t\in{\bf \mathbb{C})}$ are either 2 or 1. Depending on
the numbers of solutions we define the following subsets in $\mathbb{P}^{3}$:
\begin{align*}
U_{\mathbb{P}^{1}} & =\bigl\{[x]\mid c_{2}=c_{1}=c_{0}=0\bigr\},\;\; U_{2}=\bigl\{[x]\mid c_{2}\not=0,\; c_{1}^{2}-4c_{2}c_{0}\not=0\bigr\},\\
 & U_{1}=\bigl\{[x]\mid c_{2}\not=0,\; c_{1}^{2}-4c_{2}c_{0}=0\bigr\}\sqcup\bigl\{[x]\mid c_{2}=0,c_{1}\not=0\bigr\}.\end{align*}
Then the Euler number is evaluated by \begin{equation}
e(\Ssp)=2\, e(U_{2})+e(U_{1})+\bigl(e(\mathbb{P}^{1})-1\bigr)\, e(U_{\mathbb{P}^{1}})+e(v_{0}).\label{eq:eulerN}\end{equation}

We denote the discriminant surface by $D_{s}$, i.e., $D_{s}=\bigl\{[x]\in\mathbb{P}^{3}\mid c_{1}^{2}-4c_{2}c_{0}=0\bigr\}.$
The subset $U_{1}$ consists of those points in an open subset of
$D_{s}$ or where $f(t)$ becomes linear. Consider the inclusions:
\begin{equation}
\{c_{2}=0\}\supset\{c_{2}=c_{1}=0\}\supset\{c_{2}=c_{1}=c_{0}=0\},\label{eq:inclusionsA}\end{equation}
and denote these by $V_{c_{2}}\supset V_{c_{2},c_{1}}\supset V_{c_{2,}c_{1},c_{0}}$
with the obvious definitions. 
\begin{lem}
$\,$\begin{equation}
e(\Ssp)=2\, e(\mathbb{P}^{3})+1-e(D_{s})-e(V_{c_{2}})+e(V_{c_{2},c_{1},c_{0}}).\label{eq:EulerA}\end{equation}
\end{lem}
\begin{proof}
By definition, we have $U_{1}=\bigl(D_{s}\setminus(D_{s}\cap V_{c_{2}})\bigr)\sqcup(V_{c_{2}}\setminus V_{c_{2},c_{1}})$
and $D_{s}\cap V_{c_{2}}=V_{c_{2},c_{1}}.$ Since the union is disjoint,
we have \[
e(U_{1})=\bigl(e(D_{s})-e(V_{c_{2},c_{1}})\bigr)+\bigl(e(V_{c_{2}})-e(V_{c_{2},c_{1}})\bigr)=e(D_{s})+e(V_{c_{2}})-2\, e(V_{c_{2},c_{1}}).\]
Similarly, we have $U_{2}=\mathbb{P}^{3}\setminus(D_{s}\cup V_{c_{2}})$
and hence\[
e(U_{2})=e(\mathbb{P}^{3})-e(D_{s}\cup V_{c_{2}})=e(\mathbb{P}^{3})-e(D_{s})-e(V_{c_{2}})+e(V_{c_{2},c_{1}}).\]
Also note that $U_{\mathbb{P}^{1}}=V_{c_{2},c_{1},c_{0}}$ holds.
Substituting all these expressions into (\ref{eq:eulerN}), the claimed
formula follows.
\end{proof}
Let us introduce the ${\bf \mathbb{C}}$-bases $e_{1},...,e_{5}$
of $\mathbb{C}^{5}$ by which we can write $[x_{1}:x_{2}:...:x_{5}]=[x_{1}e_{1}+x_{2}e_{2}+...+x_{5}e_{5}]$
for the projective space $\mathbb{P}^{4}=\mathbb{P}(\mathbb{C}^{5})$.
We define coordinate (projective) lines $L_{ij}$ and also coordinate
(projective) planes $L_{ijk}$ by \begin{align*}
L_{ij}=\langle e_{i},e_{j}\rangle,\;\; & L_{ijk}=\langle e_{i},e_{j},e_{k}\rangle,\end{align*}
where $\langle e_{i_{1}},e_{i_{2}},..,e_{i_{k}}\rangle$ represents
the projective space spanned by the vectors $e_{i_{1}},e_{i_{2}},$
$..,e_{i_{k}}$. We also define the following projective lines and
planes:\begin{align*}
L_{i,jk}=\langle e_{i,}be_{j}-e_{k}\rangle, & \;\; L_{ij,mn}=\langle e_{i},e_{j},be_{m}-e_{n}\rangle.\end{align*}

\begin{lem}
We have $e(V_{c_{2}})=4$ and $e(V_{c_{2},c_{1},c_{0}})=3.$\end{lem}
\begin{proof}
From the form of $c_{2}$ we have the following decomposition into
planes: \[
V_{c_{2}}=L_{12,34}\cup L_{41,23}\cup L_{34,12},\]
where three components are normal crossing in $\mathbb{P}^{3}$. From
this, we have $e(V_{c_{2}})=3\; e(\mathbb{P}^{2})-3\; e(\mathbb{P}^{1})+1=4.$
Observe that $V_{c_{2},c_{1}}=V_{c_{2}}\cap\{x_{1}x_{2}x_{3}x_{4}=0\}$.
From this, we deduce that \[
V_{c_{2},c_{1}}=\partial L_{12,34}\cup\partial L_{41,23}\cup\partial L_{34,12},\]
where $\partial L_{12,34}$ represents the union of the boundary 3
lines $L_{12}\cup L_{2,34}\cup L_{1,34}$, and similarly for $\partial L_{41,23}$
and $\partial L_{34,12}.$ Inspecting the intersection points of the
9 lines carefully, we evaluate $e(V_{c_{2},c_{1}})=3$. Note that
in the present case, we have $V_{c_{2},c_{1},c_{0}}=V_{c_{2},c_{1}}$,
hence $e(V_{c_{2},c_{1},c_{0}})=3.$\end{proof}
\begin{prop}
\label{pro:Ds}For $(a,b)\in(\mathbb{C}^{*})^{2}$ with non-vanishing
discriminant (\ref{eq:dis-Z}), the discriminant $D_{s}$ is an irreducible,
singular octic surface in $\mathbb{P}^{3}$ with its Euler number
$e(D_{s})=18$.\end{prop}
\begin{proof}
We defer the detailed calculations to the next section.
\end{proof}
Using the above Proposition and the preceding two Lemmas, we evaluate
the Euler number $e(\Ssp)=9-18-4+3=-10$. We remark that the arguments
above are still valid for non-vanishing $a=b$ as long as $a,b$ are
away from the zero of the discriminant (\ref{eq:dis-Z}).
\begin{prop}
\label{pro:EulerS}For $(a,b)\in(\mathbb{C}^{*})^{2}$ with non-vanishing
discriminant (\ref{eq:dis-Z}), the determinantal quintic $\Ssp$
(\ref{eq:defEqS}) has its topological Euler number $e(\Ssp)=-10.$
Also for $a=b\in\mathbb{C}^{*}$ with the same property, we have $e(S_{sp})=-10$.
\vspace{1cm}

\end{prop}

\subsection{Euler number $e(\Hsp)$}

For $(a,b)\in(\mathbb{C}^{*})^{2}$ with non-vanishing discriminant
(\ref{eq:discriminant-Hessian-quintic}), similar calculations apply
to the determinantal quintic $\Hsp$ given in (\ref{eq:def-eq-HessAB}).
Let us first consider the affine line \[
l:\;[\lambda_{1}:\lambda_{2}:\lambda_{3}:\lambda_{4}:\lambda_{5}]=[x_{1}:x_{2}:x_{3}:x_{4}:t]\;\;\;(t\in\mathbb{C})\]
such that $l\cup\{v_{0}\}=\mathbb{P}^{1}$ with $v_{0}=[0:0:0:0:1]\in\Hsp$.
Substituting the coordinates into the defining equation of $\Hsp$,
we obtain $f(t)=c_{2}t^{2}+c_{1}t+c_{0}$ with \begin{align*}
c_{2}= & a^{2}b^{2}(x_{2}x_{3}^{2}+x_{4}x_{2}^{2})-\; ab\, x_{2}x_{3}x_{4},\\
c_{1}= & (1+a^{5}+b^{5})x_{1}x_{2}x_{3}x_{4}+a^{2}b^{2}x_{1}^{2}x_{3}^{2}-\; ab(x_{3}x_{4}x_{1}^{2}+x_{4}x_{1}x_{2}^{2}+x_{1}x_{2}x_{3}^{2}),\\
c_{0}= & a^{2}b^{2}(x_{1}x_{2}^{2}x_{4}^{2}+x_{3}x_{4}^{2}x_{1}^{2})-\; ab\, x_{1}x_{2}x_{3}x_{4}^{2}.\end{align*}
This time, it turns out that the discriminant surface $D_{s}=\{c-4c_{2}c_{0}=0\}$
in $\mathbb{P}^{3}$ consists of two irreducible components $D_{s}^{1}$
and $D_{s}^{2}:=\{x_{1}=0\}$, where the component $D_{s}^{1}$ is
an irreducible, singular septic in $\mathbb{P}^{3}.$ We may verify
these properties by\texttt{\textit{ Macaulay2.}} The general formula
(\ref{eq:EulerA}) is still valid for the present case of $e(\Hsp)$,
since it is topological. However we see some complications in the
necessary calculations, which we will sketch briefly below. 

We use \texttt{\textit{Macaulay2}} for the calculations $e(V_{c_{2}})$and
$e(V_{c_{2},c_{1},c_{0}}).$ For these Euler numbers, we make suitable
primary decompositions of the ideals of $V_{c_{2}}$ and $V_{c_{2},c_{1},c_{0}}$,
respectively. From the decompositions, we obtain \begin{equation}
V_{c_{2}}=L_{134}\cup Cone([e_{1}],C_{0}),\label{eq:A2A21decomp}\end{equation}
where $C_{0}$ is a plane conic defined by $C_{0}:=V(x_{1},ab\, x_{3}^{2}+ab\, x_{2}x_{4}-x_{3}x_{4})$
in $\mathbb{P}^{3}$ and $Cone([e_{1}],C_{0})$ is the cone over $C_{0}$
from the vertex $[e_{1}]\in\mathbb{P}^{3}$. Also we have \[
V_{c_{2},c_{1},c_{0}}=C_{0}\cup L_{12}\cup L_{14}\cup L_{34}\cup\{q_{1,}q_{2}\},\]
where the set of two points $\{q_{1,}q_{2}\}$ is given by the intersection
of the plane $a^{2}b^{2}\; x_{2}-(a^{5}+b^{5})x_{4}=0$ with the (space)
conic $Q$ in $\mathbb{P}^{3}$ defined by \[
Q=V\bigl(a^{2}b^{2}x_{1}-(a^{5}+b^{5})x_{3},ab\, x_{3}^{2}+abx_{2}x_{4}-x_{3}x_{4}\bigr).\]

\begin{prop}
\label{pro:Ds-Hessian}For $(a,b)\in(\mathbb{C}^{*})^{2}$ with non-vanishing
discriminant (\ref{eq:discriminant-Hessian-quintic}), the topological
Euler number of the determinantal quintic $\Hsp$ is given by $e(\Hsp)=11-e(D_{s})$,
where $e(D_{s})$ is the Euler number of the (reducible) discriminant
octic surface. \end{prop}
\begin{proof}
For the numbers $e(V_{c_{2}})$ and $e(V_{c_{2},c_{1},c_{0}})$, it
suffices to see the intersections of each component of the respective
irreducible decompositions. For the former, we see\[
L_{134}\cap Cone([e_{1}],C_{0}))=L_{14}\cup L'_{1,34},\]
where $L'_{i,jk}=\langle e_{i},e_{j}+ab\, e_{k}\rangle$ represents
the lines generated by the two vectors indicated. Using this, we obtain
\begin{align*}
e(V_{c_{2}}) & =e(L_{134}\cup Cone([e_{1}],C_{0}))\\
 & =e(L_{134})+e(Cone([e_{1}],C_{0}))-e(L_{14}\cup L'_{1,34}),\end{align*}
which we evaluate as $3+3-3=3$. For the latter $e(V_{c_{2},c_{1},c_{0}})$,
we note that the two points $q_{1}$ and $q_{2}$ do not lie on any
other components for the values of $a,b$. We also note \[
C_{0}\cap(L_{12}\cup L_{14}\cup L_{34})=\{[e_{2}],[e_{4}],[e_{3}+ab\, e_{4}]\}.\]
Looking the configurations of the lines $L_{12}\cup L_{14}\cup L_{34}$,
we see two intersection points among the lines. Taking into account
$3+2=5$ intersection points in total, we finally evaluate the Euler
number as \[
e(V_{c_{2},c_{1},c_{0}})=e(C_{0}\cup L_{12}\cup L_{14}\cup L_{34}\cup\{q_{1,}q_{2}\})=2\times4+2-5=5.\]
 The claim follows from the general formula (\ref{eq:EulerA}), i.e.,
$e(\Hsp)=9-e(D_{s})-3+5$.\end{proof}
\begin{rem*}
In Proposition$\,\ref{pro:Ds-Hessian}$, we assumed non-vanishing
discriminant (\ref{eq:discriminant-Hessian-quintic}) and $(a,b)\in(\mathbb{C}^{*})^{2}$.
However, as we see in the arguments above, Proposition$\,\ref{pro:Ds-Hessian}$
holds also for $a=b\in\mathbb{C}^{*}$ as long as the factor $\prod_{k,l=0}^{4}(\mu^{k}a+\mu^{l}b+1)$
of the discriminant does not vanish. $\qquad${[}{]}\end{rem*}
\begin{prop}
\label{pro:Ds3}The reducible octic surface $D_{s}=D_{s}^{1}\cup D_{s}^{2}$
has the topological number $e(D_{s})=21$ (resp. $16$) for $(a,b)\in(\mathbb{C}^{*})^{2}$
with non-vanishing discriminant (\ref{eq:discriminant-Hessian-quintic})
(resp. for $a=b\in\mathbb{C}^{*}$ with $\prod_{k,l=0}^{4}(\mu^{k}a+\mu^{l}b+1)\not=0$).
Hence we have $e(\Hsp)=-10$ and also $e(H_{sp})=-5$.\end{prop}
\begin{proof}
We briefly sketch the calculations of $e(D_{s})$ in the next section.
We evaluate the Euler numbers by $e(\Hsp)=11-e(D_{s})$.
\end{proof}
\vspace{2cm}

\section{\textbf{\textup{Calculations of the Euler numbers $e(D_{s})$}}}

This section is devoted to rather technical calculations of the Euler
number $e(D_{s})$ appeared in Propositions \ref{pro:Ds}, \ref{pro:Ds-Hessian}
and \ref{pro:Ds3}. Our method is essentially based on a similar formula
to (\ref{eq:eulerN}) which counts the number of solutions for a given
equation. Since the degrees of the relevant polynomial equations become
higher, the necessary calculations are more involved than the previous
section. For readers' convenience, we briefly summarize the technical
details required to do the calculations.

\subsection{$e(D_{s})$ for $\Ssp$}

Let us consider the determinantal quintic $\Ssp$. The octic surface
$D_{s}$ is defined as the discriminant of $f(t)=c_{2}t^{2}+c_{1}t+c_{0}:$
\[
D_{s}:\;(c_{1}^{2}-4c_{2}c_{1}=0)\subset\mathbb{P}^{3},\]
with the definitions of $c_{i}=c_{i}(x_{1},x_{2},x_{3},x_{4})$ as
in (\ref{eq:ckS}). We first note that $[0:0:0:1]\in\mathbb{P}^{3}$
is a point on $D_{s}$. We then consider an affine line $\ell:[y_{1};y_{2}:y_{3}:t]\;(t\in\mathbb{C})$
such that $\ell\cup[0:0:0:1]=\mathbb{P}^{1}$. As before, we understand
that $t=\infty$ represents $[0:0:0:1]$. The number of the intersection
points $\ell\cap D_{s}$ is determined by the number of solutions
of $g(t)=0$ with \[
g(t)=d_{4}t^{4}+d_{3}t^{3}+d_{2}t^{2}+d_{1}t+d_{0},\]
where the coefficients $d_{i}=d_{i}(y_{1},y_{2},y_{3})$ are read
from the defining octic equation of $D_{s}$. As in the previous section,
we can determine $e(D_{s})$ by carefully analyzing the numbers of
solutions of the quartic equation $g(t)=0$ parametrized by $[y_{1}:y_{2}:y_{3}]\in\mathbb{P}^{2}$.
There may appear several possibilities for the equation $g(t)=0$.
If $d_{4}\not=0$, then $g(t)=0$ is an quartic equation which has
4 roots admitting following types of multiple roots: $2+1+1,\;2+2,\;3+1,\;4.$
For each type of the multiple roots, we can determine the corresponding
component of the discriminant of $g(t)$ as follows: As an example,
consider the case of $2+1+1,$ i.e., one double roots and two simple
roots. We assume the following forms for $g(t)$: \[
g(t)=d_{4}(t-\alpha)^{2}(t-\beta)(t-\gamma)=d_{4}t^{4}+d_{3}t^{3}+d_{2}t^{2}+d_{1}t+d_{0},\]
and read an ideal in $\mathbb{C}[\alpha,\beta,\gamma,y_{1},y_{2},y_{3}]$
by comparing the coefficients of $t^{k}(k=0,..,4)$ in the second
equality. Then the elimination ideal in $\mathbb{C}[y_{1},y_{2,}y_{3}]$
determines the Zariski closure of the components where we have multiple
roots of type $2+1+1$. The loci of the other types of multiple roots
can be analyzed in a similar way. 
\begin{lem}
For $(a,b)\in(\mathbb{C}^{*})^{2}$ with non-vanishing (\ref{eq:dis-Z}),
the equation $g(t)=0$ has multiple roots of type $2+1+1$ over the
generic points of a plane curve $C$ of degree 9. $C$ is singular
at 2 points of $A_{1}$singularity, 3 points of $E_{6}$ singularity,
and 2 points of $E_{12}$ singularity (according to Arnold's classification
\cite{AGV}). 
\end{lem}
\begin{figure}
\includegraphics[scale=0.3]{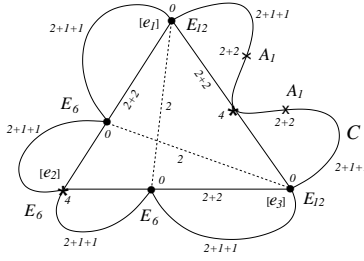}

\caption{\label{fig:eDs}Several components of the discriminant of $g(t)$
are drawn in $\mathbb{P}^{2}=\langle e_{1},e_{2},e_{3}\rangle$. The
singular plane curve $C$ of degree 9 is the main component where
we have quartics $g(t)=0$ with multiple roots of type $2+1+1$. Over
the broken lines, $g(t)=0$ becomes quadrics with double roots. Also
over the 4 points, indicated by $\bullet,$ the equation $g(t)=0$
becomes an identity $g(t)\equiv0$. Over the complement of the discriminant
in $\mathbb{P}^{2}$, we have quartics $g(t)=0$ with only simple
roots. See the text for more details.}

\end{figure}

Doing similar calculations, we can stratify the discriminant loci
of the equation $g(t)=0\,(d_{4}\not=0)$. Incorporating the cases
where $d_{4}=0,$ we have summarized the entire picture of the degeneracies
of the solutions for the equation $g(t)=0$ in Fig.\ref{fig:eDs}:
For the generic points on the nonic curve $C,$ the equation has the
multiplicity $2+1+1$ and this changes at special points as shown.
Since the equation of $C$ is lengthy, we refrain from writing it
here. Over the other components, the multiplicities may be seen from
the following forms of the polynomial $g(t)$: Over the coordinate
lines $\langle e_{1},e_{2}\rangle,\langle e_{2},e_{3}\rangle,\langle e_{3},e_{1}\rangle,$
respectively, $g(t)$ are given by $g(t)=b^{2}y_{2}^{2}(y_{1}+by_{3})^{2}t^{2}(t-b^{2}y_{1})^{2},\; b^{2}y_{2}^{2}(y_{2}+by_{3})^{2}t^{2}(y_{3}+bt)^{2}$
and $b^{2}y_{1}^{2}y_{3}^{2}(t-b^{2}y_{1})^{2}(bt+y_{2})^{2}.$ Over
the (broken) lines $\langle e_{1},e_{23}\rangle$ and $\langle e_{3},e_{12}\rangle,$
respectively, $g(t)$ becomes quadrics of the form $a^{10}b^{2}y_{1}^{2}y_{3}^{4}t^{2}$
and $a^{10}b^{2}y_{2}^{4}y_{3}^{2}t^{2}.$ 
\begin{prop}
We have $e(D_{s})=18$. \end{prop}
\begin{proof}
We first calculate the Euler number of the curve $C$ as $e(C)=-10.$
We can determine this number by representing $\mathbb{P}^{2}$ as
the cone from $[0:0:1]$ over $\mathbb{P}^{1}=\{z=0\}$. Or one can
obtain the same number by taking into account the vanishing cycles
of the singularities $3\times E_{6},2\times E_{12}$ and $2\times A_{1}$
to the Euler number of smooth plane curve of degree 9: $e(C)=2-2g+3\times6+2\times12+2\times1=-10.$ 

Now we count the numbers of the solutions $g(t)=0$ with forgetting
multiplicities from the preceding Lemma and Fig.$\,\ref{fig:eDs}$.
Four solutions are possible only for $g(t)$ being a quartic with
only simple roots. This occurs over $\mathbb{P}^{2}\setminus\bigl(C\cup(5\text{ lines})\bigr)$,
where 5 lines are those depicted in the figure. The case of three
solutions are given over $C\setminus(8\text{ point}s$) as we see
in the figure. The case of two solutions occurs over three coordinate
lines $\langle e_{1},e_{2}\rangle,\langle e_{2},e_{3}\rangle,\langle e_{3},e_{1}\rangle$
except three points for each, and also two points of $A_{1}$ singularity
on $C$. The case of one solution occurs over the two broken lines
in the figure except two points ($\bullet$'s ) for each, and also
over the two points indicated by $\mathbf{*}\;$. Over the four points
shown by $\bullet$ in the figure, we have $g(t)\equiv0$, i.e., the
entire $\mathbb{P}^{1}$ as the 'solutions'. 

For each case above, we evaluate the Euler number of the corresponding
loci. Summing up all the cases, we evaluate $e(D_{s})$ as \begin{align*}
e(D_{s}) & =4\bigl\{ e(\mathbb{P}^{2})-e(C)-5(e(\mathbb{P}^{1})-3)-1\bigr\}+3\bigl\{ e(C)-8\bigr\}\\
 & \quad+2\bigl\{3(e(\mathbb{P}^{1})-3)+2\bigr\}+1\bigl\{2+2(e(\mathbb{P}^{1})-3)+1\bigr\}+4e(\mathbb{P}^{1})-3\\
 & =4(3+10+4)+3(-18)+2(-1)+1+8-3=18.\end{align*}

\end{proof}

\subsection{$e(D_{s})$ for $\Hsp$}

For this case, we consider again an affine line $\ell:[y_{1}:y_{2}:y_{3}:t](t\in\mathbb{C})$
such that $\ell\cup[0:0:0:1]=\mathbb{P}^{1}$. This time we have $g(t)=d_{3}t^{3}+d_{2}t^{2}+d_{1}t+d_{0}$
for the equation $g(t)=0$ which determines the intersection $\ell\cap D_{s}$.
Although $g(t)$ looks simpler than the previous section, the stratification
of the discriminant of the equation $g(t)=0$ turns out to be more
complicated. For example, for $(a,b)\in(\mathbb{C}^{*})^{2}$ with
non-vanishing (\ref{eq:discriminant-Hessian-quintic}), we have a
singular irreducible curve of degree 9 for the locus of the multiplicity
$2+1$ which intersects with other components at many points in a
rather complicated way. For $a=b\in\mathbb{C}^{*}$ with $\prod_{k,l=0}^{4}(\mu^{k}a+\mu^{l}b+1)\not=0$,
this irreducible curve split into two smooth cubics and simplifies
the stratification slightly.

Since the calculations are essentially the same as in the previous
subsection, we omit the details here. After careful analysis, we obtain: 
\begin{prop}
We have $e(D_{s})=21$ (resp.$16$) for $(a,b)\in(\mathbb{C}^{*})^{2}$
with non-vanishing (\ref{eq:discriminant-Hessian-quintic}) (resp.
for $a=b\in\mathbb{C}^{*}$ with $\prod_{k,l=0}^{4}(\mu^{k}a+\mu^{l}b+1)\not=0$)
.
\end{prop}
\vspace{2cm}

\section{\label{sec:A-Crepant-resolution}\textbf{\textup{Crepant resolutions
$\cXsp\rightarrow\Xsp$}}}

To study the resolution of $\Xsp$ it will be convenient to extend
our diagram (\ref{eq:SHdiagam1}) to

\begin{equation}
\begin{matrix}\xymatrix{ & \ar[dl]\,\Uspp\ar[dr] &  & \ar[dl]_{\pi_{1}}\Xsp\ar[dr]^{\pi_{2}} &  & \ar[dl]\Usp\ar[dr]\\
\Hsp &  & \Sspp &  & \Ssp &  & \Hsp,}
\end{matrix}\label{eq:SHdiagram2}\end{equation}
where $\pi_{1},\pi_{2}$ represent the projections to the first and
the second factors of $\mathbb{P}^{4}\times\mathbb{P}^{4}$, respectively,
and \begin{align*}
\Sspp & =\left\{ \;[z]\in\mathbb{P}^{4}\mid{\rm \det}(\,^{t}zA_{1}\,\,^{t}zA_{2}...\,^{t}zA_{5})=0\right\} ,\\
\Uspp & =\bigl\{\;[z]\times[\lambda]\in\mathbb{P}^{4}\times\mathbb{P}^{4}\mid\,^{t}zA_{\lambda}=0\bigr\}.\end{align*}
Note that the same quintic hypersurface $\Hsp=\{\det(A_{\lambda})=0\}$
appears twice in the diagram. Note also that the defining equation
of $\Sspp$ may be obtained from $\Ssp$ by simply exchanging $a$
and $w_{i}$ with $b$ and $z_{i}$, respectively.

\subsection{\label{sub:sing-loci-Zsp}Singular loci of $\Sspp$ and $\Ssp$ }

As introduced in Proposition \ref{pro:singular-loci-S-and-H}, the
determinantal quintic $\Ssp$ is singular along 15 lines and so is
$\Sspp$. To write down all these lines and their configurations,
we denote as before by $[e_{i}]$ the coordinate points of the projective
space $\mathbb{P}^{4}=\langle e_{1},e_{2},...,e_{5}\rangle$, which
is the second factor in the product $\mathbb{P}^{4}\times\mathbb{P}^{4}$.
Similarly, we use the notation $[\tilde{e}_{i}]$ for the first factor
$\mathbb{P}^{4}=\langle\tilde{e}_{1},...,\tilde{e}_{5}\rangle$ of
the product $\mathbb{P}^{4}\times\mathbb{P}^{4}$. For these projective
spaces, the coordinate lines are the projective lines spanned by the
coordinate points, i.e., $\langle e_{i,}e_{j}\rangle$ and $\langle\tilde{e}_{i},\tilde{e}_{j}\rangle$.
More generally, we use the notation $\langle v_{i,}v_{j}\rangle$,
$\langle v_{i,}v_{j},v_{k}\rangle,$ etc. to describe the projective
lines, planes, etc. spanned by the vectors indicated. Using this,
we define the following lines:

\begin{align*}
\tilde{q}_{i} & =\langle\tilde{e}_{i},\tilde{e}_{i+1}\rangle,\; & q_{i} & =\langle e_{i,}e_{i+1}\rangle,\\
\tilde{l}_{i} & =\langle\tilde{e}_{i\, i+1},\tilde{e}_{i+2}\rangle, & l_{i} & =\langle e_{i\, i+1},e_{i+2}\rangle,\end{align*}
where we set $\tilde{e}{}_{ij}=-a\,\tilde{e}_{i}+\tilde{e}_{j}$,
$e_{ij}=-b\, e_{i}+e_{j}$ and the indices $i,j=1,...,5$ should be
read cyclically, i.e., by modulo $5$. 

Since the quintic $\Ssp$ has a rather simple defining equation (\ref{eq:defEqS}),
we can derive the following results by using\texttt{\textsl{ Macaulay2}}
or \texttt{\textsl{Singular}} \cite{Sing3}:
\begin{prop}
\label{pro:Lemma-S}For $(a,b)\in(\mathbb{C}^{*})^{2}$ with non-vanishing
discriminant (\ref{eq:dis-Z}), the determinantal quintic $\Ssp$
is singular along the lines $q_{i}(i=1,...,5)$ with singularities
of type $A_{2}$, and also singular along $l_{i}$ $(i=1,...,5)$
and additional 5 lines (see Remark \ref{rem:Singularities-of-S})
with singularities of type $A_{1}.$ Likewise $\Sspp$ is singular
along $\tilde{q}_{i}$ of $A_{2}$ singularity, and singular along
$\tilde{l}_{i}$ $(i=1,...,5)$ and additional 5 lines of $A_{1}$-singularities. \end{prop}
\begin{proof}
These are among the properties described in Proposition \ref{pro:singular-loci-S-and-H}.
For the derivations we use the Jacobian criteria and primary decompositions
for the corresponding ideals. For each lines, taking local coordinates
of the normal bundles, we can determine the claimed types of singularities.
Since calculations are straightforward, we omit the details.
\end{proof}
\vspace{1cm}

\subsection{\label{sub:sing-loci-Xsp}Singular loci of $\Xsp$ }

As we see in the diagram (\ref{eq:SHdiagram2}), $\Xsp$ is a partial
resolution of both of $\Sspp$ and $\Ssp.$ The map: $\pi_{2}:\Xsp\rightarrow\Ssp$
is birational since the inverse image $\pi_{2}^{-1}([w])$ of a point
$[w]\in\Ssp$ is given by the left kernel of the matrix $(A_{1}wA_{2}w...A_{5}w)$,
i.e., $([z],[w])$ s.t. $\,^{t}z(A_{1}wA_{2}w...A_{5}w)=0$, which
is uniquely determined for a generic $[w]\in\Ssp$. The birational
map $\pi_{2}$ has non-trivial fibers over the loci where the matrix
has co-rank $\geq2$, and $\Xsp$ naturally defines a blow-up along
these loci introducing the projective spaces spanned by the null spaces.
The same property holds for the first projection $\pi_{1}:\Xsp\rightarrow\Sspp.$ 
\begin{prop}
\label{pro:resol-by-pi2}The birational map $\pi_{2}$ has non-trivial
fibers over the 5 coordinate lines $q_{i}\,(i=1,..,5)$, and over
the complement of these, this is an isomorphism. The fiber $\pi_{2}^{-1}([e_{i}])$
is given by the plane $\langle\tilde{e}_{i\, i+1},\tilde{e}{}_{i+2},\tilde{e}_{i+3}\rangle\simeq\mathbb{P}^{2}$,
and the inverse image $\pi_{2}^{-1}(q{}_{i})$ of the line $q_{i}$,
more precisely the closure of the inverse image of $q_{i}\setminus\{[e_{i,}],[e_{i+1}]\}$,
is isomorphic to $\mathbb{P}^{1}\times\mathbb{P}^{1}.$ Similar properties
hold also for $\pi_{1}:\Xsp\rightarrow\Sspp$ with $\pi_{1}^{-1}([\tilde{e}_{i}])=\langle e_{i,i+1},e_{i+2},e_{i+3}\rangle$. \end{prop}
\begin{proof}
By studying the left kernels of matrices $(A_{1}w$$A_{2}w$$A_{3}w$$A_{4}w$$A_{5}w$)
with $[w]\in\Ssp,$ it is straightforward to obtain the claimed properties
of $\pi_{2}.$ For the properties of $\pi_{1}$, we study the right
kernel of matrices $(\,^{t}zA_{1}\,^{t}zA_{2}\,...\,^{t}zA_{5}$)
with $[z]\in\Sspp$, where we use a convention $(\,^{t}zA_{1}\,^{t}zA_{2}\,...\,^{t}zA_{5}$):=$\,^{t}(\,^{t}A_{1}z\:^{t}A_{2}z\,...\,^{t}A_{5}z)$
for simplicity.
\end{proof}
The Jacobian criterion for the complete intersection $\Xsp$ is rather
involved, since we need to handle large ideal. In our case, however,
we can utilize the properties of the partial resolutions $\pi_{1}$and
$\pi_{2}$ efficiently. For example, we can deduce that the singular
loci of $\Xsp$ must be in the inverse images of the 15 lines in $\Sspp$(resp.
$\Ssp$) under $\pi_{1}$(resp. $\pi_{2}$) (see Proposition \ref{pro:Lemma-S}).
Combining this fact with the Jacobian criterion for $\Xsp$, we obtain
the following:
\begin{prop}
\label{pro:Singular-tilde-Xosp-A1s}For $(a,b)\in(\mathbb{C}^{*})^{2}$
with non-vanishing discriminant (\ref{eq:discXandU}), the complete
intersection $\Xsp$ is singular along the following 20 lines:\begin{equation}
\begin{aligned}Q_{i} & =\left\{ [a^{2}s\,\tilde{e}_{i}+(s+bt)\,\tilde{e}_{i+1,i+2}]\times[se_{i}+te_{i+1}]|\,[s,t]\in\mathbb{P}^{1}\right\} ,\\
\tilde{Q}_{i} & =\left\{ [s\tilde{e}_{i}+t\tilde{e}_{i+1}]\times[b^{2}s\, e_{i}+(s+at)\, e_{i+1,i+2}]|\,[s,t]\in\mathbb{P}^{1}\right\} ,\\
 & L_{i}=[\tilde{e}_{i}]\times\langle e_{i,i+1},e_{i+2}\rangle,\:\tilde{L}_{i}=\langle\tilde{e}_{i,i+1},\tilde{e}_{i+2}\rangle\times[e_{i}],(i=1,...,5),\end{aligned}
\label{eq:curves-paramet}\end{equation}
where $Q_{i}\subset\pi_{2}^{-1}(q_{i}),\;\tilde{Q}_{i}\subset\pi_{1}^{-1}(\tilde{q}_{i})$.
$L_{i}$ and $\tilde{L}_{i}$ are the proper transforms of the lines
$l_{i}$ and $\tilde{l}_{i}$ under $\pi_{2}$ and $\pi_{1}$, respectively.
The singularities along these 20 lines are of $A_{1}$ type for all,
and these lines intersect at 20 points. \end{prop}
\begin{rem}
\textcolor{black}{\label{rem:Singularities-of-S}Now we may restate
Proposition \ref{pro:Lemma-S} as follows: For $(a,b)\in(\mathbb{C}^{*})^{2}$
with non-vanishing discriminant (\ref{eq:discXandU}), the determinantal
quintic $\Ssp$ is singular along the lines $q_{i}(i=1,...,5)$ with
singularities of type $A_{2}$, and also singular along $l_{i}$ and
$\pi_{2}(\tilde{Q}_{i})$ $(i=1,...,5)$ with singularities of type
$A_{1}.$ These 15 lines intersect at 15 points. Similarly, $\Sspp$
is singular along 15 lines $\tilde{q}_{i}$, $\tilde{l}_{i}$ and
$\pi_{1}(Q_{i})$ intersecting at 15 points. }
\end{rem}
We have depicted the schematic picture of the blow-up $\pi_{2}:\Xsp\rightarrow\Ssp$
in Fig.\ref{fig:FigResolution-S}. The intersection points should
be clear in this figure.

\begin{figure}
\includegraphics[scale=0.6]{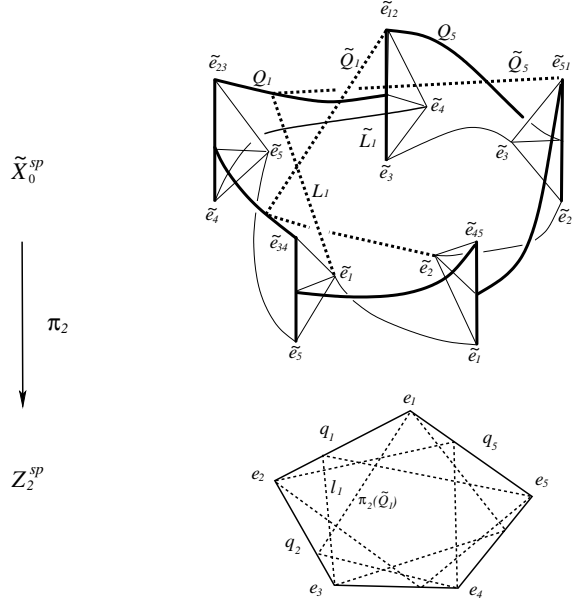}

\caption{\label{fig:FigResolution-S}The blow-up $\pi_{2}:\Xsp\rightarrow\Ssp$.
Bold lines and broken lines upstairs are lines with $A_{1}$ singularity.
Not all broken lines are drawn on the upstairs. }

\end{figure}

\medskip{}

We note that the structure of singularities in $\Xsp$ is quite similar
to that of the Barth-Nieto quintic \cite{Barth-Nieto}, where we see
20 lines of $A_{1}$ singularity intersecting at 15 points, in addition
to 10 isolated ordinary double points (called Segre points). In the
case of the Barth-Nieto quintic, the local geometries near 15 intersection
points are all isomorphic. In our case of the complete intersection
$\Xsp$, which is a partial resolution of the determinantal quintics
$\Sspp$ and $\Ssp$, the 20 intersection points of the 20 lines ($Q_{i},\tilde{Q}_{i},L_{i},\tilde{L}_{i})$
split into two isomorphic classes as we see below. Also we see in
the next sub-section a new isomorphic local geometry near the infinity
points of the 10 lines $L_{i},\tilde{L}_{i}$ (see Fig.$\,\ref{fig:The-local-affine}).$

\vspace{2cm}

\subsection{Blowing-ups of $\Xsp$}

There are two types of the intersections among the 20 singular lines
in $\Xsp$: 1) the point where 3 lines of the singularities meet,
2) the point where 2 lines meet. This should be clear from a careful
inspection of Fig.\ref{fig:FigResolution-S} and also from the symmetry
of the defining equations. We denote by $(3A_{1},\mathcal{U}_{1})$
and $(2A_{1,}\mathcal{U}_{2})$, respectively, the local geometries
around the points of type 1) and 2). These may be summarized as follows:\[
\begin{aligned}(3A_{1},\mathcal{U}_{1}) & \;\;\text{around }\tilde{Q}_{i}\cap\tilde{L}_{i}\cap Q_{i-1}\text{ and }Q_{i}\cap L_{i}\cap\tilde{Q}_{i-1}\;\;(i=1,...,5),\\
(2A_{1},\mathcal{U}_{2}) & \;\;\text{around }\tilde{L}_{i}\cap Q_{i}\text{ and }L_{i}\cap\tilde{Q}_{i}\;\;(i=1,...,5).\end{aligned}
\]
Also, from the reasons which will become clear soon (in the proof
of Proposition \ref{pro:blow-up-along-s-2A1}, 2)), we need to study
(isomorphic) local geometries around the points $[\tilde{e}_{i+2}]\times[e_{i}${]}
on $\tilde{L}_{i}$ and $[\tilde{e}_{i}]\times[e_{i+2}]$ on $L_{i}$
$(i=1,...,5)$, which we denote by $(\partial A_{1},\mathcal{U}_{3})$. 

It will be helpful to list the relevant local geometries on each lines
as follows:\begin{equation}
\begin{aligned}(\partial A_{1},\mathcal{U}_{3}),(2A_{1},\mathcal{U}_{2}),(3A_{1},\mathcal{U}_{1}) & \text{ on each }\tilde{L}_{i},\, L_{i},\\
(3A_{1},\mathcal{U}_{1}),(3A_{1},\mathcal{U}_{1}),(2A_{1},\mathcal{U}_{2}) & \text{ on each }\tilde{Q}_{i},\, Q_{i}.\end{aligned}
\label{eq:singularities-on-LQ}\end{equation}

In the following arguments, we will focus on the point $\tilde{Q}_{1}\cap\tilde{L}_{1}\cap Q_{5}=[\tilde{e}_{12}]\times[e_{1}]$
for $(3A_{1},\mathcal{U}_{1})$, and $Q_{1}\cap\tilde{L}_{1}=[-a\,\tilde{e}_{12}+\tilde{e}_{3}]\times[e_{1}]$
for $(2A_{1,}\mathcal{U}_{2})$. We will also focus on the point $[\tilde{e}_{3}]\times[e_{1}]$
on $\tilde{L}_{1}$ for $(\partial A_{1},\mathcal{U}_{3})$. See Fig.\ref{fig:The-local-affine}.
Also, in the following arguments in this subsection, we assume non-vanishing
discriminant (\ref{eq:discXandU}) and $ab\not=0$ for the parameters
$a$ and $b$.

\begin{figure}
\includegraphics[scale=0.6]{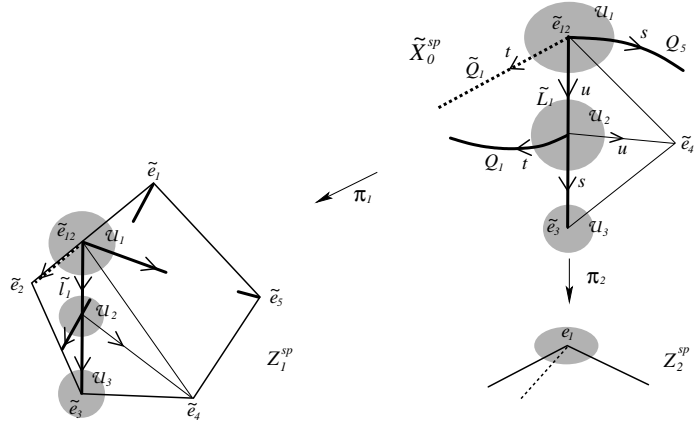}

\caption{\label{fig:The-local-affine}The local affine coordinates $\mathcal{U}_{1}$,
$\mathcal{U}_{2}$ and their coordinate axes along singular lines
with the projections to $\Sspp$ and $\Ssp$.}

\end{figure}

\begin{figure}
\includegraphics[scale=0.45]{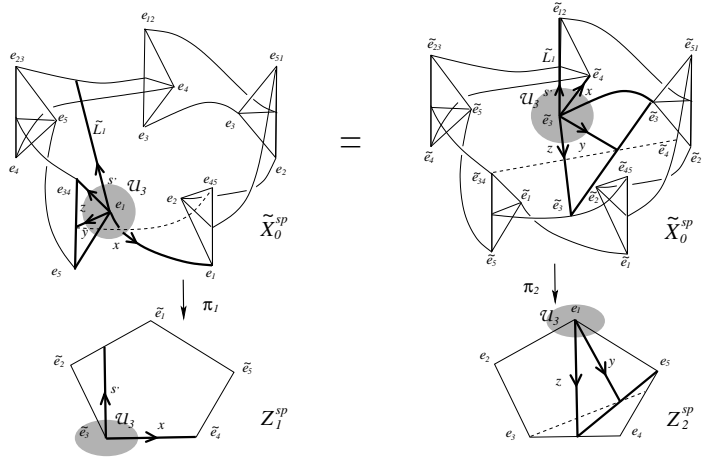}

\caption{\label{fig:pd-2A1}The local geometry $(\partial A_{1},\mathcal{U}_{3})$
with affine coordinates.}

\end{figure}

\bigskip{}

\subsubsection{Resolution of $(3A_{1},\mathcal{U}_{1})$}

We choose an affine coordinate $(s,t,u,v,\omega_{2,}\omega_{3,}\omega_{4,}\omega_{5})$
so that $s$, $t$ and $u$, respectively, coincide with the local
parameters of the curves $Q_{5,}$ $\tilde{Q}_{1}$ and $\tilde{L}_{1},$
and also the origin represents the point $[\tilde{e}_{12}]\times[e_{1}]=[\tilde{e}_{1}-\frac{1}{a}\tilde{e}_{2}]\times[e_{1}].$
For this, we use the parametrization given in (\ref{eq:curves-paramet}).
Explicitly, we write the points on $Q_{5}$ by \begin{align*}
\Bigl[a^{2}s\,\tilde{e}_{5}+(s+b)\tilde{e}_{12}\Bigr]\times\Bigl[se_{5}+e_{1}\Bigr] & =\Bigl[\tilde{e}_{1}-\frac{1}{a}\tilde{e}_{2}-\frac{as}{s+b}\tilde{e}_{5}\Bigr]\times\Bigr[e_{1}+se_{5}\Bigr]\\
 & =\Bigl[\tilde{e}_{1}-\frac{1}{a}\tilde{e}_{2}+s\tilde{e}_{5}\Bigr]\times\Bigl[e_{1}-\frac{bs}{s+a}e_{5}\Bigr],\end{align*}
where we have changed $-\frac{as}{s+b}$ in the middle to $s$ using
$Aut({\bf \mathbb{P}}^{1})$. Similarly, we can parametrize the points
on $\tilde{Q}_{1}$ and $\tilde{L}_{1}$, respectively, by \[
\Bigl[\tilde{e}_{1}-\frac{1}{a}\tilde{e}_{2}+t\tilde{e}_{2}\Bigr]\times\Bigl[e_{1}-\frac{a}{b}t\, e_{2}+\frac{a}{b^{2}}t\, e_{3}\Bigr]\;,\;\;\Bigl[\tilde{e}_{1}-\frac{1}{a}\tilde{e}_{2}+u\tilde{e}_{3}\Bigr]\times\Bigl[e_{1}\Bigr].\]
Introducing additional parameters $v,\omega_{2},\omega_{3},\omega_{4},\omega_{5}$,
we take an affine coordinate of $\mathbb{C}^{4}\times\mathbb{C}^{4}\subset\mathbb{P}^{4}\times\mathbb{P}^{4}$
by \[
\Bigl[\tilde{e}_{1}+\bigl(t-\frac{1}{a}\bigr)\tilde{e}_{2}+u\tilde{e}_{3}+v\tilde{e}_{4}+s\tilde{e}_{5}\Bigr]\times\Bigl[e_{1}+\bigl(\omega_{2}-\frac{a}{b}t\bigr)e_{2}+\bigl(\omega_{3}+\frac{a}{b^{2}}t\bigr)e_{3}+\omega_{4}e_{4}+\bigl(\omega_{5}-\frac{bs}{s+a}\bigr)e_{5}\Bigr].\]
 In order to see the local geometry about the origin, we work in the
local ring $\mathbb{C}[s,t,u,v,\omega_{2},..,\omega_{5}]_{m_{0}}$
with respect to the maximal ideal $m_{0}$ of the origin. Writing
the defining equations of $\Xsp$ in this ring, it is straightforward
to see that the three equations (1st, 2nd and 5th equations in (\ref{eq:defeqsCICY}))
may be solved as $\omega_{2}=\omega_{5}=0$ and $\omega_{3}=-\frac{a^{3}tu}{b^{2}(1-at)}$.
After substituting these into the remaining equation, we obtain\[
b^{3}(1-at)u\omega_{4}+at(u+av)(1-at-a^{2}u)=0,\;\omega_{4}(a+s)(as+v)-b^{2}sv=0.\]
Setting $\omega_{4}=w$, and focusing on the property near the origin,
we have:
\begin{prop}
\label{pro:g1-g2}The local geometry $(3A_{1},\mathcal{U}_{1})$ near
the singular point $[\tilde{e}_{12}]\times[e_{1}]$ is represented
by the germ $(\{g_{1},g_{2}\},\mathbb{C}^{5})$ near the origin $(s,t,u,v,w)=(0,...,0)$
with\[
g_{1}=a\, t(u+a\, v)+b^{3}uw,\;\; g_{2}=aw(v+a\, s)-b^{2}sv.\]
\end{prop}
\begin{rem*}
The coordinate $w=w_{4}$ has a special meaning related to the blow-up
$\pi_{1}:\Xsp\rightarrow\Sspp.$ In fact, in our affine coordinate
$(s,t,u,v,\omega_{2,}\omega_{3,}\omega_{4,}\omega_{5})$, the exceptional
divisor over the line $\tilde{q}_{1}$ can be written as \[
\Bigl[\tilde{e}_{1}-\frac{1}{a}\tilde{e}_{2}+t\tilde{e}_{2}\Bigr]\times\Bigl[e_{1}-\frac{a}{b}t\, e_{2}+\frac{a}{b^{2}}t\, e_{3}+\omega_{4}e_{4}\Bigr]\;\;(t,\omega_{4}\in\mathbb{C}).\]
Based on this, after eliminating the variable $w$ from the local
equations by $\{g_{1},g_{2}\},$ we have a germ $(g_{3,}\mathbb{C}^{4})$
near the origin with \begin{align*}
g_{3} & =a^{2}\, t(u+a\, v)(v+a\, s)-b^{5}suv\\
 & =a^{3}\, stu+a^{4}\, stv+b^{5}\, suv+a^{2}\, tuv+a^{3}\, tv^{2},\end{align*}
where the polynomial $g_{3}$ coincides with the lowest oder terms
of the defining (quintic) polynomial $\Sspp$ represented by the local
parameters $(1,z_{2},z_{3},z_{4,}z_{5})=(1,t-\frac{1}{a},u,v,s)$.
The $A_{1}$-singularity along $\tilde{Q}_{1}$ is the partial resolution
of the $A_{2}$-singularity along $\tilde{q}_{1}$ in $\Ssp$. {[}{]} 
\end{rem*}
Let us consider the blowing-up $\tilde{\mathbb{C}}^{5}\rightarrow\mathbb{C}^{5}$
at the origin of the local geometry $(\{g_{1},g_{2}\},\mathbb{C}^{5})$,
and denote the exceptional divisor by $E_{1}$. $E_{1}$ is the surface
$\{g_{1}=g_{2}=0$\} considered in $\mathbb{P}^{4}$ with the homogeneous
coordinate $[S:T:U:V:W]$ corresponding to $(s,t,u,v,w)$. 
\begin{prop}
\label{pro:e(E1)}$E_{1}$ is a singular del Pezzo surface of degree
four with three nodal points, and has the Euler number $e(E_{1})=5$. \end{prop}
\begin{proof}
The equations $g_{1}=g_{2}=0$ in $\mathbb{P}^{4}$ define a del Pezzo
surface of degree 4. By evaluating the Jacobian ideal, it is immediate
to see that this is singular at $[S:T:U:V:W]=[1:0:0:0:0],[0:1:0:0:0]$
and $[0:0:1:0:0]$, where the exceptional divisor $E_{1}$ intersects
with the $s$-, $t$- and $u$-axes of $A_{1}$-singularities. Since
$E_{1}$ is a singular Pezzo surface of degree 4 with three ordinary
double points, it can be given as $\mathbb{P}^{2}$ blown-up at 5
points and then contracting three $(-2)$ curves \cite{HW}. Therefore
we have $e(E_{1})=8-3=5$.\end{proof}
\begin{prop}
\label{pro:E1-3A1}After the blowing-up $(3A_{1},\mathcal{U}_{1})$
at the origin, the three singular lines separate from each other and
intersect with $E_{1}$ at the three nodal points.\end{prop}
\begin{proof}
We have chosen our coordinate of $\mathcal{U}_{1}$ so that $s$-,
$t$-, $u$-axes coincide with the lines of $A_{1}$-singularity.
The blowing-up $\tilde{\mathbb{C}}^{5}\rightarrow\mathbb{C}^{5}$
at the origin introduces the exceptional set $\mathbb{P}^{4}$, which
separate the coordinate axes. Hence the blowing-up separates the $s$-,
$t$-, $u$-axes of $A_{1}$-singularity from each other with introducing
the exceptional divisor $E_{1}$. The intersection points of the (proper
transforms of the) $s$-, $t$-, $u$-axes with $E_{1}$ coincides
with the three nodal points of $E_{1}$. (This is similar to the case
of the Barth-Nieto quintic \cite{Barth-Nieto}.)
\end{proof}
Now, we blow-up all the 10 local geometries of type $(3A_{1},\mathcal{U}_{1})$
at their origins, and denote the blow-ups by $\varphi_{1}:\Xspi\rightarrow\Xsp$.
Also we denote by $Q_{i}^{(1)},\tilde{Q}_{i}^{(1)},L_{i}^{(1)},\tilde{L}_{i}^{(1)}$
the proper transforms of the 20 lines $Q_{i},\tilde{Q}_{i},L_{i},\tilde{L}_{i}$
of $A_{1}$ singularity, respectively. Along these proper transforms
of lines, we still have $A_{1}$ singularities. Also, these lines
intersect at the origins of the 10 isomorphic local geometries of
type $(2A_{1},\mathcal{U}_{2}^{(1)})$, which is isomorphic to $(2A_{1,}\mathcal{U}_{1})$
in $\Xsp$. We also denote by $(\partial A_{1},\mathcal{U}_{3}^{(1)})\cong(\partial A_{1},\mathcal{U}_{3})$
the 10 isomorphic local geometries near the infinity points of the
lines $L_{i}^{(1)},\tilde{L}_{i}^{(1)}$(see Fig.$\,$\ref{fig:The-local-affine}).
The local geometries on each lines are now summarized as\begin{equation}
\begin{aligned}(\partial A_{1},\mathcal{U}_{3}^{(1)}),(2A_{1},\mathcal{U}_{2}^{(2)}) & \text{ on each }\tilde{L}_{i}^{(1)},L_{i}^{(1)},\\
(2A_{1},\mathcal{U}_{1}^{(1)}) & \text{ on each }\tilde{Q}_{i}^{(1)},Q_{i}^{(1)}.\end{aligned}
\label{eq:singularities-on-LQ-(1)}\end{equation}

\bigskip{}

\subsubsection{Resolution of $(2A_{1},\mathcal{U}_{2}^{(1)})$ }

As in the previous case, we choose an affine coordinate $(s,t,u,v,\omega_{2},\omega_{3},\omega_{4},\omega_{5})$
centered at $[-a\,\tilde{e}_{12}+\tilde{e}_{3}]\times[e_{1}]$ with
$s,t$ being along the lines $\tilde{L}_{1}^{(1)}$, $Q_{1}^{(1)}$
in $\Xspi$. For this we parametrize the line $\tilde{L}_{1}^{(1)}$
by \[
\bigl[-a\tilde{e}_{12}+(1+s)\tilde{e}_{3}\bigr]\times[e_{1}]=\bigl[a^{2}\tilde{e}_{1}-a\tilde{e}_{2}+(1+s)\tilde{e}_{3}\bigr]\times[e_{1}],\]
and also the line $Q_{1}^{(1)}$ by \[
\bigl[a^{2}\tilde{e}_{1}+(1+t)\tilde{e}_{23}\bigr]\times\Bigl[e_{1}+\frac{1}{b}te_{2}\Bigr]=\bigl[a^{2}\tilde{e}_{1}-a(1+t)\tilde{e}_{2}+(1+t)\tilde{e}_{3}\bigr]\times\Bigl[e_{1}+\frac{t}{b}e_{2}\Bigr].\]
 Taking these forms into account, we introduce the affine coordinate
by\begin{align*}
\bigl[a^{2}\tilde{e}_{1}-a(1+t)\tilde{e}_{2}+(1+s+t)\tilde{e}_{3}+u\tilde{e}_{4}+v\tilde{e}_{5}\bigr]\\
\times\Bigl[e_{1}+\bigl(\omega_{2}+\frac{t}{b}\bigr)e_{2}+\omega_{3}e_{3}+\omega_{4}e_{4}+\omega_{5}e_{5}\Bigr].\end{align*}
In the local ring $\mathbb{C}[s,t,u,v,\omega_{2},..,\omega_{5}]_{m_{0}}$,
four of the five defining equations of $\Xspi$ may be solved with
respect to $\omega_{2},\omega_{3},\omega_{4},\omega_{5}$ and one
equation leftover determines the germ about the origin. 
\begin{prop}
The local geometry $(2A_{1},\mathcal{U}_{2}^{(1)})$ near the singular
point $[-a\,\tilde{e}_{12}+\tilde{e}_{3}]\times[e_{1}]$ is represented
by a germ $(h,\mathbb{C}^{4})$ near the origin $(s,t,u,v)=(0,0,0,0)$
with\begin{equation}
h=b^{5}uv+a^{3}\, stu+a^{4}\, stv+b^{5}suv+2b^{5}tuv.\label{eq:h-def}\end{equation}

\end{prop}
We may derive the same form directly from the quintic equation of
$\Sspp$ since the projection $\pi_{1}:\,\Xsp\rightarrow\Sspp$ (composed
with the blow-up $\Xspi\rightarrow\Xsp)$ defines an isomorphism on
the neighborhood $\mathcal{U}_{2}$, see Fig. \ref{fig:The-local-affine}.
In the figure, the geometric meaning of the parameters $u$ and $v$
should be clear. By our choice of the coordinates, we have $A_{1}$-singularities
along $s$- and $t$-axes, i.e., along the lines $\tilde{L}_{1}^{(1)}$
and $Q_{1}^{(1)}$, respectively. We will consider the blowing-up
along $\tilde{L}_{1}^{(1)}$, which is locally described by the blowing-up
$\mathbb{C}\times\tilde{\mathbb{C}}^{3}\rightarrow\mathbb{C}\times\mathbb{C}^{3}$
along the $s$-axis. 
\begin{prop}
\label{pro:blow-up-along-s-2A1}1) The exceptional divisor of the
blow-up $\mathbb{C}\times\tilde{\mathbb{C}}^{3}\rightarrow\mathbb{C}\times\mathbb{C}^{3}$
of $(2\, A_{1},\mathcal{U}_{2}^{(1)})$ along the $s$-axis is a conic
bundle over $s\in\mathbb{C}$ $(|s|\ll1)$, which has a reducible
fiber over $s=0$. This conic bundle is singular only at an ODP over
$s=0$.

\noindent 2) The conic bundle over $\mathbb{C}$ extends to a conic
bundle $E_{2}\rightarrow\tilde{L}_{1}^{(1)}\cong\mathbb{P}^{1}$,
which has reducible fibers over $s=0$ and $s=\infty$. This conic
bundle is singular only at an ODP over $s=0,$ and also admits a section.

\noindent3) After the blowing-up of $(2A_{1},\mathcal{U}_{2}^{(1)})$,
the singularity leftover near the local geometry is the $A_{1}$ singularity
along the proper transform of the $t$-axis. The proper transform
of $Q_{1}^{(1)}$ intersects with the conic bundle $E_{2}$ at the
ODP over $s=0$. \end{prop}
\begin{proof}
1) We introduce the coordinate $(s,[T,:U:V])$ for the exceptional
set $\mathbb{C}\times\mathbb{P}^{2}$ of the blow-up $\mathbb{C}\times\tilde{\mathbb{C}}^{3}\rightarrow\mathbb{C}\times\mathbb{C}^{3}$.
Then from the local equation of $(2A_{1},\mathcal{U}_{2})$, we have
the equation of the exceptional divisor as \[
b^{5}UV+a^{3}sTU+a^{4}sTV+b^{5}sUV=0.\]
This defines a family of conics in $\mathbb{P}^{2}$ over $s\in\mathbb{C}\,(|s|\ll1)$,
which is reducible at $s=0$. Also we see that the conic bundle is
singular only at an ODP over $s=0.$

2) To see the geometry of the exceptional divisor over $\mathbb{C}$($=\mathbb{P}^{1}\setminus\{s=\infty\}$),
we need to have the equation (\ref{eq:h-def}) in all order in $s$
but with homogeneous of degree two for $t,u,v$. It is easy to have
the equation from the defining equation of $\Sspp$. After some algebra,
we have the equation for the exceptional divisor:\begin{equation}
(s+1)\left(b^{5}UV+a^{3}sTU+a^{4}sTV\right)=0,\label{eq:E2-along-sTUV}\end{equation}
which defines a conic bundle over $s\in\mathbb{C}\,(s\not=-1)$ with
only one singular fiber over $s=0$. We see that $s=-1$ correspond
to the intersection point of the exceptional divisor $E_{1}$ and
the s-axis. Since this intersection point is one of the three nodal
points on $E_{1}$, we see that the conic bundle extends to $s=-1$
with smooth fiber over it. 

The point $s=\infty$ in the s-axis corresponds to the center of the
local geometry $(\partial A_{1},\mathcal{U}_{3}^{(1)})$. We introduce
the local parameters $s'=\frac{1}{s},x,y,z$ to represent the relevant
lines in this geometry, see Fig.$\,$\ref{fig:pd-2A1}. With other
parameters $v_{1},v_{5}$ and $\omega_{2},\omega_{4}$, we consider
the following affine coordinate centered at $[\tilde{e}_{3}]\times[e_{1}]$
of $\mathbb{P}^{4}\times\mathbb{P}^{4}:$ \[
\begin{alignedat}{1} & [(v_{1}-as')\tilde{e}_{1}+s'\tilde{e}_{2}+\tilde{e}_{3}+x\tilde{e}_{4}+v_{5}\tilde{e}_{5}]\times\\
 & [e_{1}+\omega_{2}e_{2}+(b^{2}y-z)e_{3}+(\omega_{4}-by+z)e_{4}+ye_{5}].\end{alignedat}
\]
Writing the defining equations (\ref{eq:defeqsCICY}) in this coordinate,
we can solve four equations with respect to $v_{1},v_{5},\omega_{2},\omega_{5}$
to obtain one equation $abxz+s'(bxz+a^{4}yz-a^{4}by^{2})=0$ which
describes the local geometry $(\partial A_{1},\mathcal{U}_{3}^{(1)})$
near the origin. Now we have the following local equation of the exceptional
divisor of the blowing up along $s'$-axis: \[
abXZ+s'\left(bXZ+a^{4}YZ-a^{4}bY^{2}\right)=0\;\;(|s'|\ll1),\]
where $(s',[X,Y,Z])$ represents the coordinates of the exceptional
set $\mathbb{C}\times\mathbb{P}^{2}$ of the blow-up $\mathbb{C}\times\tilde{\mathbb{C}}^{3}\rightarrow\mathbb{C}\times\mathbb{C}^{3}$.
From this equation, we see that the exceptional divisor is a conic
bundle with reducible fiber over $s'=0\,(s=\infty$) but smooth for
$|s'|\ll1$.

Finally, from the equation (\ref{eq:E2-along-sTUV}), we see that
$U=V=0$, for example, gives a section. 

3) Let $(s,t,\tilde{u},\tilde{v})=(s,t,\frac{U}{T},\frac{V}{T})$
be the one of the affine coordinates of the blow-up. Then we have
$u=\tilde{u}t,v=\tilde{v}t.$ Substituting these into the local equation
$h$ of $(2A_{1},\mathcal{U}_{2}),$ i.e., for $|s|,|t|\ll1$, we
obtain \[
\tilde{h}=b^{5}\tilde{u}\tilde{v}+a^{3}s\tilde{u}+a^{4}s\tilde{v}+b^{5}s\tilde{u}\tilde{v}+2b^{5}t\tilde{u}\tilde{v}\]
with $h=t^{2}\tilde{h}$. If we set $t=0,$ then we have the equation
of the exceptional divisor ($|s|\ll1$) above. When we set $s=0,$
then we have $\tilde{h}=\tilde{u}\tilde{v}(b^{5}+2b^{5}t)$. This
shows that the ODP of the exceptional divisor $E_{2}$ over $s=0$
merges to the $A_{1}$-singularity along the proper transform of the
$t$-axis ( see Fig. \ref{fig:E1-E2-E3}). Since the singularity along
the line $Q_{1}^{(1)}$ is of $A_{1}$-type except $t=0$, i.e., at
the intersection $Q_{1}^{(1)}\cap\tilde{L}_{1}^{(1)}$, we now see
that, near $t=0$, the singularity along the proper transform of $Q_{1}^{(1)}$
is of $A_{1}$-type.
\end{proof}
All the intersections of $Q_{i}^{(1)}$ and $\tilde{L}_{i}^{(1)}$
($\tilde{Q}_{i}^{(1)}$ and $L_{i}^{(1)}$) have the local geometries
isomorphic to $(2\, A_{1},\mathcal{U}_{2}^{(1)})$. We blow-up along
all the 10 lines $\tilde{L}_{i}^{(1)}$ and $L_{i}^{(1)}$, and denote
the blow-ups by $\varphi_{2}:\,\Xspii\rightarrow\Xspi$. We denote
the proper transforms of the 10 lines $Q_{i}^{(1)}$ and $\tilde{Q}_{i}^{(1)}$,
respectively, by $Q_{i}^{(2)}$ and $\tilde{Q}_{i}^{(2)}$. 

\bigskip{}

\subsubsection{Crepant Resolution $\cXsp\rightarrow\Xsp$ }

We finally construct a crepant resolution.
\begin{prop}
1) All the singularities of $\Xspii$ are along the non-intersecting
10 lines $Q_{i}^{(2)}$ and $\tilde{Q}_{i}^{(2)}$. 

\noindent 2) The singularities along $Q_{i}^{(2)}$ and $\tilde{Q}_{i}^{(2)}$
are of $A_{1}$-type. Blowing-up along each line of $Q_{i}^{(2)}$
and $\tilde{Q}_{1}^{(2)}$ resolves the singularity with introducing
the exceptional divisor $E_{3}$ which is a $\mathbb{P}^{1}$-bundle
with a section. \end{prop}
\begin{proof}
1) Since each line of $Q_{i}$ and $\tilde{Q}_{i}$ intersects with
others at the center of the local geometry $(3A_{1},\mathcal{U}_{1})$,
it is clear that $Q_{i}^{(2)}$ and $\tilde{Q}_{i}^{(2)}$ are separated
after the blow-ups (see Proposition \ref{pro:E1-3A1}). 

2) By symmetry, it suffices to show the properties for a line, say
$\tilde{Q}_{1}^{(2)}$. Note that $\tilde{Q}_{1}^{(2)}$ in $\Xspii$
is given by the successive proper transform of the line $\tilde{Q}_{1}$
in $\Xsp$ under the blow-ups of the local geometries, two $(3A_{1},\mathcal{U}_{1})$'s
and one $(2A_{1},\mathcal{U}_{2})$ on the line. Therefore, the local
geometry around $\tilde{Q}_{1}^{(2)}$ is isomorphic to that around
the line $\tilde{Q}_{1}$ except the three centers of the blowing-ups
on the line. We further note that the local geometry around the line
$\tilde{Q}_{1}$ except the three centers is projected isomorphically
to $Z_{2}^{sp}$ under the partial resolution $\pi_{2}:\Xsp\rightarrow Z_{2}^{sp}$.
The local geometry around $\pi_{2}(\tilde{Q}_{1})$ is easily analyzed
by introducing the following affine coordinate:\[
[w_{1}:w_{2}:...:w_{5}]=[e_{1}+(u-bt)e_{2}+te_{3}+ve_{4}+we_{5}],\]
where $t$ parametrizes the line $\pi_{2}(\tilde{Q}_{1})$. Substituting
this into the defining equation (\ref{eq:defEqS}) of $Z_{2}^{sp}$and
taking the polynomial of homogeneous degree up to two with respect
to $u,v,w$ but all for $t$, we obtain\begin{equation}
bt\bigl\{ a^{5}tvw-(1-b^{2}t)u(v+bw)\bigr\}=0,\label{eq:localeq-Q1}\end{equation}
which shows $A_{1}$-singularity along the $t$-axis except $t=0,\frac{1}{b^{2}}$
and $\infty$. These three values exactly correspond to the two local
geometries $(3A_{1},\mathcal{U}_{1})$'s and one $(2A_{1},\mathcal{U}_{2})$
on the line $\tilde{Q}_{1}$, whose blowing-up we studied in Proposition
\ref{pro:E1-3A1} and Proposition \ref{pro:blow-up-along-s-2A1}.
Combined with the results there, we conclude that the singularity
along $\tilde{Q}_{1}^{(2)}$ is of $A_{1}$-type, and it is resolved
by the blowing-up along the line with introducing an exceptional divisor
$E_{3}$ which is isomorphic to a $\mathbb{P}^{1}$-bundle over the
line. Also from the equation (\ref{eq:localeq-Q1}), it is easy to
see that $E_{3}$ has a section (cf. Proposition \ref{pro:blow-up-along-s-2A1}
2) ). 
\end{proof}
Let us now denote the blowing-up along the 10 lines by $\varphi_{3}:\Xspiii\rightarrow\Xspii$.
Defining $\cXsp:=\Xspiii$, we may summarize the whole process of
the blowing-ups by\[
\varphi:\;\cXsp=\Xspiii\xrightarrow[\varphi_{3}]{}\Xspii\xrightarrow[\varphi_{2}]{}\Xspi\xrightarrow[\varphi_{1}]{}\Xsp,\]
 where $\varphi:\cXsp\rightarrow\Xsp$ represents the composition. 
\begin{thm}
\label{thm:Thm-crepant-e100}For $(a,b)\in(\mathbb{C}^{*})^{2}$ with
non-vanishing discriminant (\ref{eq:discXandU}), the blowing-up $\varphi:\cXsp\rightarrow\Xsp$
is a crepant resolution and gives a smooth Calabi-Yau manifold with
the Euler number $e(\cXsp)=2(h^{1,1}(\cXsp)-h^{2,1}(\cXsp))=100.$ \end{thm}
\begin{proof}
For the proof of $K_{\tilde{X}_{0}^{*}}\cong\mathcal{O}_{\tilde{X}_{0}^{*}}$,
we show the existence of a nowhere vanishing holomorphic 3-form explicitly,
although an abstract argument is possible. We first consider the blow-up,
$\varphi_{1}:\Xspi\rightarrow\Xsp$ at the origin of the local geometries
$(3A_{1},\mathcal{U}_{1})$. As before we introduce the affine coordinate
$(s,t,u,v,w_{2},\cdots,w_{5})$. We start with the standard form of
a nowhere vanishing holomorphic 3-form $\Omega(\Xsp)$ for the complete
intersection Calabi-Yau variety $\Xsp$ given in (\ref{eq:OmegaXsp}).
In this affine coordinate, we have\[
\Omega(\Xsp)|_{\mathcal{U}_{1}}=Res_{f_{1}=\cdots=f_{5}=0}\Bigl(\frac{ds\wedge dt\wedge du\wedge dv\wedge dw_{2}\wedge\dots\wedge dw_{5}}{f_{1}f_{2}f_{3}f_{4}f_{5}}\Bigr).\]
Evaluating the Jacobian $\frac{\partial(w_{2},w_{3},w_{5})}{\partial(f_{1},f_{2},f_{5})}=\frac{-a}{b^{2}(a+s)(1-at)}$,
we calculate the residue as\begin{equation}
\Omega(\Xsp)|_{\mathcal{U}_{1}}=\frac{-1}{a}Res_{g_{1}=g_{2}=0}\Bigl(\frac{ds\wedge dt\wedge du\wedge dv\wedge dw}{g_{1}g_{2}}\Bigr),\label{eq:Omega-g1-g2}\end{equation}
where $w=w_{4}$ and $g_{1,}g_{2}$ are given in Proposition \ref{pro:g1-g2}
(precisely $g_{1,}g_{2}$ here contain all higher order terms, but
this does not affect the following arguments). Consider the blow-up
$\varphi_{1}:\tilde{\mathbb{C}}^{5}\rightarrow\mathbb{C}^{5}$ at
the origin, and one of the affine coordinate $(s,\tilde{t},\tilde{u},\tilde{v,}\tilde{w})=(s,\frac{T}{S},\frac{U}{S},\frac{V}{S},\frac{W}{S}$)
with $t=\tilde{t}s,u=\tilde{u}s,v=\tilde{v}s,w=\tilde{w}s$. Then,
pulling back the 3-form, it is immediate to have\begin{equation}
\varphi_{1}^{*}\Omega(\Xsp)\Bigl|_{\mathcal{U}_{1}^{(1)}}=\frac{-1}{a}Res_{\tilde{g}_{1}=\tilde{g}_{2}=0}\Bigl(\frac{ds\wedge d\tilde{t}\wedge d\tilde{u}\wedge d\tilde{v}\wedge d\tilde{w}}{\tilde{g}_{1}\tilde{g}_{2}}\Bigr),\label{eq:3-form-local-(1)}\end{equation}
where $g_{1}=s^{2}\tilde{g}_{1},g_{2}=s^{2}\tilde{g}_{2}$ and $\tilde{g}_{1}=\tilde{g}_{2}=0$
is the defining equation of the blow-up. Up to the non-vanishing constant,
the right hand side is the holomorphic 3-form $\Omega(\Xspi)$ of
$\Xspi$. Calculations are similar for other affine coordinates, and
we see that the pull-back $\varphi_{1}^{*}\Omega(\Xsp)$ coincides
with $\Omega(\Xspi),$ i.e., $\varphi_{1}$ is crepant. The next step
$\varphi_{2}:\Xspii\rightarrow\Xspi$ has an effect on (\ref{eq:3-form-local-(1)})
as the blowing-up along the $s$-axis. Again, it is straightforward
to see that $\Omega(\Xspii)\bigl|_{\mathcal{U}_{1}^{(2)}}=\varphi_{2}^{*}\Omega(\Xspi)\bigl|_{\mathcal{U}_{1}^{(2)}}$
holds up to a non-vanishing constant on all the affine coordinates.
Doing similar calculations for the blow-up $\varphi_{3}$, we finally
verify that $\Omega(\Xspiii)\bigl|_{\mathcal{U}_{1}^{(3)}}=\varphi_{3}^{*}\Omega(\Xspii)\bigl|_{\mathcal{U}_{1}^{(3)}}$.
Thus near the 10 points of the local geometry $(3A_{1,}\mathcal{U}_{1})$,
we see that $\varphi:\cXsp\rightarrow\Xsp$ is crepant. 

For the local geometry $(2A_{1},\mathcal{U}_{2})$, since the first
blow-up $\varphi_{1}$ has no effect, we start with $\Omega(\Xspi)\bigl|_{\mathcal{U}_{2}}=\Omega(\Xsp)\bigr|_{\mathcal{U}_{2}}$.
As in the previous subsection, we introduce the affine coordinate
$(s,t,u,v,w_{2},w_{3},w_{4})$. Evaluating the Jacobian $\frac{\partial(w_{2},w_{3},w_{4},w_{5})}{\partial(f_{1},f_{2},f_{3},f_{5})}$,
we have\[
\Omega(\Xsp)\bigr|_{\mathcal{U}_{2}}=\frac{-1}{a}Res_{h=0}\Bigl(\frac{ds\wedge dt\wedge du\wedge dv}{h}\Bigr),\]
where $h$ is given in (\ref{eq:h-def}) (again, precisely $h$ should
be understood with the higher order terms). Then $\varphi_{2}$ is
the blow-up along the $s$-axis, see Proposition \ref{pro:blow-up-along-s-2A1}.
Using one of the affine coordinate of the blow-up, $(s,t,\tilde{u},\tilde{v})=(s,t,\frac{U}{T},\frac{V}{T})$
with $u=\tilde{u}t,v=\tilde{v}t$, we evaluate the pull-back as \[
\varphi_{2}^{*}\Omega(\Xspi)\bigl|_{\mathcal{U}_{2}^{(2)}}=\frac{-1}{a}Res_{\tilde{h}=0}\Bigl(\frac{ds\wedge dt\wedge d\tilde{u}\wedge d\tilde{v}}{\tilde{h}}\Bigr),\]
with $h=t^{2}\tilde{h}$. Since $\tilde{h}=0$ is the local equation
of the blow-up $\Xspii$, we see that $\Omega(\Xspii)\bigl|_{\mathcal{U}_{2}^{(2)}}=\varphi_{2}^{*}\Omega(\Xspi)\bigl|_{\mathcal{U}_{2}^{(2)}}$
up to a non-vanishing constant. The next blow-up $\varphi_{3}$ is
along the $t$-axis, and this is done locally by $(s,t,\tilde{u}',\tilde{v}')=(s,t,\frac{\tilde{U}}{\tilde{S}},\frac{\tilde{V}}{\tilde{S}})$
with $\tilde{u}=\tilde{u}'s,\tilde{v}=\tilde{v}'s.$ The local equation
of the blow-up is given by $\tilde{h}'=0$ with $\tilde{h}=s^{2}\tilde{h}'$,
and we have $\Omega(\Xspiii)\bigl|_{\mathcal{U}_{2}^{(3)}}=\varphi_{3}^{*}\Omega(\Xspii)\bigl|_{\mathcal{U}_{2}^{(3)}}$,
up to a non-vanishing constant. From the local equation $\tilde{h}'=0$,
we see that $\Xspiii=\cXsp$ is smooth. The calculations are valid
for all the 10 points of the local geometry $(2A_{1},\mathcal{U}_{2})$. 

Combined with the results for $(3A_{1,}\mathcal{U}_{1})$, we conclude
that $\varphi:\cXsp\rightarrow\Xsp$ is a crepant resolution. 

Next we show that $\tilde{X}_{0}^{*}$ is a Calabi-Yau manifold, namely,
i) $K_{\tilde{X}_{0}^{*}}\cong\mathcal{O}_{\tilde{X}_{0}^{*}}$ and
ii) $h^{1}(\mathcal{O}_{\tilde{X}_{0}^{*}})=h^{2}(\mathcal{O}_{\tilde{X}_{0}^{*}})=0$.
For the property i), we note that $K_{\Xsp}\cong\mathcal{O}_{\Xsp}$
since $\Xsp$ is a complete intersection of $5$ divisors of $(1,1)$-type
in $\mathbb{P}^{4}\times\mathbb{P}^{4}$. Then $K_{\tilde{X}_{0}^{*}}=\varphi^{*}K_{\Xsp}\cong\mathcal{O}_{\tilde{X}_{0}^{*}}$
is immediate since $\varphi$ is crepant. For the second ii), we note
that all the higher direct images $R^{i}\varphi_{*}\mathcal{O}_{\tilde{X}_{0}^{*}}\;(i>0)$
vanish by the Grauert-Riemenschneider vanishing since $\varphi$ is
crepant. Then, by the Leray spectral sequence, we have $H^{i}(\mathcal{O}_{\tilde{X}_{0}^{*}})\cong H^{i}(\mathcal{O}_{\Xsp})$$\,(i=1,2)$.
Hence we have only to show that the r.h.s vanishes. Note that $\Xsp$
is a complete intersection of $5$ divisors of $(1,1)$-type in $\mathbb{P}^{4}\times\mathbb{P}^{4}$,
and consider the following Koszul resolution of $\mathcal{O}_{\Xsp}$
:\begin{align*}
0\to\mathcal{O}_{\mathbb{P}^{4}\times\mathbb{P}^{4}}(-5,-5)\to\mathcal{O}_{\mathbb{P}^{4}\times\mathbb{P}^{4}}(-4,-4)^{\oplus4}\to\mathcal{O}_{\mathbb{P}^{4}\times\mathbb{P}^{4}}(-3,-3)^{\oplus10}\to\\
\mathcal{O}_{\mathbb{P}^{4}\times\mathbb{P}^{4}}(-2,-2)^{\oplus10}\to\mathcal{O}_{\mathbb{P}^{4}\times\mathbb{P}^{4}}(-1,-1)^{\oplus5}\to\mathcal{O}_{\mathbb{P}^{4}\times\mathbb{P}^{4}}\to\mathcal{O}_{\Xsp}\to0.\end{align*}
As for the sheaves in this exact sequence except $\mathcal{O}_{\Xsp}$,
all the cohomology groups vanish except $H^{5}(\mathcal{O}_{\mathbb{P}^{4}\times\mathbb{P}^{4}}(-5,-5))\cong\mathbb{C}$
and $H^{0}(\mathcal{O}_{\mathbb{P}^{4}\times\mathbb{P}^{4}})\cong\mathbb{C}$
by the Kodaira vanishing theorem and the Serre duality. Now it is
standard to see that $H^{2}(\mathcal{O}_{\Xsp})\cong H^{5}(\mathcal{O}_{\mathbb{P}^{4}\times\mathbb{P}^{4}}(-5,-5))\cong\mathbb{C}$,
$H^{0}(\mathcal{O}_{\Xsp})\cong H^{0}(\mathcal{O}_{\mathbb{P}^{4}\times\mathbb{P}^{4}})\cong\mathbb{C}$
and $H^{i}(\mathcal{O}_{\Xsp})\,(i=1,2)$ vanish.%
\begin{figure}
\includegraphics[scale=0.6]{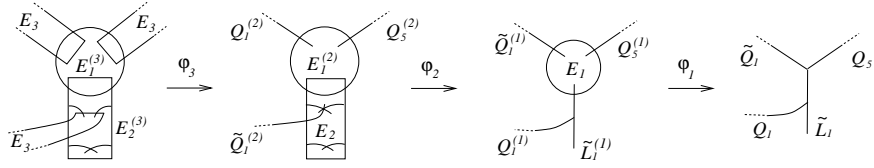}

\caption{\label{fig:E1-E2-E3}Exceptional divisors $E_{1},E_{2},E_{3}$ of
the blowing-ups $\varphi_{1},\varphi_{2},\varphi_{3}$, respectively.
Only the local geometries around the line $\tilde{L}_{1}$ in $\Xsp$
are depicted.}

\end{figure}

For the calculation of the Euler number, let us first note that we
have $e(\Xsp)=e(\Ssp)+5\times\bigl(e(\mathbb{P}^{2})-1\bigr)=-10+10=0$.
This follows form Proposition \ref{pro:EulerS} and Proposition \ref{pro:resol-by-pi2},
see also Fig. \ref{fig:FigResolution-S}. Now we note that, under
the blow-up, the origin of $(3A_{1,}\mathcal{U}_{1})$ is replaced
by the exceptional divisor $E_{1}$ with its Euler number $e(E_{1})=5.$
Similarly for $(2\, A_{1},\mathcal{U}_{2}^{(1)})$, one line is replaced
by a conic bundle $E_{2}$ over $\mathbb{P}^{1}$ with two reducible
fibers, hence $e(E_{2})=6$. Since we have 10 isomorphic geometries
for $(3A_{1,}\mathcal{U}_{1})$ and 10 for $(2\, A_{1},\mathcal{U}_{2}^{(1)})$,
taking into account the final blow-ups of 10 lines, we evaluate the
Euler number $e(\cXsp$) as \begin{align*}
e(\cXsp) & =10\times\bigl(e(E_{1})-1\bigr)+10\times\bigl(e(E_{2})-e(\mathbb{P}^{1})\bigr)+10\times\bigl(e(E_{3})-e(\mathbb{P}^{1})\bigr)\\
 & =40+10\times(6-2)+10\times2=100.\end{align*}

\end{proof}

\subsection{Hodge numbers}

Recall that the crepant resolution is obtained as the composite of
the blowing-ups $\varphi_{1}:\Xspi\rightarrow\Xsp$, $\varphi_{2}:\Xspii\rightarrow\Xspi$,
$\varphi_{3}:\Xspiii\rightarrow\Xspii$. The first blow-up $\varphi_{1}$
introduces the exceptional divisors $E_{1}(=:E_{1}^{(1)})$ in $\Xspi$
which is a del Pezzo surfaces of degree 4 with three lines are contracted
to three points. One of the three points is resolved in the proper
transform $E_{1}^{(2)}$ under $\varphi_{2}$, and the other two are
resolved in the proper transform $E_{1}^{(3)}$ under $\varphi_{3}$.
Similarly, the resolution $\varphi_{2}$ introduces the conic bundle
$E_{2}(=:E_{2}^{(2)})$ over $\mathbb{P}^{1}$ which has an ordinary
double point (over $s=0$), and $\varphi_{3}$ resolves this singularity
to have smooth ruled surface $E_{2}^{(3)}$ in $\Xspiii$. The final
blow-up $\varphi_{3}$ introduces the divisor $E_{3}=E_{3}^{(3)}$
which is a $\mathbb{P}^{1}$-bundle over $\mathbb{P}^{1}$ with a
section. Note that all these divisors $E_{1}^{(3)},E_{2}^{(3)}$ and
$E_{3}^{(3)}$ are smooth in $\Xspiii=\cXsp$. 

In this subsection, following \cite{HulekEtAl}, we apply the Weil
conjecture to determine the Hodge numbers of the resolution $\cXsp.$
We set our parameters to $a=b=1$ and consider the mod $p$ reduction
of $\cXsp.$ We write $\mathbb{F}_{p}=\mathbb{Z}/p\mathbb{Z}$. 
\begin{lem}
For all but finite primes, the reduction of $\cXsp$ modulo $p$ is
smooth over $\mathbb{F}_{p}$. \end{lem}
\begin{proof}
The smoothness of $\cXsp$ in the tori $(\mathbb{C}^{*})^{4}\times(\mathbb{C}^{*})^{4}$
follows from the discriminant (in Proposition \ref{pro:dis-X-S-H})
$dis(\Xsp|_{(\mathbb{C}^{*})^{4}})=3\times11^{3}$ for $a=b=1.$ The
exceptional divisors $E_{1},E_{2}$ of the blowing-ups $\varphi_{1}$
and $\varphi_{2}$, respectively, are blown-up to smooth surfaces
$E_{1}^{(3)}$ and $E_{2}^{(3)}$ in $\cXsp$, hence the resolution
$\cXsp$ is smooth over $\mathbb{F}_{p}$ except finite primes $p$. 
\end{proof}
Let $\cXsp(\mathbb{F}_{p})$ be the set of points in $\cXsp$ which
are rational over $\mathbb{F}_{p}.$ We use the Lefschetz fixed point
formula due to Grothendieck, \begin{equation}
\#\cXsp(\mathbb{F}_{p})=1-t_{1}+t_{2}-t_{3}+t_{4}-t_{5}+t_{6},\label{eq:Np-fixed-pt-formula}\end{equation}
with $t_{j}=\text{tr (\text{Frob}}_{p}^{*}\bigl|H_{\acute{e}t}^{j}(\cXsp,\mathbb{Q}_{\ell}))$
and $\text{Frob}_{p}:\cXsp\rightarrow\cXsp$ the Frobenius morphism.
Since $\cXsp$ is a Calabi-Yau threefold, we have $t_{0}=1,t_{1}=t_{5}=0,t_{6}=p^{3}$.
By the Weil conjecture (see \cite[Appendix C]{Har} for example),
the eigenvalues of $\text{Frob}_{p}$ on $H_{\acute{e}t}^{j}(\cXsp,\mathbb{Q}_{\ell})$
are algebraic integers, which do not dependent on $\ell,$ with absolute
values $p^{j/2}$. Also, by the Weil conjecture again, $t_{j}$'s
are (ordinary) integers and satisfy $|t_{j}|\leq b_{j}(\cXsp)\; p^{j/2}$.
We derive the following property following the arguments in \cite[Prop. 2.4]{HulekEtAl}
made for the Barth-Nieto quintic.
\begin{prop}
\label{pro:mod-p-H2}For every good prime $p$, all eigenvalues of
$\text{Frob}_{p}$ on $H_{\acute{e}t}^{2}(\cXsp,\mathbb{Q}_{\ell})$
are equal to $p$. \end{prop}
\begin{proof}
Due to Lemma \ref{lem:lefschets-Fp} below, we can use the Lefschetz
hyperplane theorem \cite[Corollary I.9.4]{FK} and have the claimed
property for $\Xsp$. Then from the Leray spectral sequence associated
to $\varphi_{1}:\Xspi\rightarrow\Xsp,$ we obtain the claimed property
for $\Xspi$ (see \cite[Lemma 2.16]{HulekEtAl}). To go further to
$\Xspii$, we use the Leray spectral sequence associated to $\varphi_{2}:\Xspii\rightarrow\Xspi$,\[
E_{2}^{j,2-j}=H_{\acute{e}t}^{j}(\Xspi,R^{2-j}{\varphi_{2}}_{*}(\mathbb{Q}_{\ell}))\Rightarrow H_{\acute{e}t}^{2}(\Xspii,\mathbb{Q}_{\ell}),\]
where $E_{2}^{2,0}=H_{\acute{e}t}^{2}(\Xspi,\mathbb{Q}_{\ell}),\; E_{2}^{1,1}=0$
and $E_{2}^{0,2}=H_{\acute{e}t}^{2}(\Xspi,R^{2}{\varphi_{2}}_{*}(\mathbb{Q}_{\ell}))$.
Due to Lemma \ref{lem:conic} below, we have the claimed property
for $E_{2}^{0,2}$ as well as $E_{2}^{2,0}$, hence for $H_{\acute{e}t}^{2}(\Xspii,\mathbb{Q}_{\ell})$,
too. To go from $\Xspii$ to $\Xspiii=\cXsp$, we can use the argument
in {[}\textit{ibid}, Lemma 2.16{]} since the exceptional divisor $E_{3}(=E_{3}^{(3)})$
of $\varphi_{3}:\Xspiii\rightarrow\Xspii$ is a $\mathbb{P}^{1}$-bundle
over $\mathbb{P}^{1}$. Thus we obtain the claimed property for $H_{\acute{e}t}^{2}(\cXsp,\mathbb{Q}_{\ell})$.\end{proof}
\begin{lem}
\label{lem:lefschets-Fp}Consider $\Xsp$ as the linear section $(\mathbb{P}^{4}\times\mathbb{P}^{4})\cap H_{1}\cap...\cap H_{5}$
in $\mathbb{P}^{24}$ by the Segre embedding with $H_{k}$ representing
the defining equation $f_{k}=\,^{t}zA_{k}w\,(k=1,...,5)$. Then for
all but finite primes $p$, there exists a sequence linear forms $H_{1}',H_{2}',...,H_{5}'$
over $\mathbb{Z}$ with the following properties over $\mathbb{F}_{p}$:
1) Sing$(X_{i-1})$$\subset X_{i}$ holds for $i=2,..,5,$ where $X_{i}=(\mathbb{P}^{4}\times\mathbb{P}^{4})\cap H_{1}'\cap..\cap H_{i}'$
and Sing$(X_{i-1})$ is the singular loci of $X_{i-1}$. 2) $X_{5}=\Xsp$. \end{lem}
\begin{proof}
Since $f_{k}$'s are defined over $\mathbb{Z}$, it suffices to have
the properties 1) and 2) over $\mathbb{C}$. We can verify explicitly
that the sequence $H_{1}',H_{2}',...,H_{5}'$ corresponding to $f_{1}+f_{3}+f_{5},f_{2}+f_{4},3f_{2}+f_{5},5f_{3}+f_{4},f_{5}$
satisfies the desired properties over $\mathbb{C}$. \end{proof}
\begin{lem}
\label{lem:conic} All the eigenvalues of $\mathrm{Frob}_{p}$ on
$H_{\acute{e}t}^{0}(\Xspi,R^{2}{\varphi_{2}}_{*}(\mathbb{Q}_{l}))$
are equal to $p$ for every good prime $p$. 
\end{lem}
\begin{proof} Set $\rho_{1}:=\varphi_{2}|_{E_{2}^{(2)}}$. First
note that $H_{\acute{e}t}^{0}(\Xspi,R^{2}{\varphi_{2}}_{*}\mathbb{Q}_{l})\simeq H_{\acute{e}t}^{0}(E_{2}^{(2)},R^{2}{\rho_{1}}_{*}\mathbb{Q}_{l})$.
Let $\rho_{2}\colon E_{2}^{(3)}\to E_{2}^{(2)}$ be the blow-up of
the ordinary double point of $E_{2}^{(2)}$ on the fiber of $E_{2}^{(2)}\to\mathbb{P}^{1}$
over $s=0$ (see Proposition \ref{pro:blow-up-along-s-2A1}). Denote
by $\rho$ the composite of $\rho_{2}$ and $\rho_{1}$. We have the
spectral sequence: \begin{equation}
E_{2}^{i,j}:=R^{i}{\rho_{1}}_{*}(R^{j}{\rho_{2}}_{*}\mathbb{Q}_{l})\Longrightarrow R^{i+j}\rho_{*}(\mathbb{Q}_{l}).\label{eq:spec1}\end{equation}
 By standard calculations, we have 
\begin{itemize}
\item $E_{2}^{2,0}=R^{2}{\rho_{1}}_{*}(\mathbb{Q}_{l}$). 
\item Since the nontrivial fiber of $\rho_{2}$ is a $\mathbb{P}^{1}$,
we have $R^{1}{\rho_{2}}_{*}(\mathbb{Q}_{l})=0$. Hence $E_{2}^{1,1}=0$. 
\item $E_{2}^{0,2}\simeq H_{\acute{e}t}^{2}(\mathbb{P}^{1},\mathbb{Q}_{l})$,
where $\mathbb{P}^{1}$ is the nontrivial fiber of $\rho_{2}$, and
we consider $H_{\acute{e}t}^{2}(\mathbb{P}^{1},\mathbb{Q}_{l})$ as
a skyscraper sheaf supported on $s=0$. 
\end{itemize}
Then, by standard properties of the spectral sequence, we have the
following exact sequence: \[
0\to R^{2}{\rho_{1}}_{*}(\mathbb{Q}_{l})\to R^{2}{\rho}_{*}(\mathbb{Q}_{l})\to H_{\acute{e}t}^{2}(\mathbb{P}_{1},\mathbb{Q}_{l})\to0.\]
 Therefore, to show the claimed property for $H_{\acute{e}t}^{0}(\Xspi,R^{2}{\varphi_{2}}_{*}(\mathbb{Q}_{l}))$,
we have only to show that the claimed property holds for $H_{\acute{e}t}^{0}(E_{2}^{(2)},R^{2}{\rho}_{*}(\mathbb{Q}_{l}))$.

Let $\rho_{3}\colon E_{2}^{(3)}\to E_{2}'$ be the contraction of
three $(-1)$-curves on $E_{2}^{(3)}$, two of which are the strict
transforms of the components of the fiber of $E_{2}^{(2)}\to\mathbb{P}^{1}$
over $s=0$, and the remaining one of which is one component of the
fiber of $E_{2}^{(2)}\to\mathbb{P}^{1}$ over $s=\infty$ (see Proposition
\ref{pro:blow-up-along-s-2A1}). Denote by $\rho_{4}\colon E_{2}'\to\mathbb{P}^{1}$
the natural induced morphism, which defines a $\mathbb{P}^{1}$-bundle
structure. We have the spectral sequence: \begin{equation}
R^{i}{\rho_{4}}_{*}(R^{j}{\rho_{3}}_{*}(\mathbb{Q}_{l}))\Longrightarrow R^{i+j}\rho_{*}(\mathbb{Q}_{l}).\label{eq:spec2}\end{equation}
 By similar considerations to those for (\ref{eq:spec1}), we have
the following exact sequence: \[
0\to R^{2}{\rho_{4}}_{*}(\mathbb{Q}_{l})\to R^{2}{\rho}_{*}(\mathbb{Q}_{l})\to H_{\acute{e}t}^{2}(\mathbb{P}^{1},\mathbb{Q}_{l})^{\oplus3}\to0.\]
 Note that all eigenvalues of $\mathrm{Frob}_{p}$ on $H_{\acute{e}t}^{2}(\mathbb{P}_{1},\mathbb{Q}_{l})^{\oplus3}$
are equal to $p$. Therefore, to show that the claimed property holds
for $H_{\acute{e}t}^{0}(E_{2}^{(2)},R^{2}{\rho}_{*}(\mathbb{Q}_{l}))$,
we have only to show that the claimed property holds for $H_{\acute{e}t}^{0}(E_{2}',R^{2}{\rho_{4}}_{*}(\mathbb{Q}_{l}))$.

Now we consider the Leray spectral sequence: \[
H_{\acute{e}t}^{i}(\mathbb{P}^{1},R^{j}{\rho_{4}}_{*}(\mathbb{Q}_{l}))\Longrightarrow H_{\acute{e}t}^{i+j}(E_{2}',\mathbb{Q}_{l}).\]
 Since $\rho_{4}$ is a $\mathbb{P}^{1}$-bundle, we have $R^{1}{\rho_{4}}_{*}(\mathbb{Q}_{l})=0$.
Therefore, in a similar way as above, we have the following exact
sequence: \[
0\to H_{\acute{e}t}^{2}(\mathbb{P}^{1},\mathbb{Q}_{l})\to H_{\acute{e}t}^{2}(E_{2}',\mathbb{Q}_{l})\to H_{\acute{e}t}^{0}(E_{2}',R^{2}{\rho_{4}}_{*}(\mathbb{Q}_{l}))\to0.\]
 Since $\rho_{1}\colon E_{2}^{(2)}\to\mathbb{P}^{1}$ has a section
defined over $\mathbb{Q}$, due to 2) in Proposition \ref{pro:blow-up-along-s-2A1}
applied to $a,b\in\mathbb{Z}$, so does $\rho_{4}\colon E_{2}'\to\mathbb{P}^{1}$.
Therefore $H_{\acute{e}t}^{2}(E_{2}',\mathbb{Q}_{l})$ is generated
by the classes of divisors defined over $\mathbb{Q}$, which are a
section and a fiber. Hence all eigenvalues of $\mathrm{Frob}_{p}$
on $H_{\acute{e}t}^{2}(E_{2}',\mathbb{Q}_{l})$ are equal to $p$
\cite{vanG}, and then the claimed property holds for $H_{\acute{e}t}^{0}(E_{2}',R^{2}{\rho_{4}}_{*}(\mathbb{Q}_{l}))$. 

\end{proof}

\begin{figure}
\includegraphics[scale=0.5]{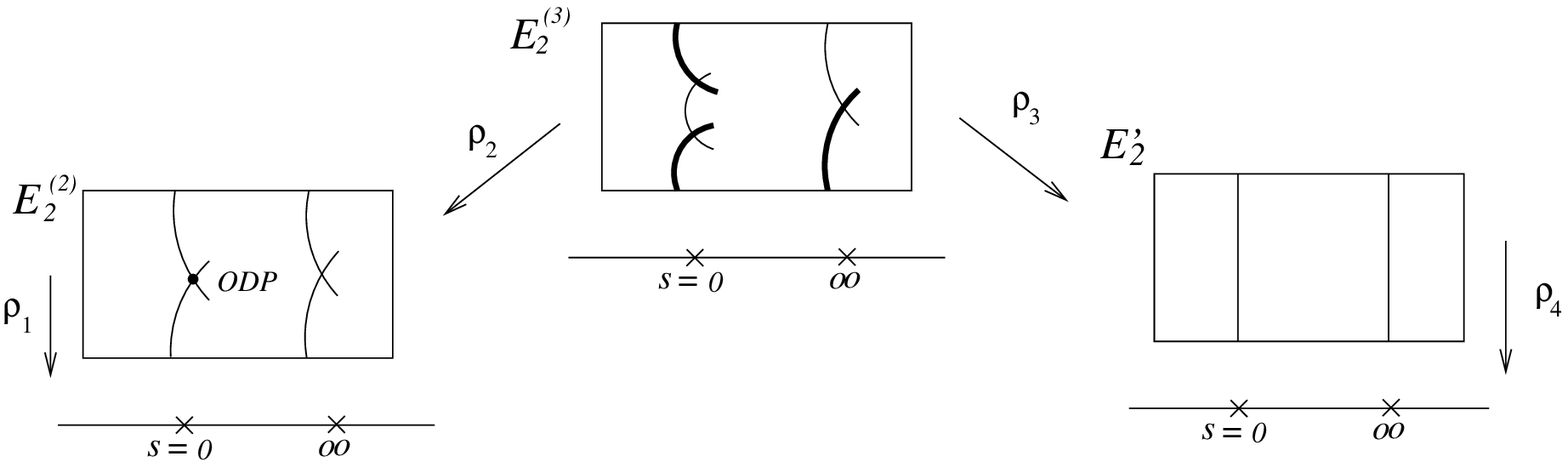}

\caption{}

\end{figure}
From the Proposition \ref{pro:mod-p-H2} and the fixed point formula
(\ref{eq:Np-fixed-pt-formula}), we have

\begin{equation}
\bigl|1+(50+h^{21})(p+p^{2})+p^{3}-\#\cXsp(\mathbb{F}_{p})\bigr|\leq(2+2\, h^{21})p^{\frac{3}{2}},\label{eq:Np-ineq}\end{equation}
where we have used $b_{2}=b_{4}$ by the Poincar\'e duality and also
expressed $b_{2}=h^{11}=(50+h^{21})$ from $e(\cXsp)=2(h^{11}-h^{21})=100.$
\begin{prop}
The number of rational points $\#\cXsp(\mathbb{F}_{p})$ is given
by \[
\#\cXsp(\mathbb{F}_{p})=\#\Ssp(\mathbb{F}_{p})+10\times\#E_{1}(\mathbb{F}_{p})+30\, p^{2}+40\, p-10,\]
where $\#\Ssp(\mathbb{F}_{p})$ and $\#E_{1}(\mathbb{F}_{p})$ are
the numbers of rational points over $\mathbb{F}_{p}$ for the determinantal
quintic (\ref{eq:defEqS}) and the singular del Pezzo surface in Proposition
\ref{pro:g1-g2}, respectively, with $a=b=1$. \end{prop}
\begin{proof}
The projection $\pi_{2}:\Xsp\rightarrow\Ssp$ is isomorphic outside
the coordinate lines $q_{i}$ ( see Fig. \ref{fig:FigResolution-S}).
Since the fibers over the coordinate point $[e_{i}]$ and $q_{i}\setminus\{[e_{i}],[e_{i+1}]\}$
are $\mathbb{P}^{2}$ and $\mathbb{P}^{1},$ respectively, we obtain\begin{align*}
\#\Xsp(\mathbb{F}_{p}) & =\#\Ssp(\mathbb{F}_{p})+5\times(N_{\mathbb{P}^{2}}-1)+5\times(N_{\mathbb{P}^{1}}-1)(N_{\mathbb{P}^{1}}-2)\\
 & =\#\Ssp(\mathbb{F}_{p})+10p^{2},\end{align*}
where $N_{\mathbb{P}^{2}}=p^{2}+p+1$ and $N_{\mathbb{P}^{1}}=p+1$,
respectively, count the number of rational points in $\mathbb{P}^{2}$
and $\mathbb{P}^{1}$ over $\mathbb{F}_{p}$. We count the number
of rational points on the conic bundle $E_{2}$ (with two reducible
fibers) over $\mathbb{P}^{1}$ as \[
\#E_{2}(\mathbb{F}_{p})=(p+1)(p-1)+(2p+1)\times2=p^{2}+4p+1.\]
The counting for $E_{3}$ is given by $\#E_{3}(\mathbb{F}_{p})=(p+1)^{2}$.
Now summarizing all, we obtain \begin{align*}
\#\cXsp(\mathbb{F}_{p}) & =\#\Xsp(\mathbb{F}_{p})+10\times\bigl(\#E_{1}(\mathbb{F}_{p})-1\bigr)\\
 & \qquad\qquad+10\times\bigl(\#E_{2}(\mathbb{F}_{p})-(p+1)\bigr)+10\times\bigl(\#E_{3}(\mathbb{F}_{p})-(p+1)\bigr)\\
 & =\#\Ssp(\mathbb{F}_{p})+10\times\#E_{1}(\mathbb{F}_{p})+30\, p^{2}+40\, p-10.\end{align*}

\end{proof}
Writing a straightforward computer codes, we have evaluated the number
$\#\cXsp(\mathbb{F}_{p})$. After the computations in several minutes,
we verify the inequality (\ref{eq:Np-ineq}) for $p\leq97$ with $h^{2,1}=2$
or $3.$ For example, we obtain $\#\cXsp(\mathbb{F}_{p})=669880,\,1118250$
and 1408330 for $p=73,\,89$ and $97$, respectively. We observe that
the inequality (\ref{eq:Np-ineq}) holds only if $h^{2,1}=2$ for
$p=59,61,71,73,89,97$. Also we can verify that these are good primes
by analyzing the Jacobian ideals over the field $\mathbb{F}_{p}.$
Since the inequality holds for all good primes, we conclude that:
\begin{thm}
\label{thm:mirror-Hodge}The smooth Calabi-Yau manifold $\cXsp$ has
Hodge numbers; \[
h^{1,1}(\cXsp)=52,\;\;\; h^{2,1}(\cXsp)=2.\]
In particular this is mirror symmetric to the generic complete intersection
$\tilde{X}_{0}$ with $h^{1,1}(\tilde{X}_{0})=2,\, h^{2,1}(\tilde{X}_{0})=52$.
\end{thm}
\vspace{2cm}
\vfill$\,$\pagebreak{}

\section{\textbf{\textup{\label{sec:Picard-Fuchs-equations}Picard-Fuchs equations
and monodromy matrices}}}

\subsection{Picard-Fuchs differential equations}

We consider a family of Calabi-Yau manifolds $\cXsp$ defined over
$(\mathbb{C}^{*})^{2}\ni(a,b)$. Here we briefly introduce a natural
compactification of $(\mathbb{C}^{*})^{2}$ to $\mathbb{P}^{2}$ which
follows from the differential equations satisfied by the period integrals,
see \cite{HKTY} and \cite{HoTa} for details. To formulate the set
of differential operators, we slightly modify the defining equations
(\ref{eq:defeqsCICY}) to\[
f_{i}=c_{i}z_{i}w_{i}+a_{i}z_{i+1}w_{i}+b_{i}z_{i}w_{i+1}\;\;(i=1,...,5),\]
where the indices are considered modulo five as before. Clearly, the
original forms are recovered by setting $a_{i}=a,b_{i}=b,c_{i}=1$.
Since we have $\cXsp|_{(\mathbb{C}^{*})^{4}\times(\mathbb{C}^{*})^{4}}\simeq\Xsp|_{(\mathbb{C}^{*})^{4}\times(\mathbb{C}^{*})^{4}}$
for $(\mathbb{C}^{*})^{4}\times(\mathbb{C}^{*})^{4}\subset\mathbb{P}^{4}\times\mathbb{P}^{4}$,
a holomorphic 3-form of the crepant resolution $\cXsp$ may be given
by the corresponding 3-form of $\Xsp$ if the 3-cycles of the period
integrals are contained in $(\mathbb{C}^{*})^{4}\times(\mathbb{C}^{*})^{4}$.
For the complete intersection $\Xsp$, the following expression of
a holomorphic 3-form is well-known \cite{Griffiths}:\begin{equation}
\Omega(\Xsp)=Res_{f_{1}=...=f_{5}=0}\left(\frac{d\mu_{1}\wedge d\mu_{2}}{f_{1}f_{2}\cdots f_{5}}\right),\label{eq:OmegaXsp}\end{equation}
where \[
d\mu_{1}=-\sum_{i=1}^{5}(-1)^{i}z_{i}dz_{1}\wedge\cdots\wedge\widehat{dz_{i}}\wedge\cdots\wedge dz_{5},\]
and similar definition for $d\mu_{2}$ with the coordinates $w_{k}$'s
. The period integral $\int_{\Gamma}\Omega(\cXsp)$ for a 3-cycle
$\Gamma\in H_{3}(\cXsp,\mathbb{Z})$ satisfies a system of differential
equations, the so-called Picard-Fuchs differential equations, see
\cite{DMo}, \cite{DGJ} for example. In the present case, assuming
that the cycle $\Gamma$ is contained in $(\mathbb{C}^{*})^{4}\times(\mathbb{C}^{*})^{4}$,
we can describe the system by noting rather trivial algebraic relations
represented in terms of differential operators, e.g., \[
\left\{ \frac{\partial\;}{\partial c_{1}}\frac{\partial\;}{\partial c_{2}}\frac{\partial\;}{\partial c_{3}}\frac{\partial\;}{\partial c_{4}}\frac{\partial\;}{\partial c_{5}}-\frac{\partial\;}{\partial a_{1}}\frac{\partial\;}{\partial a_{2}}\frac{\partial\;}{\partial a_{3}}\frac{\partial\;}{\partial a_{4}}\frac{\partial\;}{\partial a_{5}}\right\} \Omega(\Xsp)=0,\]
which represents $\Pi_{i=1}^{5}z_{i}w_{i}-\Pi_{i=1}^{5}z_{i+1}w_{i}=0$.
We should also note that the holomorphic 3-form is invariant under
the $(\mathbb{C}^{*})^{4}$-action $z_{i}\mapsto t_{i}z_{i}$, $(t_{1}t_{2}\cdots t_{5}=1)$
and similar $(\mathbb{C}^{*})^{4}$-action on the coordinates $w_{i}$'s.
We note further that $\Omega(\Xsp)$ has a simple scaling property
under $f_{i}\mapsto r_{i}f_{i}$ $(r_{i}\in\mathbb{C}^{*})$. All
these properties of invariance (or covariance) may be expressed by
the corresponding linear differential operators, and may be used to
reduce the enlarged parameters to the original $a$ and $b$. The
system of differential operators which we obtain in this way is an
example of the Gel'fand-Kapranov-Zelevinski (GKZ) system \cite{GKZ1}
for which a natural compactification of the parameters is known. In
the present case, from the $\mathbb{C}^{*}$-actions above and the
form of the defining equations (\ref{eq:defeqsCICY}), it is rather
easy to deduce that $(\mathbb{C}^{*})^{2}\ni(a,b)$ is compactified
to $\mathbb{P}^{2}\ni[a^{5}:b^{5}:1]$. According to the mirror symmetry
calculations formulated in \cite{HKTY}, we actually come to the affine
charts $\{(x,y),\mathcal{A}_{0}\},\{(x_{1},y_{1}),\mathcal{A}_{1}\}$
and $\{(x_{2},y_{2}),\mathcal{A}_{2}\}$ defined by \[
x=-a^{5},y=-b^{5};\; x_{1}=-\frac{b^{5}}{a^{5}},y_{1}=-\frac{1}{a^{5}};\; x_{2}=-\frac{a^{5}}{b^{5}},y_{2}=-\frac{1}{b^{5}}.\]
Up to signs, these relations are in accord with the standard relations
$[a^{5}:b^{5}:1]=[1:\frac{b^{5}}{a^{5}}:\frac{1}{a^{5}}]=[\frac{a^{5}}{b^{5}}:1:\frac{1}{b^{5}}]$
of the affine coordinates of $\mathbb{P}^{2}$. The extra minus signs
follows from the general definition given in \cite{HKTY}.
\begin{prop}
\label{pro:PF-equations}On the affine chart $\{(x,y),\mathcal{A}_{0}\}$$,$
the following differential operators determine the period integrals
as the solutions:\begin{align*}
\mathcal{D}_{1}(x,y) & =2\theta_{x}^{3}-3\theta_{x}^{2}\theta_{y}+3\theta_{x}\theta_{y}^{2}-2\theta_{y}^{3}-(\theta_{x}+\theta_{y})^{2}\bigl\{(2\theta_{x}+3\theta_{y})x-(3\theta_{x}+2\theta_{y})y\bigr\},\\
\mathcal{D}_{2}(x,y) & =2\theta_{x}^{2}-3\theta_{x}\theta_{y}+2\theta_{y}^{2}-(2\theta_{x}^{2}+7\theta_{x}\theta_{y}+7\theta_{y}^{2})x-(7\theta_{x}^{2}+7\theta_{x}\theta_{y}+2\theta_{y}^{2})y,\end{align*}
where $\theta_{x}=x\frac{\partial\;}{\partial x},\theta_{y}=y\frac{\partial\;}{\partial y}.$
On the other affine charts the differential operators are given by
the following gauge transforms of the operators $\mathcal{D}_{1}(x,y),\mathcal{D}_{2}(x,y)$:
\[
\mathcal{D}'_{1}(x_{1},y_{1}):=x_{1}\mathcal{D}_{1}(x_{1},y_{1})x_{1}^{-1},\;\mathcal{D}'_{2}(x_{1},y_{1}):=\mathcal{D}_{2}(x_{1},y_{1})\;\;\text{on\;}\{(x_{1},y_{1}),\mathcal{A}_{1}\}\]
and \[
\mathcal{D}''_{1}(x_{2},y_{2}):=x_{2}\mathcal{D}_{1}(x_{2},y_{2})x_{2}^{-1},\;\mathcal{D}''_{2}(x_{2},y_{2}):=\mathcal{D}_{2}(x_{2},y_{2})\;\;\text{on}\;\{(x_{2},y_{2}),\mathcal{A}_{2}\}.\]
\end{prop}
\begin{proof}
For the derivation of the differential operators $\mathcal{D}_{1},\mathcal{D}_{2}$,
we refer to \cite{HKTY}. Also see Prop.2.6 in \cite{HoTa}. Note
that the parameters $(a_{i},b_{i},c_{i})$ in \cite[(2.6) ]{HoTa}
should be read as $(c_{i},a_{i},b_{i})$ here (see the defining equations
$f_{i}$ ).
\end{proof}
\vspace{1cm}

\subsection{Determinantal quintics}

For the determinantal quintics $\Ssp,$ and $\Hsp$, we have the following
standard forms of holomorphic 3-forms:\begin{equation}
\Omega(\Ssp)=Res_{F_{w}=0}\left(\frac{d\mu_{2}}{F_{w}}\right),\;\;\Omega(\Hsp)=Res_{F_{\lambda}=0}\left(\frac{d\mu_{\lambda}}{F_{\lambda}}\right),\label{eq:OmegaSH}\end{equation}
where $\Ssp=\{F_{w}=0\}$ and $\Hsp=\{F_{\lambda}=0\}$ (see (\ref{eq:defEqS})).
We may derive these holomorphic 3-forms from (\ref{eq:OmegaXsp})
by evaluating the residue integrals: Let us take an affine coordinate
$[z_{1}:z_{2}:z_{3}:z_{4}:1]$ of $\mathbb{P}^{4}$, and regard the
relations $f_{1}=\cdots=f_{4}=0$ as linear equations for $z_{1,}...,z_{4}$
with fixed $w_{k}$'s, i.e.,\[
B\left(\begin{smallmatrix}z_{1}\\
z_{2}\\
z_{3}\\
z_{4}\end{smallmatrix}\right)=\left(\begin{smallmatrix}0\\
0\\
0\\
-aw_{4}\end{smallmatrix}\right).\]
Then, changing the variables to $\,^{t}(\xi_{1},...,\xi_{4})=B\,^{t}(z_{1,}...,z_{4})$
and taking into account the Jacobian factor $dz_{1}\wedge...\wedge dz_{4}=\frac{1}{detB}d\xi_{1}\wedge...\wedge d\xi_{4},$
we obtain \[
\Omega(\Xsp)=Res_{f_{5}=0}\left(\frac{d\mu_{2}}{\det B\; f_{5}}\right).\]
Since we can verify the equality $\det B\; f_{5}=F_{w}$, we see that
$\Omega(\Xsp)=\Omega(\Ssp)$ holds. By changing the roles of $z_{k}$'s
with $w_{k}$'s, we have a similar result for $\Omega(\Sspp).$ The
threefolds $\Uspp$ and $\Usp$ in the diagram (\ref{eq:SHdiagram2})
also have the form of complete intersections of five $(1,1)$-divisors.
The same formal arguments as above apply to the cases of $\Uspp$
and $\Usp$ starting from $\Omega(\Uspp)$ and $\Omega(\Usp)$, respectively.
By evaluating the residues, the holomorphic 3-forms $\Omega(\tilde{X}_{i}^{sp})\,(i=1,2)$
can also be connected to the holomorphic 3-form $\Omega(\Hsp)$ as
well as $\Omega(\tilde{Z}_{i}^{sp})$. Noting that there are 3-cycles
contained in the tori (see the next subsection), we have: 
\begin{prop}
The period integrals of $\Sspp,\Ssp$,$\Uspp,\Usp$, and $\Hsp$ with
the holomorphic 3-forms $\Omega(\Sspp),\,\Omega(\Ssp),\,\Omega(\Uspp),\,\Omega(\Usp)$
and $\Omega(\Hsp)$, respectively, satisfy the same Picard-Fuchs differential
equations as in Proposition \ref{pro:PF-equations}. 
\end{prop}
\vspace{1cm}

\subsection{Integral, symplectic basis and monodromy matrices}

As in \cite{Candelas1}, we can evaluate the period integral of $\Omega(\Ssp)$
over certain torus cycles. Let us first note that $\Gamma_{0}=\{[w_{1}:w_{2}:w_{3}:w_{4}:1]\in S_{sp}||w_{1}|=|w_{2}|=|w_{3}|=\varepsilon\}$
defines a 3-cycle in $\Ssp$. This simply follows by observing that
the substitution of $w_{k}=\varepsilon e^{\sqrt{-1}\theta_{k}}(k=1,2,3)$
(in the affine coordinate $w_{5}=1$) into the defining equation of
$\Ssp$ entails a quadratic equation for $w_{4}$, and one of the
two roots goes to zero when $\varepsilon\rightarrow0.$ Choosing this
vanishing root defines a 3-cycle $\Gamma_{0}$. Combined with the
residues contained in the definition of $\Omega(\Ssp)$, one obtain
\begin{equation}
\int_{\Gamma_{0}}\Omega(\Ssp)=\frac{1}{(2\pi i)^{4}}\int_{\gamma_{0}}\frac{d\mu_{2}}{F_{w}},\label{eq:period-int-gamma0}\end{equation}
where $\gamma_{0}=\{|w_{1}|=\cdots=|w_{4}|=\varepsilon,w_{5}=1\}$
is a torus cycle in $\mathbb{P}^{4}.$ 
\begin{prop}
The period integral (\ref{eq:period-int-gamma0}) can be evaluated
in three different ways depending on the (relative) magnitudes of
$a$ and $b$: \[
\int_{\Gamma_{0}}\Omega(\Ssp)=w_{0}\bigl(-a^{5},-b^{5}\bigr),\;\;\frac{1}{a^{5}}w_{0}\bigl(\frac{-1}{a^{5}},-\frac{b^{5}}{a^{5}}\bigr),\;\;\frac{1}{b^{5}}w_{0}\bigl(-\frac{a^{5}}{b^{5}},\frac{-1}{b^{5}}\bigr),\]
 where we set $w_{0}(x,y)=\sum_{n,m\geq0}\frac{((n+m)!)^{5}}{(n!)^{5}(m!)^{5}}x^{n}y^{m}$.
The series $w_{0}(x,y)$ converges absolutely for $|x|,|y|<\frac{1}{2^{5}}$.\end{prop}
\begin{proof}
Since the cycle $\gamma_{0}$ is contained in the affine coordinate
$w_{5}=1$ (in fact $\gamma_{0}\subset(\mathbb{C}^{*})^{4}$), we
may use $\frac{d\mu_{2}}{F_{w}}=\frac{dw_{1}\wedge\cdots\wedge dw_{4}}{F_{w}(w_{1},\cdots,w_{4},1)}$
for the evaluations. The claimed expansions follow from the three
different ways of handling $\frac{1}{F_{w}}$: The first one is obtained
by \[
\frac{dw_{1}\wedge\cdots\wedge dw_{4}}{F_{w}}=\left(1+\left(\frac{F_{w}}{w_{1}w_{2}w_{3}w_{4}}-1\right)\right)^{-1}\frac{dw_{1}\wedge\cdots\wedge dw_{4}}{w_{1}w_{2}w_{3}w_{4}}\]
and taking the residue integrals about $|w_{k}|=\varepsilon$, see
\cite{Bat-Cox} for example. Similarly, the second one follows from
\[
\frac{dw_{1}\wedge\cdots\wedge dw_{4}}{F_{w}}=\frac{1}{a^{5}}\left(1+\left(\frac{F_{w}}{a^{5}w_{1}w_{2}w_{3}w_{4}}-1\right)\right)^{-1}\frac{dw_{1}\wedge\cdots\wedge dw_{4}}{w_{1}w_{2}w_{3}w_{4}}.\]
For the third one, we simply replace the $a^{5}$'s by $b^{5}$'s
in the above equation.

Since the convergence follows from the standard estimates using the
duplication formula of the $\Gamma$-functions, our derivation may
be brief here. Assume $|x|,|y|<r,$ then we have \[
\sum_{d\geq0}\sum_{{n+m=d\atop n,m\geq0}}c_{n,m}|x^{n}||y^{m}|\leq1+\sum_{d\geq1}(d+1)c_{[\frac{d}{2}],[\frac{d+1}{2}]}r^{d}\leq1+\sum_{d\geq1}(d+1)\left(\frac{2^{d}}{\sqrt{\pi}}\right)^{5}r^{d},\]
where $c_{n,m}=\bigl(\frac{(n+m)!}{n!m!}\bigr)^{5}$, and the duplication
formula is used to have the second inequality. Since the last series
converges for $2^{5}r<1,$ we obtain the claim.
\end{proof}
It should be clear that the three different series expansions of the
period integral originate from the symmetry of the defining equations
$f_{i}$ of $\Xsp$, which we started with. Also, we can observe here
the natural compactification of the deformations by $(a,b)\in(\mathbb{C}^{*})^{2}$
to $[a^{5}:b^{5}:1]\in\mathbb{P}^{2}$ discussed above. Moreover,
we may observe that the three infinity points $[0:0:1],[0:1:0]$ and
$[1:0:0]$ are all isomorphic up to suitable factors or {}``gauge''
transformations as claimed in Proposition \ref{pro:PF-equations}. 

In the next subsection, we set up a canonical integral, symplectic
basis for the solutions which follows from the mirror symmetry. 

\bigskip{}

\subsubsection{Canonical integral and symplectic basis}

The space of the solutions of the Picard-Fuchs differential equation
is endowed with an integral and symplectic structure in their monodromy
property which come from those in $H_{3}(\cXsp,\mathbb{Z})$. Using
the mirror symmetry of $\tilde{X}_{0}^{*}$ to $\tilde{X}_{0}$, we
have a canonical form of the (conjectural) integral and symplectic
basis of the solutions \cite[Prop.1]{IIAmonod}, \cite[Conj.2.2]{CentralCh}.

Recall that, under the mirror symmetry, the integral and symplectic
structure in $H_{3}(\cXsp,\mathbb{Z})$ is conjecturally isomorphic
to those in the (numerical) Grothendieck group $K(\tilde{X}_{0})$
of the mirror Calabi-Yau manifold $\tilde{X}_{0}$ to $\cXsp$ \cite{Ko}.
Note that the Euler characteristic $\chi(\mathcal{E},\mathcal{F})=\sum_{i}(-1)^{i}\dim Ext_{\mathcal{O}_{\tilde{X}_{0}}}^{i}(\mathcal{E},\mathcal{F})$
of coherent sheaves $\mathcal{E},$$\mathcal{F}$ on $\tilde{X}_{0}$
defines a skew symmetric form on the Grothendieck group $K(\tilde{X}_{0})$
due to the fact that $\tilde{X}_{0}$ is a Calabi-Yau threefold. This
skew symmetric form (as well as the integral structure) in $K(\tilde{X}_{0})$
may be transferred into $H^{even}(\tilde{X}_{0},\mathbb{Q})$ by the
Chern character homomorphism: $\text{ch}:K(\tilde{X}_{0})\rightarrow H^{even}(\tilde{X}_{0},\mathbb{Q})$
and the Riemann-Roch formula for $\chi(\mathcal{E},\mathcal{F})$.
Explicitly, the skew form on $H^{even}(\tilde{X}_{0},\mathbb{Q})$
may be written by $(\alpha,\beta)=\int_{\tilde{X}_{0}}(\alpha_{0}-\alpha_{2}+\alpha_{4}-\alpha_{6})\cup(\beta_{0}+\beta_{2}+\beta_{4}+\beta_{6})\cup Td_{\tilde{X}_{0}}$,
where $Td_{\tilde{X}_{0}}$ represents the Todd class and $\alpha=\alpha_{0}+\alpha_{2}+\alpha_{4}+\alpha_{6}$
represents the decomposition with respect to $H^{even}(\tilde{X}_{0},\mathbb{Q})=\oplus_{i=0}^{3}H^{2i}(\tilde{X}_{0},\mathbb{Q})$
and similarly for $\beta=\beta_{0}+\beta_{2}+\beta_{4}+\beta_{6}$. 

The Calabi-Yau manifold $\tilde{X}_{0}$ is a smooth complete intersection
of five generic $(1,1)$-divisors in $\mathbb{P}^{4}\times\mathbb{P}^{4}$.
The cohomology $H^{even}(\tilde{X}_{0},\mathbb{Q})$ is generated
by the hyperplane classes $J_{1},J_{2}$ from the respective projective
spaces $\mathbb{P}^{4}$ with the ring structure compatible with their
intersection numbers $(\int_{\tilde{X}_{0}}J_{1}^{3},\int_{\tilde{X}_{0}}J_{1}^{2}J_{2},\int_{\tilde{X}_{0}}J_{2}J_{1}^{2},\int_{\tilde{X}_{0}}J_{2}^{3})=(5,10,10,5).$
Using this ring structure in $H^{even}(\tilde{X}_{0},\mathbb{Q})$,
the mirror symmetry stated above can be summarized into the following
cohomology-valued hypergeometric series \cite[Sect.2]{CentralCh}:
\begin{equation}
\omega\left(x,y;\frac{J_{1}}{2\pi i},\frac{J_{2}}{2\pi i}\right)=\sum_{n,m\geq0}\frac{\Gamma(1+n+m+\frac{J_{1}}{2\pi i}+\frac{J_{2}}{2\pi i})^{5}}{\Gamma(1+n+\frac{J_{1}}{2\pi i})^{5}\Gamma(1+m+\frac{J_{2}}{2\pi i})^{5}}x^{n+\frac{J_{1}}{2\pi i}}y^{m+\frac{J_{2}}{2\pi i}},\label{eq:cohomology-valued-hyp}\end{equation}
where the right hand side is defined by the series expansion with
respect to the nilpotent elements $J_{1},J_{2}$ in the cohomology
ring. By this series expansion in the cohomology ring, we effectively
generate the solutions of the Picard-Fuch differential equations formulated
in \cite{HLY}, \cite{HKTY}. Then the (conjectural) claim made in
\cite[Prop.1]{IIAmonod}, \cite[Conj.2.2]{CentralCh} is as follows:
In this form of the cohomology-valued hypergeometric series, the integral
and symplectic structure in $H^{even}(\tilde{X}_{0},\mathbb{Q})$
is transformed canonically to that of the hypergeometric series representing
the period integrals. The canonical integral, symplectic structure
may be read by arranging $\omega\left(x,y;\frac{J_{1}}{2\pi i},\frac{J_{2}}{2\pi i}\right)$
as follows: \[
w_{0}(x,y)1+\sum_{k}w_{k}^{(1)}(x,y)\bigl(J_{k}-\sum_{l}C_{kl}K_{l}\bigr)Td_{\tilde{X}}^{-1}+\sum_{k}w_{k}^{(2)}(x,y)K_{k}+w^{(3)}(x,y)V_{\tilde{X}},\]
where $Td_{\tilde{X}_{0}}=1+\frac{c_{2}(\tilde{X}_{0})}{12}$ is the
Todd class and $K_{k}=\frac{1}{5}J_{k}^{2},V_{\tilde{X}_{0}}=-\frac{1}{10}(J_{1}^{3}+J_{2}^{3}).$
Here, $K_{k}$ and $V_{\tilde{X}_{0}}$ are defined so that we have
$\int_{\tilde{X}_{0}}J_{k}K_{l}=\delta_{kl}$ and $\int_{\tilde{X}_{0}}V_{\tilde{X}_{0}}=-1.$
$C_{kl}$'s are constants satisfying $C_{kl}=C_{lk}$ which must be
fixed (by hand) from the explicit monodromy calculations of the hypergeometric
series (Proposition \ref{pro:Monodromy-Matrices}). The integral structure
on $H^{even}(\tilde{X}_{0},\mathbb{Q})$ can be introduced through
the basis $\{1,\bigl(J_{k}-\sum_{l}C_{kl}K_{l}\bigr)Td_{\tilde{X}}^{-1},K_{k},V_{\tilde{X}}\}$
by noting ${\rm ch}(\mathcal{O}_{\tilde{X}_{0}})=1$, ${\rm ch}(\mathcal{O}_{p})=-V_{\tilde{X}_{0}}$,
etc. Then, with respect to this basis, the symplectic form $(*,*):H^{even}(\tilde{X}_{0},\mathbb{Q})\times H^{even}(\tilde{X}_{0},\mathbb{Q})\rightarrow\mathbb{Z}$
described above takes the following form: \begin{equation}
\Sigma_{0}=\left(\begin{smallmatrix}0 & 0 & 0 & 0 & 0 & 1\\
0 & 0 & 0 & 0 & 1 & 0\\
0 & 0 & 0 & 1 & 0 & 0\\
0 & 0 & -1 & 0 & 0 & 0\\
0 & -1 & 0 & 0 & 0 & 0\\
-1 & 0 & 0 & 0 & 0 & 0\end{smallmatrix}\right)\label{eq:symplectic-form}\end{equation}
 with no dependence on $C_{kl}$. From the above calculations of the
cohomology-valued hypergeometric series, we read the (conjectural)
integral, symplectic basis of the period integrals as \[
\Pi(x,y)=\,^{t}(w_{0}(x,y),w_{1}^{(1)}(x,y),w_{2}^{(1)}(x,y),w_{2}^{(2)}(x,y),w_{1}^{(2)}(x,y),w^{(3)}(x,y)).\]
 For notational simplicity, we will understand by\[
\Pi(x,y)=\,^{t}(w_{0}(x,y),w_{k}^{(1)}(x,y),w_{l}^{(2)}(x,y),w^{(3)}(x,y)),\;\;\;(k,l=1,2)\]
 the period integrals arranged in the above order. \bigskip{}

We observed in Proposition \ref{pro:PF-equations} that there appear
two other local structures on $\{(x_{1},y_{1}),\mathcal{A}_{1}\}$
and $\{(x_{2},y_{2}),\mathcal{A}_{2}\}$. It has been noted in \cite{HoTa}
that these local structures correspond to $\tilde{X}_{1}$ and $\tilde{X}_{2}$,
respectively, both of which are smooth complete intersections of $(1,1)$-divisors
and birational to $\tilde{X}_{0}(\not\simeq\tilde{X}_{i},i=1,2)$.
By symmetry, up to the gauge transformations, we have the corresponding
cohomology valued hypergeometric series\[
x_{1}\omega\left(x_{1},y_{1};\frac{J'_{1}}{2\pi i},\frac{J'_{2}}{2\pi i}\right),\;\; x_{2}\omega\left(x_{2},y_{2};\frac{J''_{1}}{2\pi i},\frac{J''_{2}}{2\pi i}\right)\;\]
under the integral, symplectic structures on $H^{even}(\tilde{X}_{1},\mathbb{Q})$
and $H^{even}(\tilde{X}_{2},\mathbb{Q})$, respectively. The definitions
and the calculations of these cohomology valued hypergeometric series
are parallel to (\ref{eq:cohomology-valued-hyp}) with the corresponding
generators $J'_{k}$ and $J''_{k}$. We read the canonical symplectic
form $\mathtt{\Sigma_{0}}$ as above, and the canonical integral,
symplectic basis of the period integrals as\begin{align*}
\Pi'(x_{1,}y_{1}) & =\,^{t}(x_{1}w_{0}(x_{1},y_{1}),x_{1}w_{k}^{(1)}(x_{1},y_{1}),x_{1}w_{l}^{(2)}(x_{1},y_{1}),x_{1}w^{(3)}(x_{1},y_{1})),\\
\Pi''(x_{2},y_{2}) & =\,^{t}(x_{2}w_{0}(x_{2},y_{2}),x_{2}w_{k}^{(1)}(x_{2},y_{2}),x_{2}w_{l}^{(2)}(x_{2},y_{2}),x_{2}w^{(3)}(x_{2},y_{2})).\end{align*}
Note that $\Pi(x,y),\Pi'(x_{1},y_{1})$ and $\Pi''(x_{2},y_{2})$
contain the same unknown constants $C_{kl}$ in common, which will
be fixed later in Proposition \ref{pro:Monodromy-Matrices}.

To make the Taylor expansion of the cohomology valued hypergeometric
series (\ref{eq:cohomology-valued-hyp}), let us introduce the following
notation:\begin{align*}
\partial_{\rho_{k}}w(x,y) & =\frac{\partial\;}{\partial\rho_{k}}w\bigl(x,y;\frac{\rho_{1}}{2\pi i},\frac{\rho_{2}}{2\pi i}\bigr)\bigl|{}_{\rho=0},\\
\partial_{\rho_{k}}\partial_{\rho_{l}}w(x,y) & =\frac{\partial^{2}\;}{\partial\rho_{k}\partial\rho_{l}}w\bigl(x,y;\frac{\rho_{1}}{2\pi i},\frac{\rho_{2}}{2\pi i}\bigr)\bigl|_{\rho=0},\cdots\end{align*}
with formal variables $\rho_{1},\rho_{2}$. Using the intersection
numbers $\int_{\tilde{X}_{0}}J_{1}^{3}=\int_{\tilde{X}_{0}}J_{2}^{3}=5,$
$\int_{\tilde{X}_{0}}J_{1}^{2}J_{3}=\int_{\tilde{X}_{0}}J_{1}J_{2}^{2}=10$,
and also the values $\int_{\tilde{X}_{0}}c_{2}J_{1}=\int_{\tilde{X}_{0}}c_{2}J_{2}=50,$
we have the explicit form of the period integral $\Pi(x,y)$:

\begin{equation}
\Pi(x,y)=\left(\begin{smallmatrix}w_{0}(x,y)\\
\partial_{\rho_{1}}w(x,y)\\
\partial_{\rho_{2}}w(x,y)\\
5\partial_{\rho_{1}}^{2}w+10\partial_{\rho_{1}}\partial_{\rho_{2}}w+\frac{5}{2}\partial_{\rho_{2}}^{2}w+\sum_{b}C_{2b}\partial_{\rho_{b}}w\\
\frac{5}{2}\partial_{\rho_{1}}^{2}w+10\partial_{\rho_{1}}\partial_{\rho_{2}}w+5\partial_{\rho_{2}}^{2}w+\sum_{b}C_{1b}\partial_{\rho_{b}}w\\
-\frac{5}{6}(\partial_{\rho_{1}}^{3}w+\partial_{\rho_{2}}^{3}w)-5(\partial_{\rho_{1}}^{2}\partial_{\rho_{2}}w+\partial_{\rho_{1}}\partial_{\rho_{2}}^{2}w)-\frac{50}{12}(\partial_{\rho_{1}}w+\partial_{\rho_{2}}w)\end{smallmatrix}\right),\label{eq:Pi-xy}\end{equation}
and similar forms for $\Pi'(x_{1},y_{1})$ and $\Pi''(x_{2,}y_{2})$.
In the following calculations, we use the powerseries expansions of
these period integrals to sufficiently higher orders.

\bigskip{}

\subsubsection{Analytic continuations }

Let us consider the analytic continuations of the three isomorphic
local structures noted in Proposition \ref{pro:PF-equations} to the
'center' $[1:1:1]$ of $\mathbb{P}^{2}.$ We introduce a local coordinate
$s=x+1,t=y+1$ of $\mathcal{A}_{0}$ which locates the center $[1:1:1]$
at the origin, and write the Picard-Fuchs differential equations as
$\mathcal{D}_{1}\varphi_{k}(s,t)=\mathcal{D}_{2}\varphi_{k}(s,t)=0$
$(k=0,...,5).$ We arrange the solutions into the column vector \[
\varphi(s,t)=\,^{t}(\varphi_{0}(s,t),\varphi_{1}(s,t),\varphi_{2}(s,t),\varphi_{3}(s,t),\varphi_{4}(s,t),\varphi_{5}(s,t)).\]
Similarly we consider the local solutions satisfying $\mathcal{D}'_{1}\varphi_{k}'(s_{1},t_{1})=\mathcal{D}'_{2}\varphi_{k}'(s_{1},t_{1})=0$
with the local coordinates $s_{1}=x_{1}+1,t_{1}=y_{1}+1$ of $\mathcal{A}_{1}$,
and also $\mathcal{D}''_{1}\varphi_{k}''(s_{2},t_{2})=\mathcal{D}''_{2}\varphi_{k}''(s_{2},t_{2})=0$
with $s_{2}=x_{2}+1,t_{2}=y{}_{2}+1$ of $\mathcal{A}_{2}$. Since
the center $[1:1:1]$ is a regular point of the differential equations
$\mathcal{D}_{1}\varphi_{k}(s,t)=\mathcal{D}_{2}\varphi_{k}(s,t)=0$
(see Proposition \ref{pro:PF-dis0}), we have 6 power series solutions.
After some calculations, we see that the following leading behaviors
determine the local solutions uniquely: \begin{align}
\varphi_{0}(s,t) & =1+c_{11}^{(0)}st+\cdots, & \varphi_{1}(s,t) & =t+c_{11}^{(1)}st+\cdots,\nonumber \\
\varphi_{2}(s,t) & =s+c_{11}^{(2)}st+\cdots, & \varphi_{3}(s,t) & =s^{2}+c_{11}^{(3)}st+\cdots,\label{eq:solutions-leading-behavior}\\
\varphi_{4}(s,t) & =t^{2}+c_{11}^{(4)}st+\cdots, & \varphi_{5}(s,t) & =s^{3}+\cdots,\nonumber \end{align}
where $\cdots$ represent higher order terms (degree $\geq3$) which
do not contain the $s^{3}$-term. Since the differential operators
$\mathcal{D}_{i}'$ and $\mathcal{D}_{i}''$ are related to $\mathcal{D}_{i}$
as in Proposition \ref{pro:PF-equations}, the corresponding local
solutions are simply given by\begin{equation}
\varphi_{i}'(s_{1},t_{1})=(s_{1}-1)\varphi_{i}(s_{1},t_{1})\;,\;\;\varphi_{i}''(s_{2},t_{2})=(s_{2}-1)\varphi_{i}(s_{2},t_{2}).\label{eq:connect-Phis}\end{equation}

\begin{prop}
\label{pro:The-local-solutions-at-111}The three local solutions are
related by \[
\varphi'(s_{1},t_{1})=M_{1}\;\varphi(s,t)\;,\;\;\varphi''(s_{2},t_{2})=M_{2}\;\varphi(s,t)\,,\]

with \begin{align*}
M_{1}=\left(\begin{smallmatrix}-1 & 0 & -1 & -\frac{7}{11} & 0 & -\frac{58}{121}\\
0 & -1 & 1 & 0 & 0 & -\frac{3}{11}\\
0 & 0 & 0 & 1 & 0 & -\frac{6}{11}\\
0 & 0 & 0 & 1 & 0 & \frac{26}{11}\\
0 & 0 & 0 & 1 & -1 & \frac{13}{11}\\
0 & 0 & 0 & 0 & 0 & -1\end{smallmatrix}\right), & M_{2}=\left(\begin{smallmatrix}-1 & -1 & 0 & 0 & -\frac{7}{11} & 0\\
0 & 1 & -1 & 0 & 0 & \frac{3}{11}\\
0 & 1 & 0 & 0 & 0 & -3\\
0 & 0 & 0 & 0 & 1 & \frac{13}{11}\\
0 & 0 & 0 & -1 & 1 & -\frac{13}{11}\\
0 & 0 & 0 & 0 & 0 & -1\end{smallmatrix}\right).\end{align*}
\end{prop}
\begin{proof}
Since $[-x:-y:1]=[1:-y_{1}:-x_{1}]=[-y_{2}:1:-x_{2}]$ by definition,
we have $[1-s:1-t:1]=[1:1-t_{1}:1-s_{1}]=[1-t_{2}:1:1-s_{2}]$ and
\[
s_{1}=\frac{-s}{1-s},\; t_{1}=\frac{t-s}{1-s}\;;\;\; s_{2}=\frac{-t}{1-t},\; t_{2}=\frac{s-t}{1-t}.\]
 Then we should have \[
\varphi'\bigl(\frac{-s}{1-s},\frac{t-s}{1-s}\bigr)=M_{1}\varphi(s,t)\;,\;\;\varphi''\bigl(\frac{-t}{1-t},\frac{s-t}{1-t}\bigr)=M_{2}\varphi(s,t),\]
for $|s|,|t|\ll1$. Using the relations (\ref{eq:connect-Phis}) for
the left hand sides, we obtain the claimed form of the matrices $M_{1},M_{2}$.
\end{proof}
~

~

Now by Proposition \ref{pro:The-local-solutions-at-111}, the connection
problems of the three period integrals $\Pi(x,y),\Pi'(x_{1},y_{1})$,
$\Pi''(x_{2},y_{2})$ to each other may be solved by the analytic
continuations of each to the corresponding local solutions around
the center. By symmetry, we note that connecting $\Pi(x,y)$ to $\varphi(s,t)$
is sufficient for our purpose. 
\begin{prop}
\label{pro:PF-dis0}The singular loci of the Picard-Fuchs differential
equations consist of the three coordinate lines of $\mathbb{P}^{2}$and
an irreducible nodal rational curve of genus 6. The defining equation
of the nodal curve in the affine chart $\{(x,y),\mathcal{A}_{0}\}$
has the following form \[
dis_{0}=(1-x-y)^{5}-5^{4}xy(1-x-y)^{2}+5^{5}xy(xy-x-y).\]
\end{prop}
\begin{proof}
This follows from calculating the characteristic variety of the differential
operators $\mathcal{D}_{1}(x,y)$, $\mathcal{D}_{2}(x,y)$, see \cite[Remark 2.7]{HoTa}.
\end{proof}
Since the irreducible component $dis_{0}=0$ is rational, this can
be parametrized globally by $\mathbb{P}^{1}$. In fact, we can verify
that the equation $dis_{0}=0$ follows from the discriminant $dis(\Xsp|_{(\mathbb{C}^{*})^{4}})=0$
determined in Proposition \ref{pro:dis-X-S-H} eliminating the variables
$a$ and $b$ under the relations $x=-a^{5},y=-b^{5}.$ Hence as a
global parameter of the curve we can adopt an affine line $a+b+1=0$
in $(\mathbb{C}^{*})^{2}$ (which we compactify to $\mathbb{P}^{1}$
with infinity). Using this, we have depicted a schematic picture of
the singular loci in Fig.\ref{fig:discriminant-PF}. In the figure,
the curve $dis_{0}(x,y)=0$ of complex-one dimension is reduced to
the corresponding real curve by imposing a condition ${\rm Im}(x)={\rm Im}(y)$.
The real plane curves drawn in the figure are the projection of the
space curve $\{({\rm Re}(x),{\rm Re}(y),{\rm Im}(x))|\, dis_{0}(x,y)=0\}$
to the first two coordinates. Also, the three affine coordinates are
taken {}``outward direction'' from the standard right-triangular
shape of the moment polytope of $\mathbb{P}^{2}$ whose vertices represent
the three affine coordinates $[a^{5}:b^{5}:1]=[1:\frac{b^{5}}{a^{5}}:\frac{1}{a^{5}}]=[\frac{a^{5}}{b^{5}}:1:\frac{1}{b^{5}}]$.

\begin{figure}
\includegraphics[scale=0.4]{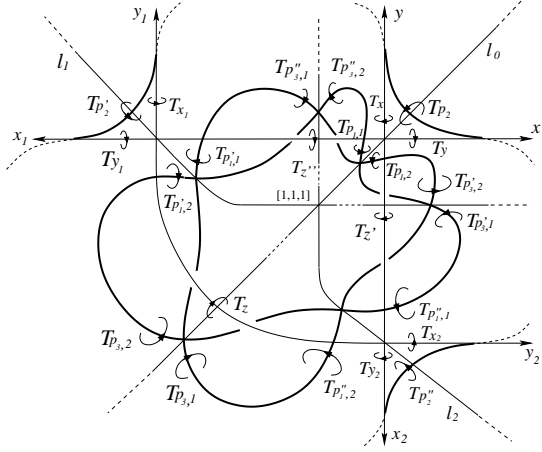}

\caption{\label{fig:discriminant-PF}Singular loci of the Picard-Fuchs differential
equations. Each loop represents the monodromy transformation with
a base point near $(x,y)=(0,0)$ and a path taken over the lines $\ell_{i}$
(${\rm Im}(x)={\rm Im}(y)>0)$. }

\end{figure}

\bigskip{}

\subsubsection{Monodromy transformations shown in Fig.\ref{fig:discriminant-PF}}

As explained above, the defining equation $dis_{0}(x,y)=0$ can be
solved by the line $a+b+1=0$. We set $a=\frac{1}{2}+(\alpha+i\beta),$
$b=\frac{1}{2}-(\alpha+i\beta)$ and solve the additional condition
${\rm Im}(x)={\rm Im}(y)$ $(x=-a^{5},y=-b^{5})$ for $\beta$ to
have the space curve

\[
\{({\rm Re}(x(\alpha)),{\rm Re}(y(\alpha)),{\rm Im}(x(\alpha)))|-\infty<\alpha<+\infty\}.\]
Solving the equation ${\rm Im}(x)={\rm Im}(y)$ for $\beta$ introduces
five branches for the solutions. Each of the solutions determines
a partial parametrization of the curve by $\alpha$. As shown in Fig.\ref{fig:discriminant-PF},
we have two connected components for the real space (plane) curve
in this way. One component comes from the obvious solution $\beta=0$,
and this is represented by the component that consists of 3 solid-bold
(hyperbola-shaped) lines and 3 broken lines. Due to the repetition
of the regions in the coordinate planes, each of the 3 broken lines
should be identified with the solid-bold line in the opposite side.
The other component contains the 6 nodes. It is left to readers to
draw a picture of real Riemannian surface of genus 6 with 6 nodes
whose real hyperplane section is given by the plane curves shown in
Fig.\ref{fig:discriminant-PF}.

In Fig.$\,$\ref{fig:discriminant-PF}, we have also drawn three lines
$\ell_{i}$ :\[
\ell_{0}:(x=y),\;\ell_{1}:(x_{1}=y_{1}),\;\ell_{2}:(x_{2}=y_{2})\]
which intersect at the center $[1,:1:1]$. Each line intersects with
the curve at the two nodal points, and transversally at one point,
as shown in the figure. We name all these points of the intersection
by $p_{1},p_{2},p_{3};p'_{1},p_{2}',p_{3}';p_{1}'',p_{2}'',p_{3}''$
with their explicit coordinates:\begin{alignat*}{2}
p_{1}:[-\rho_{_{-}}:-\rho_{_{-}}:1], & \;\; p_{2}:\bigl[\frac{-1}{32}:\frac{-1}{32}:1\bigr], & \;\; p_{3}:[-\rho_{_{_{+}}}:-\rho_{_{_{+}}}:1],\\
p_{1}':[1:-\rho_{_{-}}:-\rho_{_{-}}], & \;\; p_{2}':\bigl[1:\frac{-1}{32}:\frac{-1}{32}\bigr], & \;\; p_{3}':[1:-\rho_{_{+}}:-\rho_{_{+}}],\\
p_{1}'':[-\rho_{_{-}}:1:-\rho_{_{-}}], & \;\; p_{2}'':\bigl[\frac{-1}{32}:1:\frac{-1}{32}\bigr], & \;\; p_{3}'':[-\rho_{_{+}}:1:-\rho_{_{+}}],\end{alignat*}
where $\rho_{_{\mp}}=\frac{11\mp\sqrt{5}}{2}$. 

For the monodromy calculation of the period integral $\Pi(x,y)$,
we take a base point $\mathbf{o}$ near the origin $(x,y)=(0,0)$.
We fix it to be a real point near $(0,0)$ and $1\gg{\rm Im}(x)>0$
in the figure. Starting this base point, we define the monodromy transformations
$T_{x,}T_{y}$ around the coordinate axes via the loops shown. Similarly
we define monodromy transformations $T_{p_{1},1,}T_{p_{1},2};$$T_{p_{2}};$
$T_{p_{3},1},T_{p_{3},2};$$T_{z}$ by connecting the small loops
shown in the figure with the paths 'over' the line $\ell_{0}$ (a
line near $\ell_{0}$ with ${\rm Im}(x)={\rm Im}(y)>0\text{)}$ from
the base point. We define the monodromy representation, $\rho:\pi_{1}(\mathbb{P}^{2}\setminus\mathcal{D}_{PF},\mathbf{o})\rightarrow Sp(6,\mathbb{Z})$
with $\mathcal{D}_{PF}$ representing the singular loci of the Picard-Fuchs
differential operators and \[
Sp(6,\mathbb{Z})=\{\, P\in GL(6,\mathbb{Z})\,\vert\,\,^{t}P\Sigma_{0}P=\Sigma_{0}\,\}\]
with respect to the symplectic form $\Sigma_{0}$ in (\ref{eq:symplectic-form}).
We adopt the convention that, for example, $T_{x}.\Pi(x,y)=\rho(T_{x})\Pi(x,y)$
represents the analytic continuation $T_{x}.\Pi(x,y)$ of the local
solution $\Pi(x,y)$ along the path with the loop $T_{x}$ in terms
of the local solution $\Pi(x,y)$. Thus under our convention, the
monodromy representation $\rho$ is an anti-homomorphism satisfying
$\rho(T_{1}T_{2})=\rho(T_{2})\rho(T_{1})$. 
\begin{prop}
\label{pro:Monodromy-Matrices}When we take $C_{11}=C_{22}=-\frac{1}{2},C_{12}=C_{21}=0$
in the canonical integral, symplectic basis $\Pi(x,y)$ in (\ref{eq:Pi-xy}),
all the monodromy transformations above are represented by the elements
$\rho(T_{*})$ in $Sp(6,\mathbb{Z})$. Explicitly, the corresponding
monodromy matrices acting on the period integral $\Pi(x,y)$ are given
by: \begin{gather*}
T_{x}:\left(\begin{smallmatrix}1 & 0 & 0 & 0 & 0 & 0\\
1 & 1 & 0 & 0 & 0 & 0\\
0 & 0 & 1 & 0 & 0 & 0\\
5 & 10 & 10 & 1 & 0 & 0\\
2 & 5 & 10 & 0 & 1 & 0\\
\ms5 & \ms3 & \ms5 & 0 & \ms1 & 1\end{smallmatrix}\right),T_{y}:\left(\begin{smallmatrix}1 & 0 & 0 & 0 & 0 & 0\\
1 & 1 & 0 & 0 & 0 & 0\\
1 & 0 & 1 & 0 & 0 & 0\\
2 & 10 & 5 & 1 & 0 & 0\\
5 & 10 & 10 & 0 & 1 & 0\\
\ms5 & \ms5 & \ms3 & \ms1 & 0 & 1\end{smallmatrix}\right),T_{z}:\left(\begin{smallmatrix}41 & \ms17 & \ms17 & 6 & 6 & 15\\
4 & 0 & \ms6 & \ms2 & 3 & 1\\
4 & \ms6 & 0 & 3 & \ms2 & 1\\
\ms72 & 28 & 23 & \ms13 & \ms9 & \ms28\\
\ms72 & 23 & 28 & \ms9 & \ms13 & \ms28\\
\ms30 & 18 & 18 & \ms4 & \ms4 & \ms9\end{smallmatrix}\right),\\
T_{p_{1},1}:\left(\begin{smallmatrix}\ms4 & 5 & 2 & \ms1 & 0 & \ms1\\
0 & 1 & 0 & 0 & 0 & 0\\
\ms5 & 5 & 3 & \ms1 & 0 & \ms1\\
\ms10 & 10 & 4 & \ms1 & 0 & \ms2\\
\ms25 & 25 & 10 & \ms5 & 1 & \ms5\\
25 & \ms25 & \ms10 & 5 & 0 & 6\end{smallmatrix}\right),T_{p_{1},2}:\left(\begin{smallmatrix}6 & \ms2 & \ms5 & 0 & 1 & 1\\
5 & \ms1 & \ms5 & 0 & 1 & 1\\
0 & 0 & 1 & 0 & 0 & 0\\
25 & \ms10 & \ms25 & 1 & 5 & 5\\
10 & \ms4 & \ms10 & 0 & 3 & 2\\
\ms25 & 10 & 25 & 0 & \ms5 & \ms4\end{smallmatrix}\right),T_{p_{2}}:\left(\begin{smallmatrix}1 & 0 & 0 & 0 & 0 & 1\\
0 & 1 & 0 & 0 & 0 & 0\\
0 & 0 & 1 & 0 & 0 & 0\\
0 & 0 & 0 & 1 & 0 & 0\\
0 & 0 & 0 & 0 & 1 & 0\\
0 & 0 & 0 & 0 & 0 & 1\end{smallmatrix}\right),\\
T_{p_{3},1}:\left(\begin{smallmatrix}41 & \ms4 & \ms20 & 0 & 12 & 16\\
30 & \ms2 & \ms15 & 0 & 9 & 12\\
0 & 0 & 1 & 0 & 0 & 0\\
50 & \ms5 & \ms25 & 1 & 15 & 20\\
10 & \ms1 & \ms5 & 0 & 4 & 4\\
\ms100 & 10 & 50 & 0 & \ms30 & \ms39\end{smallmatrix}\right),T_{p_{3},2}:\left(\begin{smallmatrix}\ms39 & 20 & 4 & \ms12 & 0 & \ms16\\
0 & 1 & 0 & 0 & 0 & 0\\
\ms30 & 15 & 4 & \ms9 & 0 & \ms12\\
\ms10 & 5 & 1 & \ms2 & 0 & \ms4\\
\ms50 & 25 & 5 & \ms15 & 1 & \ms20\\
100 & \ms50 & \ms10 & 30 & 0 & 41\end{smallmatrix}\right).\end{gather*}
\end{prop}
\begin{proof}
Our proof is based on numerical calculations except for $T_{x}$ and
$T_{y}$. To have the matrix of $T_{p_{1},1}$, for example, we generate
the power series for $\Pi(x,y)$ in (\ref{eq:Pi-xy}) up to total
degree 60. From the base point to a small loop for $T_{p_{1},1}$,
we may take a path over the line $\ell_{0}$, i.e., a real line near
$\ell_{0}$ with ${\rm Im}(x)={\rm Im}(y)>0.$ This choice of path,
however, is not efficient to attain numerically high accuracy due
to the 'degeneration' of the period integrals which we see in $w_{1}^{(k)}(x,y)=w_{2}^{(k)}(x,y)$
when $x=y$. To avoid this degeneration, we deform the path satisfying
${\rm Im}(x)={\rm Im}(y)$ to that satisfying ${\rm Im}(y)=0$ by
making use of the homotopy $\varepsilon{\rm Im}(x)={\rm Im}(y),\varepsilon\in[0,1]$.
Thereby, we verify that the path does not intersect the singular loci
$\mathcal{D}_{PF}$ at any $\varepsilon\in[0,1]$. The path for our
actual calculation is a path over the $\ell_{0}$ satisfying ${\rm Im}(y)=0$.
We divide the deformed line into 200 segments and also the small loop
into 100 arcs. Then, for each endpoint of them, we have constructed
the local solutions imposing the same leading behavior in (\ref{eq:solutions-leading-behavior}).
The monodromy matrix, by definition, follows by relating these solutions
at each ends along the path. We obtained the claimed integral, symplectic
matrix for $T_{p_{1},1}$ in the accuracy $10^{-5}\sim10^{-6}.$ Other
monodromies are determined in the same way with the same level of
accuracy in their numerical calculations.
\end{proof}
We now consider the analytic continuation of the local solution $\Pi(x,y)$
from the base point to the center $[1,:1:1]$ along (over) the line
$\ell_{0},$ and further continue to a point near $(x_{1},y_{1})=(0,0)$
along (over) the line $\ell_{1}$. We express the local solution $\Pi'(x_{1,}y_{1})$
in terms of the analytically continued solution $\Pi(x,y)$ by $\Pi'(x_{1},y_{1})=\mathcal{C}_{10}\Pi(x,y)$.
In a similar way, we consider the analytic continuation of $\Pi(x,y)$
along the line $\ell_{0}$ followed by $\ell_{2}$, and define the
relation $\Pi''(x_{2},y_{2})=\mathcal{C}_{20}\Pi(x,y)$. 
\begin{prop}
\label{pro:connection-matrix-C}The above relations $\Pi'(x_{1},y_{1})=\mathcal{C}_{10}\Pi(x,y)$
and $\Pi''(x_{2},y_{2})=\mathcal{C}_{20}\Pi(x,y)$ are solved by 

\begin{align*}
\mathcal{C}_{10}=\left(\begin{matrix}\ms4 & 8 & 4 & \ms2 & 1 & 0\\
\ms4 & 4 & 2 & \ms1 & 0 & \ms1\\
3 & 2 & 1 & 0 & 1 & 2\\
6 & 4 & 0 & 1 & 2 & 4\\
17 & 0 & \ms4 & 2 & 4 & 8\\
0 & \ms17 & \ms6 & 3 & \ms4 & \ms4\end{matrix}\right),\;\; & \mathcal{C}_{20}=\left(\begin{matrix}4 & \ms4 & \ms8 & \ms1 & 2 & 0\\
4 & \ms2 & \ms4 & 0 & 1 & 1\\
\ms3 & \ms1 & \ms2 & \ms1 & 0 & \ms2\\
\ms6 & 0 & \ms4 & \ms2 & \ms1 & \ms4\\
\ms17 & 4 & 0 & \ms4 & \ms2 & \ms8\\
0 & 6 & 17 & 4 & \ms3 & 4\end{matrix}\right).\end{align*}
\end{prop}
\begin{proof}
As in the previous proposition, we do numerically the analytic continuation
of $\Pi(x,y)$ to $\varphi(s,t)$ along $\ell_{0}$, and $\Pi'(x_{1},y_{1})$
to $\varphi'(s_{1},t_{1})$ along $\ell_{1}$. Then use Proposition
\ref{pro:The-local-solutions-at-111} to relate $\varphi(s,t)$ and
$\varphi'(s_{1},t_{1})$, and obtain the claimed matrix $\mathcal{C}_{10}$.
The matrix $\mathcal{C}_{20}$ follows in the same way. 
\end{proof}
Since the three local forms of the period integral $\Pi(x,y),\Pi'(x_{1},y_{1})$
and $\Pi''(x_{2},y_{2})$ are governed by the isomorphic system of
differential equations (see Proposition \ref{pro:PF-equations}) and
also from the obvious symmetry in Fig.\ref{fig:discriminant-PF},
the entire monodromy properties of the period integral $\Pi(x,y)$
can be described by the monodromy transformations 

\[
\{\xi_{m}\}:=\{T_{x,}T_{y},T_{p_{1},1,}T_{p_{1},2},T_{p_{2}},T_{p_{3},1},T_{p_{3},2},T_{z}\},\]
or the corresponding transformations: \begin{align*}
\{\xi'_{m}\} & :=\{T_{x_{1}},T_{y_{1}},T_{p'_{1},1,}T_{p_{'1},2},T_{p'_{2}},T_{p'_{3},1},T_{p'_{3},2},T_{z'}\},\text{ or}\\
\{\xi''_{m}\} & :=\{T_{x_{2}},T_{y_{2}},T_{p''_{1},1,}T_{p''_{1},2},T_{p''_{2}},T_{p''_{3},1},T_{p''_{3},2},T_{z''}\}.\end{align*}

\begin{prop}
\label{pro:Monodromy-relations-1)-3)}1) For the monodromy matrices
we have \[
\rho(\xi'_{m})=\mathcal{C}_{10}^{-1}\rho(\xi_{m})\mathcal{C}_{10},\;\;\rho(\xi''_{m})=\mathcal{C}_{20}^{-1}\rho(\xi_{m})\mathcal{C}_{20}.\]

\noindent 2) The following relations can be observed: \begin{align*}
\rho(T_{p_{1},1}) & =\rho(T_{y}^{-1}T_{p_{2}}^{-1}T_{y})\;\;,\;\;\rho(T_{p_{1},2})=\rho(T_{x}^{-1}T_{p_{2}}T_{x})\\
\rho(T_{p_{3},1}) & =\mathcal{C}_{10}\,\rho(T_{x}T_{y}^{-1}T_{p_{2}}T_{y}T_{x}^{-1})\,\mathcal{C}_{10}^{-1},\\
\rho(T_{p_{3},2}) & =\mathcal{C}_{20}\,\rho(T_{x}T_{y}^{-1}T_{p_{2}}^{-1}T_{y}T_{x}^{-1})\,\mathcal{C}_{20}^{-1}.\end{align*}

\noindent3) We have $\rho(T_{z})=\rho(T_{p_{1}',2}^{-1}T_{x_{1}}T_{p_{1}',2})=\rho(T_{p_{1}'',2}^{-1}T_{x_{2}}T_{p_{1}'',2})$.

\noindent4) The image of the monodromy transformations in $Sp(6,\mathbb{Z})$
is given by \[
\langle\rho(T_{x}^{\pm1}),\rho(T_{y}^{\pm1}),\rho(T_{p_{2}}^{\pm1}),\mathcal{C}_{10}^{\pm1},\mathcal{C}_{20}^{\pm1}\rangle.\]
\end{prop}
\begin{proof}
1) By the symmetry summarized in Proposition \ref{pro:PF-equations}
and the definitions of $\mathcal{C}_{10}$ and $\mathcal{C}_{20}$,
the first claim follows. For 2), we verify directly the claimed relations
using the monodromy matrices in Proposition \ref{pro:Monodromy-Matrices}.
When doing this, we should note that $\rho$ is defined as an anti-homomorphism,
$\rho(T_{\alpha}T_{\beta})=\rho(T_{\beta})\rho(T_{\alpha})$. Using
the results 1) and 2), we verify the relation 3). We can also verify
3) by deforming the contours of the analytic continuations (see Fig.$\,$\ref{fig:discriminant-PF}).
The property 4) follows from 1) to 3).\end{proof}
\begin{rem*}
In the claim 3) of Proposition \ref{pro:Monodromy-relations-1)-3)},
not all monodromy relations which we read from Fig.$\,$\ref{fig:discriminant-PF}
are written out. By deforming the paths in the figure, it is easy
to deduce relations, for example:\[
\rho(T_{p_{1},1}T_{y}T_{p_{1},1}^{-1})=\rho(T_{z''}),\;\rho(T_{z}^{-1}T_{p_{3},1}T_{z})=\rho(T_{p_{1}',1}^{-1}),\;\rho(T_{z''}^{-1}T_{p_{3}'',2}T_{z''})=\rho(T_{p_{3,2}}).\]
We can also observe relations among the generators in the claim 4),
for example, \[
\rho(T_{p_{1},1})\mathcal{C}_{10}\mathcal{C}_{20}\rho(T_{p_{1},2})=\mathcal{C}_{10}\mathcal{C}_{20}.\]
The determination of the minimal set of relations is left for a future
study. Also some simplifications in the matrix expressions, like $\mathcal{C}_{10}\mathcal{C}_{20}=(\ms1)\oplus\left(\begin{smallmatrix}0 & \ms1\\
\ms1 & 0\end{smallmatrix}\right)\oplus\left(\begin{smallmatrix}0 & \ms1\\
\ms1 & 0\end{smallmatrix}\right)\oplus(\ms1)$, may have some interpretations. \hfill  {[}{]}
\end{rem*}
\vspace{1cm}

\subsection{Mirror symmetry of Reye congruences}

Over the line $\ell_{0}:x=y=0$ the six period integrals contained
in $\Pi(x,y)$ reduce to four independent integrals due to the degeneration
$w_{1}^{(k)}(x,x)=w_{2}^{(k)}(x,x)$ $(k=1,2)$. This is related to
the symmetry under the exchange $z_{i}\leftrightarrow w_{i}$ of the
defining equations $f_{i}=0$ of $\Xsp$ when $a=b$. More generally,
taking the automorphisms of $\Xsp$ into account, this symmetry appears
when $a=\mu^{k}b$ with $\mu^{5}=1$, i.e., when $x=y$. 
\begin{prop}
\label{pro:mirror-Reye-X}When $x=y\not=\frac{1}{32}$, the involution
$z_{i}\leftrightarrow w_{i}$($\cong\mathbb{Z}_{2}$) acting on $\Xsp$
has no fixed point. This action naturally lifts to a fixed point free
$\mathbb{Z}_{2}$ action on the crepant resolution $\cXsp$. Taking
a quotient by this, we obtain a Calabi-Yau threefold $X^{*}=\cXsp/\mathbb{Z}_{2}$
with the Hodge numbers $h^{1,1}(X^{*})=26,\, h^{2,1}(X^{*})=1.$\end{prop}
\begin{proof}
As above, it is clear from the form of the defining equations that
the involution $z_{i}\leftrightarrow w_{i}$ acts on $\Xsp$ when
$a=b$ ($a=\mu^{k}b$ in general). It is also straightforward to see
that if $a=b\not=-\frac{1}{2}$, there is no solution for $f_{i}=0$
with $z_{i}=w_{i}$ except $z_{i}=w_{i}=0\,(i=1,...,5)$. Clearly
the involution acts on the singular loci. Hence it lifts to a fixed
point free $\mathbb{Z}_{2}$ action on the crepant resolution $\cXsp$
when $x=y\not=\frac{1}{32}$. Since $h^{0,1}(\cXsp)=h^{0,2}(\cXsp)=0$
for the resolution, we have $h^{0,1}(X^{*})=h^{0,2}(X^{*})=0$ for
the quotient. The calculations $e(X^{*})=e(\cXsp)/2=50$ and $\#X^{*}(\mathbb{F}_{p})=\#\cXsp(\mathbb{F}_{p})/2$
are valid for the free quotient. Hence the proof of Proposition \ref{pro:mod-p-H2}
applies to the present case, and we have $h^{1,1}(X^{*})=26,h^{2,1}(X^{*})=1$.
\end{proof}
In \cite[Propositions 2.9,2.10]{HoTa}, we observed that, when $x=y=\frac{1}{32}$,
one ordinary double point appears in $\Xsp$ as a fixed point of the
involution, and this results in a singular point of $X^{*}$ where
a lens space ($\cong S_{3}/\mathbb{Z}_{2}$) vanishes. In fact, this
property has been predicted by noting a specific form of the Picard-Lefschetz
monodromy \cite{Enckv-Straten} in their study of 4th order differential
equations (see also \cite{AlmkcistEtAl}). In order to connect the
vanishing lens space directly with the Picard-Lefschetz monodromy,
let us introduce the following monodromy matrices:\begin{equation}
\begin{matrix}R_{\alpha_{1}}=\rho(T_{p_{1},1}^{-1}T_{p_{1},2}),\; R_{0}=\rho(T_{x}T_{y}),\; R_{\frac{1}{32}}=\rho(T_{p_{2}}),\;\\
R_{\alpha_{2}}=\rho(T_{p_{3},2}^{-1}T_{p_{3},1}),\; R_{\infty}=\rho(T_{z}).\end{matrix}\label{eq:monod-matrix-R}\end{equation}
 As we see in Fig.\ref{fig:discriminant-PF}, these represent the
monodromy transformations of $\Pi(x,y)$ around the intersections
of $\ell_{0}\cong\mathbb{P}^{1}$ with the discriminant, and satisfy
a relation\[
R_{\infty}R_{\alpha_{2}}R_{\alpha_{1}}R_{0}R_{\frac{1}{32}}=id.\]
These correspond to the matrices $M_{\alpha_{1}},M_{0},M_{\frac{1}{32}},M_{\alpha_{2}},M_{\infty}$
of $\Pi(x)$ given in \cite[Table 1]{HoTa}. Explicitly we evaluate
the matrices (\ref{eq:monod-matrix-R}) as follows:

\begin{gather*}
R_{\alpha_{1}}:\left(\begin{smallmatrix}11 & \ms7 & \ms7 & 1 & 1 & 2\\
5 & \ms1 & \ms5 & 0 & 1 & 1\\
5 & \ms5 & \ms1 & 1 & 0 & 1\\
35 & \ms20 & \ms29 & 3 & 5 & 7\\
35 & \ms29 & \ms20 & 5 & 3 & 7\\
\ms50 & 35 & 35 & \ms5 & \ms5 & \ms9\end{smallmatrix}\right),\; R_{0}:\left(\begin{smallmatrix}1 & 0 & 0 & 0 & 0 & 0\\
1 & 1 & 0 & 0 & 0 & 0\\
1 & 0 & 1 & 0 & 0 & 0\\
17 & 20 & 15 & 1 & 0 & 0\\
17 & 15 & 20 & 0 & 1 & 0\\
\ms20 & \ms18 & \ms18 & \ms1 & \ms1 & 1\end{smallmatrix}\right),\; R_{\frac{1}{32}}:\left(\begin{smallmatrix}1 & 0 & 0 & 0 & 0 & 1\\
0 & 1 & 0 & 0 & 0 & 0\\
0 & 0 & 1 & 0 & 0 & 0\\
0 & 0 & 0 & 1 & 0 & 0\\
0 & 0 & 0 & 0 & 1 & 0\\
0 & 0 & 0 & 0 & 0 & 1\end{smallmatrix}\right),\\
R_{\infty}R_{\alpha_{2}}R_{\infty}^{-1}:\left(\begin{smallmatrix}1 & 5 & 5 & 0 & 0 & 2\\
0 & 2 & \ms1 & \ms2 & 2 & 0\\
0 & \ms1 & 2 & 2 & \ms2 & 0\\
0 & \ms12 & \ms13 & 0 & 1 & \ms5\\
0 & \ms13 & \ms12 & 1 & 0 & \ms5\\
0 & 0 & 0 & 0 & 0 & 1\end{smallmatrix}\right),\; R_{\infty}:\left(\begin{smallmatrix}41 & \ms17 & \ms17 & 6 & 6 & 15\\
4 & 0 & \ms6 & \ms2 & 3 & 1\\
4 & \ms6 & 0 & 3 & \ms2 & 1\\
\ms72 & 28 & 23 & \ms13 & \ms9 & \ms28\\
\ms72 & 23 & 28 & \ms9 & \ms13 & \ms28\\
\ms30 & 18 & 18 & \ms4 & \ms4 & \ms9\end{smallmatrix}\right),\end{gather*}
where we consider $R_{\infty}R_{\alpha_{2}}R_{\infty}^{-1}$ instead
of $R_{\alpha_{2}}$ since the matrices in \cite[Table 1]{HoTa} satisfy
$M_{\alpha_{2}}M_{\infty}M_{\alpha_{1}}M_{0}M_{\frac{1}{32}}=id$.
Now we define $\tilde{\Pi}(x,y)$ by \[
\,^{t}\Bigl(w_{0},\frac{1}{2}(w_{1}^{(1)}+w_{2}^{(1)}),\frac{1}{2}(w_{1}^{(2)}+w_{2}^{(2)}),\frac{1}{2}w^{(3)}\Bigr)\oplus\,^{t}\Bigl(\frac{1}{2}(w_{1}^{(1)}-w_{1}^{(1)}),\frac{1}{2}(w_{1}^{(2)}-w_{2}^{(2)})\Bigr)\]
 so that the second summand becomes $\,^{t}(0,0)$ on the line $\ell_{0}$($x=y$).
It is straightforward to see the following property:
\begin{prop}
\label{pro:R-M-decomposition}In terms of the period integral $\tilde{\Pi}(x,y)=\mathcal{P}\,\Pi(x,y)$,
we have the decomposition \[
\tilde{R}_{k}=M_{k}\oplus N_{k}\;\;(k=1,..,5),\]
where $N_{k}$'s are $2\times2$ matrices and we set $\tilde{R}_{k}=\mathcal{P}R_{k}\mathcal{P}^{-1}$
with $\{R_{k}\}_{k=1}^{4}=\{R_{\alpha_{1}},R_{0},$ $R_{\frac{1}{32}},R_{\infty}R_{\alpha_{2}}R_{\infty}^{-1},R_{\infty}\}$
and $\{M_{k}\}_{k=1}^{4}=\{M_{\alpha_{1}},M_{0},M_{\frac{1}{32}},M_{\alpha_{2}},M_{\infty}\}$.
\end{prop}
From the explicit forms of (\ref{eq:monod-matrix-R}), it is clear
that $R_{\frac{1}{32}}$ represents the Picard-Lefschetz monodromy
of the vanishing cycle which appears in the fiber over $x=y=\frac{1}{32}$,
from which we identify $w^{(3)}(x,y)$ as the period integral of the
vanishing cycle. We note that $w^{(3)}(x,y)$ is contained in $\tilde{\Pi}(x,y)$
with the prefactor $\frac{1}{2}$. (If this prefactor were taken to
be 1, $\mathcal{P}$ should be symplectic with respect to $\Sigma_{0}$.)
From the above proposition and Proposition \ref{pro:mirror-Reye-X},
we can now identify $M_{\frac{1}{32}}$ with the Picard-Lefschetz
monodromy of the vanishing lens space ($S_{3}/\mathbb{Z}_{2}$) in
$X^{*}$ for $x=y=\frac{1}{32}$ which we described above.

Finally we remark that both $R_{0}$ and $R_{\infty}$ have the same
Jordan normal form $J(1,4)\oplus J(1,2)$ with eigenvalues $1$. Proposition
\ref{pro:R-M-decomposition} implies that the period integral $\tilde{\Pi}(x,y)$
is compatible with the Jordan decomposition and the first summand
of $\tilde{\Pi}(x,y)$ shows the maximally unipotent monodromies both
at $x=0$ and $\infty$. The mirror geometry which arises from $x=0$
has been identified with the Reye congruence Calabi-Yau threefold,
and that from $x=\infty$ has been identified with a new Calabi-Yau
manifold which doubly covers the generic Hessian quintic with ramification
locus being a smooth curve of genus 26 and degree 20. 

\vspace{2cm}

\vfill \;

\pagebreak{}

\section{\textbf{\textup{\label{sec:Xs-and-Ys-discussions}Special families
of Steinerian and Hessian quintics}}}

\subsection{Steinerian and Hessian quintics. }

Here we discuss the special family of the Steinerian and Hessian quintics
defined by (\ref{eq:defEqS}) and (\ref{eq:def-eq-HessAB}), respectively,
for $a=b$. Together with the mirror family $\mathfrak{X}_{\mathbb{P}^{1}}^{*}$,
we summarize the related families over $\mathbb{P}^{1}$ by writing
the generic fibers$:$ \begin{equation}
\begin{matrix}\xymatrix{\tilde{X}_{0}^{*}\ar[r]_{\varphi}\ar[d]^{/\mathbb{Z}_{2}} & \Xsp\ar[dr] &  & \ar[dl]U_{sp}\ar[dr] & Y_{sp}\ar[d]^{2:1} & \ar[l]^{\rho}Y^{*}\\
X^{*} &  & S_{sp} &  & H_{sp}}
\end{matrix}.\label{eq:Xs-Ys-diagram}\end{equation}
The Steinerian quintic $S_{sp}$ is defined as the determinantal quintics
$\Sspp=\Ssp$ for $a=b$. $U_{sp}$ in the diagram provides a partial
resolutions of $S_{sp}$, and is given as $\Usp$ for $a=b.$ There
is a natural projection from $U_{sp}$ to the Hessian quintic $H_{sp}$,
i.e., the determinantal quintic $\Hsp$ for $a=b$. 

In what follows, we describe the singularity of the Hessian quintic
$H_{sp}$ for generic $a=b$, i.e., $a=b\in\mathbb{C}^{*}$ with $\prod_{k,l=0}^{4}(\mu^{k}\, a+\mu^{l}\, b+1)\not=0$.
Then we define the double covering $Y_{sp}\rightarrow H_{sp}$ branched
along the singular loci of $H_{sp}$. It is expected that there is
a crepant resolution $\rho:Y_{sp}^{*}\rightarrow Y_{sp}$.

\subsection{Singular loci of $H_{sp}$}

The Hessian quintic $H_{sp}$ is defined in $\mathbb{P}_{\lambda}^{4}$
by the equation (\ref{eq:def-eq-HessAB}) with $a=b\in\mathbb{C}^{*}$,
which may be written $\det A_{\lambda}=0$. We note that this is actually
defined for $x=-a^{5}\in\mathbb{P}^{1}$ due to the automorphism $H_{sp}\subset\mathbb{P}_{\lambda}^{4}$.
Since $H_{sp}$ is a hypersurface, it is rather easy to determine
the singular loci by the Jacobian criterion. To describe the singular
loci, let us denote the projective space by $\mathbb{P}_{\lambda}^{4}=\langle e_{1}^{*},e_{2}^{*},...,e_{5}^{*}\rangle$
choosing a $\mathbb{C}$-bases $e_{i}^{*}(i=1,...,5)$. As before,
we denote the coordinate points and lines, respectively, by $[e_{i}^{*}]$
and $L_{ij}=\langle e_{i}^{*},e_{j}^{*}\rangle$. We also introduce
the lines:\[
M_{i}=\langle e_{i-2}^{*}+ab\, e_{i-1}^{*},e_{i}^{*}\rangle\;(i=1,2,...,5),\]
where the indices are considered cyclic or modulo 5. 
\begin{prop}
For generic $a=b$, we have: 1) the singular loci of $H_{sp}$ contain
a component of a curve $C_{E}$ given by the $4\times4$ Pfaffians
of \begin{equation}
\left(\begin{array}{ccccc}
0 & \lambda_{2} & -a\lambda_{5} & -a\lambda_{1} & \lambda_{5}\\
-\lambda_{2} & 0 & \lambda_{4} & -a\lambda_{2} & -a\lambda_{3}\\
a\lambda_{5} & -\lambda_{4} & 0 & \lambda_{1} & -a\lambda_{4}\\
a\lambda_{1} & a\lambda_{2} & -\lambda_{1} & 0 & \lambda_{3}\\
-\lambda_{5} & a\lambda_{3} & a\lambda_{4} & -\lambda_{3} & 0\end{array}\right).\label{eq:elliptic-normal-q}\end{equation}
2) $C_{E}$ is a smooth genus one curve of degree 5 in $\mathbb{P}_{\lambda}^{5}$
, i.e., an elliptic normal quintic. \end{prop}
\begin{proof}
Our proof of 1) is based on the calculations of the primary decomposition
of the Jacobian ideal. As we describe below, there appear 10 lines
in the irreducible components of the singular loci. It is efficient
to take the saturations repeatedly with respect to the ideals representing
these lines. The claimed matrix form of the ideal may be deduced from
the minimal resolution of the ideal, and has been determined by taking
suitable linear combinations of the generators.

The claim 2) is a consequence of 1), since $C_{E}$ is a Pfaffian
variety of $5\times5$ anti-symmetric matrices which may be identified
as a linear section of the Grassmannian $G(2,5)$. The smootheness
is verified by the Jacobian criterion. \end{proof}
\begin{rem}
When $a^{10}+11a^{5}-1=0$, $C_{E}$ becomes nodal at one point (at
$[1:1:1:1:1]$ up to automorphisms of $H_{sp}\subset\mathbb{P}_{\lambda}^{4}$
for every solution $a$). When $a=0$ (resp. $\infty$), $C_{E}$
becomes reducible: $C_{E}=L_{13}\cup L_{35}\cup L_{52}\cup L_{24}\cup L_{41}$
(reps. $C_{E}=L_{12}\cup L_{23}\cup L_{34}\cup L_{45}\cup L_{51}$).
See \cite{Huleck} for the geometry of elliptic normal quintics. \hfill{[}{]}

\begin{figure}
\includegraphics[scale=0.6]{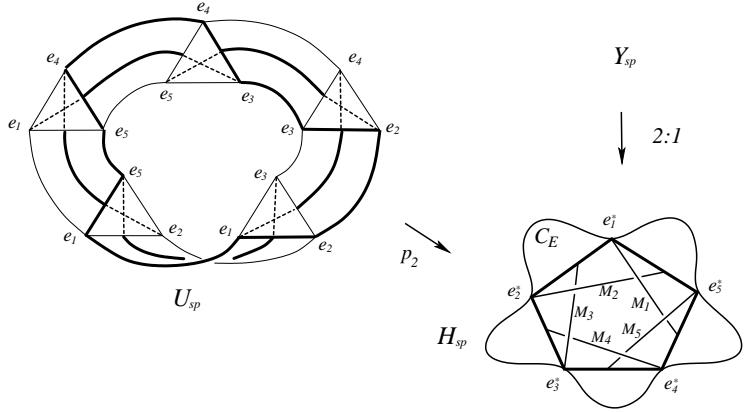}\caption{\label{fig:U-Hessian}The partial resolution $U_{sp}\rightarrow H_{sp}$
of the Hessian quintic and the double covering $Y_{sp}\rightarrow H_{sp}$
for a generic $a=b\in\mathbb{C}^{*}$. $C_{E}$ is an elliptic normal
quintic. By the partial resolution, the $A_{1}$ singularities along
the lines $M_{i}(i=1,..,5)$ and $C_{E}$ are resolved in $U_{sp}$.
The $A_{3}$ singularities along the coordinate lines $L_{i\, i+1}$
are blown up to $A_{1}$ singularities along two lines for each $L_{i\, i+1}$.
In each $p_{2}^{-1}([e_{i}^{*}])\cong\mathbb{P}^{2}$, there exist
$A_{1}$-singularities along the broken lines and the coordinate line
$\langle e_{i+2},e_{i+3}\rangle$. }

\end{figure}
\vspace{0.3cm}
By studying the Jacobian ideal of $H_{sp}$ in details, we obtain
the structure of the singularities in the special Hessian quintic
for generic $a=b$ as follows:\end{rem}
\begin{prop}
\label{pro:singular-loci-Hsp}1) For generic $a=b\in\mathbb{C}^{*}$,
the special Hessian quintic $H_{sp}$ is singular along the 5 coordinate
lines $L_{i\, i+1}$ and 5 lines $M_{i}$ (i=1,..,5) and also the
smooth elliptic normal quintic $C_{E}$. The type of singularities
are of $A_{3}$-type along the lines $L_{i\, i+1}$ and of $A_{1}$-type
along the lines $M_{i}$ and $C_{E}$. These irreducible components
intersect at 10 points as shown schematically in Fig.$\,\ref{fig:U-Hessian}$.

\noindent 2) The singular loci of $H_{sp}$ coincide set-theoretically
with $\{[\lambda]\in\mathbb{P}_{\lambda}^{4}\mid rk(A_{\lambda})\leq3\}$.\end{prop}
\begin{proof}
1) Singular loci are determined by studying the Jacobian ideal as
described in the proof of the previous proposition. The type of the
singularities are determined by taking the local coordinates of the
normal bundle at generic points of the irreducible components. 2)
The loci of $rk(A_{\lambda})\leq3$ are determined by the ideal generated
by the $4\times4$ minors of $A_{\lambda}$. We compare the primary
decompositions of this ideal with that of the Jacobian ideal. The
claim follows since we verify that the radicals of each component
coincide. \end{proof}
\begin{rem}
\label{rem:rem-U-X2-Hab}The Hessian quintic $H_{sp}$ is the special
quintic hypersurface $\Hsp$ with $a=b$. If $a\not=b$ (more generally
$a\not=\mu^{k}b$ with $\mu^{5}=1)$, it is easy to observe that the
irreducible component $C_{E}$ disappears from the singular loci.
This explains the additional factor $\prod_{k=0}^{4}(a-\mu^{k}b)$
in the discriminant (\ref{eq:discriminant-Hessian-quintic}). We note
that the $A_{1}$-singularities in $H_{sp}$ (resp. $\Hsp$) are resolved
by the partial resolution $p_{2}:U_{sp}\rightarrow H_{sp}$ (resp.
$p_{2}:\Usp\rightarrow\Hsp$). The configuration of the singularities
in $U_{sp}$ is depicted in Fig.$\,\ref{fig:U-Hessian}$. As depicted
in the figure, there appear the following 25 lines along which $U_{sp}$
has $A_{1}$-singularities: Three lines in each fiber $p_{2}^{-1}([e_{i}^{*}])=\langle e_{i+2}^{*},e_{i+3}^{*},e_{i+4}^{*}\rangle$
given by \[
\langle e_{i+2},-be_{i+3}+e_{i+4}\rangle,\;\langle-be_{i+2}+e_{i+3},e_{i+4}\rangle,\;\langle e_{i+2},e_{i+3}\rangle\;(i=1,...,5),\]
and two lines in each inverse image $p_{2}^{-1}(L_{i\, i+1})\,(i=1,...,5)$
given by \[
[e_{i+3}]\times L_{i\, i+1},\;[-b\, e_{i+3}+e_{i+4}]\times L_{i\, i+1}\;(i=1,...,5).\]
Since determining these singular loci is essentially the same as we
did for $\Xsp$ in Proposition \ref{pro:Singular-tilde-Xosp-A1s},
we omit the details. Inspecting the configuration shown in Fig.$\,\ref{fig:U-Hessian}$,
it is immediate to have the Euler number of $U_{sp}$ by using $e(H_{sp})$.
$\qquad[]$\end{rem}
\begin{prop}
The Euler number $e(U_{sp})=0$. \end{prop}
\begin{proof}
We evaluate $e(U_{sp})$ from the projection $p_{2}:U_{sp}\rightarrow H_{sp}$
shown in Fig.$\,\ref{fig:U-Hessian}$. We note that $e(p_{2}^{-1}([e_{i}^{*}])=e(\mathbb{P}^{2})=3$.
Also we note that for the generic point $z$ of the coordinate line
$L_{i\, i+1}$ (resp. the line $M_{i}$), we have $e(p_{2}^{-1}(z))=3e(\mathbb{P}^{1})-2=4$
(resp. $e(p_{2}^{-1}(z))=e(\mathbb{P}^{1})=2$ ). We also note that
$M_{i}\cap M_{j}=\phi$. Now inspecting the Fig.$\,\ref{fig:U-Hessian}$
carefully, we can evaluate the Euler number as\[
e(U_{sp})=e(H_{sp})+5\{e(\mathbb{P}^{2})-1\}+\{e(C_{E})-5\}\{e(\mathbb{P}^{1})-1\}=-5+10-5=0,\]
where we use the result $e(H_{sp})=-5$ obtained in Proposition \ref{pro:Ds3}.
\end{proof}

\subsection{The ramified covering $Y_{sp}\rightarrow H_{sp}$}

For the generic Hessian quintic $H$, we have found a ramified double
covering $Y\rightarrow H$ which gives us a smooth Calabi-Yau threefold
\cite[Theorem 3.14]{HoTa} with $h^{2,1}(Y)=h^{2,1}(X)=26$. The diagram
(\ref{eq:ReyeX-Y-diagram}) shows relations among the generic fibers
of the related families over the $26$-dimensional deformation space.
We obtain the diagram (\ref{eq:Xs-Ys-diagram}) by restricting (\ref{eq:ReyeX-Y-diagram})
to the special families over $\mathbb{P}^{1}$. 

We recall that our special family of the Hessian is defined by $H_{sp}=\bigl\{[\lambda]\in\mathbb{P}_{\lambda}^{4}\mid$${\rm det}A_{\lambda}=0\bigr\}$
with $A_{\lambda}=\sum_{k=1}^{5}\lambda_{k}A_{k}$ for generic $a=b\in\mathbb{C}^{*}$.
\begin{defn}
\label{def:def-Ysp}Consider the weighted projective space $\mathbb{P}^{9}(2^{5},1^{5})$
with its (weighted) homogeneous coordinate $[\xi:\lambda]=[\xi_{1}:...:\xi_{5}:\lambda_{1}:...:\lambda_{5}]$.
We denote by $\bar{\varphi}_{\lambda}:\mathbb{P}^{9}(2^{5},1^{5})\dashrightarrow\mathbb{P}_{\lambda}^{4}$
the natural projection to the second half of the components. We define
$Y_{sp}\subset\mathbb{P}^{9}(2^{5},1^{5})$ by the following (weighted)
homogeneous equations:\begin{equation}
\begin{array}{l}
\xi_{i}\xi_{j}=\Delta(A_{\lambda})_{ij}\;\;(1\leq i,j\leq5),\\
A_{\lambda}\xi=\mathbf{0},\end{array}\label{eq:defeq-Ysp}\end{equation}
where $\Delta(A_{\lambda})_{ij}(=:\Delta_{ij})$ represents the $ij$-cofactor
of the symmetric matrix $A_{\lambda}$. $\:[]$
\end{defn}
We read the above definition from \cite[Sect.3]{Catanese}. We note
that if ${\rm det}(A_{\lambda})\not=0$, then we have $\xi=0$ hence
$\Delta(A_{\lambda})_{ij}=0\,(1\leq i,j\leq5)$, which is a contradiction.
Therefore we have a map $\varphi_{\lambda}:Y_{sp}\rightarrow H_{sp}$
which follows from $\bar{\varphi}{}_{\lambda}:\mathbb{P}^{9}(2^{5},1^{5})\dashrightarrow\mathbb{P}_{\lambda}^{4}$. 
\begin{prop}
1) The map $\varphi_{\lambda}:Y_{sp}\rightarrow H_{sp}$ is surjective.
Moreover it is a double covering ramified along the singular loci
of $H_{sp}$ (see Proposition \ref{pro:singular-loci-Hsp}). 

\noindent2) The singular loci of $Y_{sp}$ consist of 5 lines $\tilde{L}_{i\, i+1}(i=1,...,5)$
of $A_{1}$-singularities, where $\tilde{L}_{i\, i+1}$ represents
the coordinate line $L_{i\, i+1}\subset H_{sp}$ considered in $Y_{sp}$. 

\noindent3) The Euler number of $Y_{sp}$ is given by $e(Y_{sp})=-10.$\end{prop}
\begin{proof}
1) The singular loci $Sing\, H_{sp}$ of $H_{sp}$ coincides with
the loci with $rk(A_{\lambda})\leq3$ due to 2) of Proposition \ref{pro:singular-loci-Hsp}.
If $rk(A_{\lambda})\leq3$, then $\Delta_{ij}=0\,(1\leq i,j\leq5)$.
Hence, from (\ref{eq:defeq-Ysp}), we have $\xi_{i}=0$, which implies
that $\varphi_{\lambda}$ is bijective over $Sing\, H_{sp}$. Now
take a point $[\lambda]\in H_{sp}$ such that $rk(A_{\lambda})=4$,
then we have $rk(\Delta_{ij})=1$ for the matrix of the cofactors.
Then there exists $\xi$ such that $(\Delta_{ij})=(\xi_{i}\xi_{j})$.
We may assume, without loss of generality, that $\Delta_{11}\not=0$.
Solving $\xi_{1}^{2}=\Delta_{11},$ we obtain $\varphi_{\lambda}^{-1}([\lambda])=\bigl\{[\pm\frac{\Delta_{11}}{\sqrt{\Delta_{11}}}:...:\pm\frac{\Delta_{15}}{\sqrt{\Delta_{11}}}:\lambda_{1}:...:\lambda_{5}]\bigr\}$
. This completes the proof.

\noindent2) Let us denote by $\mathcal{S}_{i}(i=1,...,5)$ the affine
subsets $\{[\xi,\lambda]\mid\lambda_{i}\not=0\}$ in the weighted
projective space $\mathbb{P}^{9}(2^{5},1^{5})$. We note that $\mathcal{S}_{i}\simeq\mathbb{C}^{9}(i=1,...,5)$
are in the smooth loci of the weighted projective space, and $Y_{sp}$
is contained in the union $\cup_{i}\mathcal{S}_{i}$ of these affine
subsets. Here we study the singular loci of $Y_{sp}$ on the affine
subset $\mathcal{S}_{5}$ with the affine coordinates $[\xi_{1}:...:\xi_{5}:\lambda_{1}:...:\lambda_{4}:1]$,
but the calculations on the other $\mathcal{S}_{i}$'s are quite parallel
due to the symmetry of the defining equations $Y_{sp}$ and $H_{sp}$. 

We obtain 20 equation from the defining equations (\ref{def:def-Ysp})
expressed by the affine coordinates of $\mathcal{S}_{5}$. We observe
that one of the 5 equations from $A_{\lambda}\xi=0$ can be solved
by $\xi_{5}=-\xi_{1}-\lambda_{4}\xi_{4}$. Eliminating $\xi_{5}$
by this, we have 19 equations for $(\xi_{1},\xi_{2},\xi_{3},\xi_{4},\lambda_{1},\lambda_{2},\lambda_{3},\lambda_{4})$
and make the Jacobian matrix of size $8\times19$. Since $\dim Y_{sp}$=3,
the Jacobian ideal of $Y_{sp}$ is generated by $5\times5$ minors
of the Jacobi matrix. By studying this Jacobian ideal in a straightforward
way, we obtain $\tilde{L}_{45}\cup\tilde{L}_{51}$ for the singular
loci of $Y_{sp}$ restricted on $\mathcal{S}_{5}$. Combined with
the similar calculations for the other $\mathcal{S}_{i}$'s, we determine
the singular loci of $Y_{sp}$ as claimed. 

To determine the type of singularity, we work with the affine coordinates
$[\xi_{1}:...:\xi_{5}:$ $\lambda_{1}:...:\lambda_{4}:1]$ and focus
on the line $\tilde{L}_{51}$. Note that the corresponding line $L_{51}\subset H_{sp}$
intersects with the line $M_{2}$ at $[\lambda]=[a^{2}:0:0:0:1].$
We describe the local geometry around the point $[\xi:\lambda]=[0^{5}:a^{2}:0:0:0:1]$
by introducing the affine coordinates by $[x_{1}:...:x_{5}:a^{2}+y_{1}:y_{2}:y_{3}:y_{4}:1]$.
We work on the defining equations (\ref{def:def-Ysp}) in the local
ring $\mathbb{C}[x_{1},x_{2},x_{3},x_{4},x_{5,}y_{1},y_{2},y_{3},y_{4}]_{m_{0}}$
at the origin. By inspecting the defining equations of $Y_{sp}$,
we see that $x_{1},x_{2},y_{3}$ as well as $x_{5}$ can be solved
by other variables in the local ring. After eliminating these variables,
we study the (eliminated) ideal of $Y_{sp}$ in the local ring $\mathbb{C}[x_{3},x_{4},y_{1},y_{2},y_{4}]_{m'_{0}}$
at the origin. It turns out that the local geometry of $Y_{sp}$ is
described by the two equations $g_{1}=g_{2}=0$ with $g_{1},g_{2}\in\mathbb{C}[x_{3},x_{4},y_{1},y_{2},y_{4}]_{m'_{0}}$
and $g_{2}$ is quadric with respect to $y_{4}$. Solving the quadric
equation $g_{2}=0$, we eliminate $y_{4}$ from $g_{1}$. Choosing
suitable branch of the solutions, we obtain\[
g_{1}=x_{3}x_{4}+x_{4}^{2}+y_{1}y_{2}^{2}+y_{2}x_{3}^{2}-y_{1}y_{2}x_{3}^{2}+\cdots,\]
where $\cdots$ represents the higher order terms of the total degree
greater than four. From this local equation, we read the singularities
along the line $\tilde{L}_{51}$ are of $A_{1}$-type generically.
(We may also observe that, as before, the blowing-up along $\tilde{L}_{51}$
introduces a smooth conic bundle with reducible fiber at the origin,
i.e., over the intersection point of $L_{51}$ and $M_{2}$. ) 

\noindent3) The singular loci $Sing\, H_{sp}$, i.e., the ramification
loci of $Y_{sp}\rightarrow H_{sp}$, consists of 10 lines $L_{i\, i+1},M_{i}$
and the curve $C_{E}$ which intersect as shown in Fig.$\,\ref{fig:U-Hessian}$.
Combined with $e(H_{sp})=-5$ in Proposition \ref{pro:Ds3}, we have
\[
e(Y_{sp})=2\left\{ e(H_{sp})-e(Sing\, H_{sp})\right\} +e(Sing\, H_{sp})=2\times(-5-0)+0=-10.\]
\end{proof}
\begin{rem}
\label{rem:def-covering-Y}The $A_{1}$-singularities along the ramification
loci of $H_{sp}$ are resolved by the covering $Y_{sp}\rightarrow H_{sp}.$
This is, in fact, a general property valid for a normal variety $Y_{sp}$.
If we consider the generic Hessian quintic $H$ with a regular linear
system $P=|A_{1},A_{2},...,A_{k}|$ and $A_{\lambda}=\sum_{k=1}^{5}\lambda_{k}A_{k}$,
then the singular loci of $H$ consist of a curve of $A_{1}$-singularity.
Hence, Definition \ref{def:def-Ysp} applied for generic Hessian quintic
$H$ gives us a smooth Calabi-Yau threefold ramified along the curve,
which is $Y$ in the theorem of the subsection \ref{sub:summary-Reye-I},
see also the diagram (\ref{eq:SHdiagam1}). This explicit realization
of the geometry $Y$ in $\mathbb{P}^{4}(2^{5},1^{5})$ should have
corresponding descriptions in physics \cite{Hori}, \cite{Jo}. \hfill {[}{]} 
\end{rem}
As is shown in the diagram (\ref{eq:Xs-Ys-diagram}), we expect a
crepant resolution $\rho:Y_{sp}^{*}\rightarrow Y_{sp}$ with a Calabi-Yau
manifold $Y_{sp}^{*}$ which satisfies $(h^{1,1}(Y_{sp}^{*}),h^{2,1}(Y_{sp}^{*}))=(h^{1,1}(X^{*}),$
$h^{2,1}(X^{*}))=(26,1)$ for the Hodge numbers. The existence of
$Y_{sp}^{*}$ with these expected properties is left for future study.

\vspace{0.5cm}

{\footnotesize Graduate School of Mathematical Sciences, University
of Tokyo, Meguro-ku, Tokyo 153-8914, Japan }{\footnotesize \par}

{\footnotesize e-mail addresses: hosono@ms.u-tokyo.ac.jp, takagi@ms.u-tokyo.ac.jp}

\vspace{2cm}

\begin{thebibliography}{Yau}
\bibitem[AEvSZ]{AlmkcistEtAl}G. Almkvist, C. van Enckevort, D. van
Straten, W. Zudilin, \textit{Tables of Calabi--Yau equations}, arXiv:math/0507430.

\bibitem[AGV]{AGV}V.I. Arnold, S.M. Gusein-Zade and A.N. Varchenko,
\textit{Singularities of Differential Maps}, Volume I. Birkh\"auser
(1985).

\bibitem[BaN]{Barth-Nieto}W. Barth and I. Nieto, \textit{Abelian
Surfaces of type (1,3) and Quartic Surfaces with 16 Skew Lines}, J.
Alg. Geom. \textbf{3} (1994) 173-222.

\bibitem[Ba]{BatPdual}V. Batyrev, \textit{Dual polyhedra and mirror
symmetry for Calabi-Yau hypersurfaces in toric varieties}, J. Alg.
Geom. 3(1994), 493-535.

\bibitem[BaBo]{BatBo}V. Batyrev and L. Borisov, \textit{On Calabi-Yau
Complete Intersections in Toric Varieties}, Higher-dimensional complex
varieties (Trento, 1994), 39--65, de Gruyter, Berlin, 1996. 

\bibitem[BaCo]{Bat-Cox}V. Batyrev and D.A. Cox, \textit{On the Hodge
structure of projective hypersurfaces in toric varieties}, Duke Math.
J. 75 (1994) 293--338. 

\bibitem[BeH]{BergHub}P. Berglund and T. H\"ubsch, \textit{A generalized
construction of mirror manifolds}, Nuclear Phys. B 393 (1993), no.
1-2, 377\textendash{}391.

\bibitem[BoCa]{BoCa}L. Borisov and A. Caldararu, \textit{The Pfaffian-Grassmannian
derived equivalence}, J. Algebraic Geom. 18 (2009), no. 2, 201--222.math/0608404. 

\bibitem[CdOGP]{Candelas1}P. Candelas, X.C. de la Ossa, P.S. Green,
and L.Parkes, \textit{A pair of Calabi-Yau manifolds as an exactly
soluble superconformal theory}, Nucl.Phys. B356(1991), 21-74. 

\bibitem[Ca]{Catanese} F. Catanese, \textit{Commutative algebra methods
and equations of regular surfaces}. Algebraic geometry, Bucharest
1982 (Bucharest, 1982), 68--111, Lecture Notes in Math., 1056, Springer,
Berlin, 1984.

\bibitem[Co]{Cossec}F. Cossec, \textit{Reye congruence}, Transactions
of the A.M.S. 280 (1983), 737--751.

\bibitem[DGPS]{Sing3}W. Decker, G.-M. Greuel, G. Pfister and H. Sch\"onemann,
Singular 3-1-3 --- A computer algebra system for polynomial computations,
http://www.singular.uni-kl.de (2011).

\bibitem[DGJ]{DGJ}C. Doran, B. Greene and S. Judes, \textit{Families
of quintic Calabi-Yau 3-folds with discrete symmetries}, Comm. Math.
Phys. 280 (2008), no. 3, 675\textendash{}725

\bibitem[EvS]{Enckv-Straten}C. Enckevort and D. van Straten, \textit{Monodromy
calculations of fourth order equations of Calabi-Yau type},in Mirror
symmetry. V, 539--559, AMS/IP Stud. Adv. Math., 38, Amer. Math. Soc.,
Providence, RI, 2006, math.AG/0412539. 

\bibitem[FK]{FK}E. Freitag and R. Kiehl, \textit{Etale Cohomology
and The Weil Conjecture}, Springer-Verlag, Berlin and New York (1988).

\bibitem[GS]{M2}D. R. Grayson and M. E. Stillman, Macaulay2, a software
system for research in algebraic geometry, Available at http://www.math.uiuc.edu/Macaulay2/. 

\bibitem[vGN]{vanG}B. van Geemen and N. O. Nygaard, On the Geometry
and Arithmetics of Some Siegel Modular Threefolds, Jour. of Number
Theory 53(1995), 45--87.

\bibitem[Gep]{Gepner}D.Gepner, \textit{Exactly solvable string compactifications
on manifolds of SU(n) holonomy}, Phys.Lett.199B(1987)380. 

\bibitem[GP]{Greene-Pressor}B.R.Greene and M.R.Plesser, \textit{Duality
in Calabi-Yau moduli space}, Nucl.Phys. B338 (1990) 15-37.

\bibitem[GKZ]{GKZ1}I.M. Gel'fand, A. V. Zelevinski, and M.M. Kapranov,
\textit{Equations of hypergeometric type and toric varieties}, Funktsional
Anal. i. Prilozhen. 23 (1989), 12--26; English transl. Functional
Anal. Appl. 23(1989), 94--106. 

\bibitem[Gr]{Griffiths}P. A. Griffiths, On the periods of certain
rational integrals. I, II. Ann. of Math. (2) 90 (1969), 460-495; ibid.
(2) 90 1969 496\textendash{}541.

\bibitem[Har]{Har}R. Hartshorne, \textit{Algebraic geometry}. Graduate
Texts in Mathematics 52, Springer-Verlag, New York, Heidelberg, Berlin,
1977.

\bibitem[Ho1]{IIAmonod}S. Hosono,\textit{ Local Mirror Symmetry and
Type IIA Monodromy of Calabi-Yau Manifolds}, Adv. Theor. Math. Phys.
4 (2000), 335--376. 

\bibitem[Ho2]{CentralCh}S. Hosono, \textit{Central charges, symplectic
forms, and hypergeometric series in local mirror symmetry}, in ``Mirror
Symmetry V'', S.-T.Yau, N. Yui and J. Lewis (eds), IP/AMS (2006),
405--439. 

\bibitem[HKTY]{HKTY}S. Hosono, A. Klemm, S. Theisen and S.-T. Yau,
\textit{Mirror Symmetry, Mirror Map and Applications to complete Intersection
Calabi-Yau Spaces}, Nucl. Phys. B433(1995)501--554. 

\bibitem[HLY]{HLY}S. Hosono, B.H. Lian, and S.-T. Yau, \textit{GKZ-Generalized
hypergeometric systems in mirror symmetry of Calabi-Yau hypersurfaces},
Commun. Math. Phys. 182 (1996) 535--577. 

\bibitem[HT]{HoTa} S. Hosono and . H. Takagi, \textit{Mirror Symmetry
and Projective Geometry of Reye congruences I}, preprint arXiv:1101.2746(2011)
to appear in J. Alg. Geom.

\bibitem[HW]{HW}F. Hidaka, K. Watanabe, \textit{Normal Gorenstein
surfaces with ample anti-canonical divisor}, Tokyo J. Math. 4 (1981),
no. 2, 319\textendash{}330.

\bibitem[Hor]{Hori}K. Hori, \textit{Duality In Two-Dimensional (2,2)
Supersymmetric Non-Abelian Gauge Theories}, arXiv:1104.2853, hep-th.(2011).

\bibitem[Hu]{Huleck}K. Hulek, Projective geometry of elliptic curves,
Ast\'erisque No. 137 (1986), 143 pp.

\bibitem[HSvGvS]{HulekEtAl}K. Hulek, J. Spandaw, B. van Geemen, and
D. van Straten, \textit{The modularity of the Barth-Nieto quintic
and its relatives}, Adv. Geom. 1 (2001) 263\textendash{}289. 

\bibitem[JKLMR]{Jo}H. Jockers, V. Kumar, J. M. Lapan, D. R. Morrison,
M. Romo, \textit{Nonabelian 2D Gauge Theories for Determinantal Calabi-Yau
Varieties}, arXiv:hep-th/1205.3192.

\bibitem[Ko]{Ko}M. Kontsevich, \textit{Homological algebra of mirror
symmetry}, Proceedings of the Interna- tional Congress of Mathematicians
(Z\"urich, 1994) Birkh\"auser (1995) pp. 120 --139.

\bibitem[Ku]{Kuz}A. Kuznetsov, \textit{Homological projective duality
for Grassmannians of lines}, arXiv:math/0610957. 

\bibitem[Mo]{DMo}D. R. Morrison, \textit{Picard-Fuchs equations and
mirror maps for hypersurfaces,} in \textit{{}``Essays on Mirror Manifolds}'',
Ed. S.-T.Yau, International Press, Hong Kong (1992) 241-264.

\bibitem[Ol]{Oliva}C. Oliva, \textit{Algebraic cycles and Hodge theory
on generalized Reye congruences}, Compositio Math. \textbf{\textit{92}}(1994),
1--22. 

\bibitem[Ro]{Rod}E.A. R\text{\o}dland, \textit{The Pfaffian Calabi-Yau,
its Mirror and their link to the Grassmannian $G(2,7)$}, Compositio
Math. 122 (2000), no. 2, 135--149, math.AG/9801092. 

\bibitem[Ty]{Tyurin}A.N. Tyurin, \textit{On intersections of quadrics},
Russian Math. Surveys \textbf{30} (1975), 51--105. 

\bibitem[Yau]{essay} \textit{{}``Essays on Mirror Manifolds}'',
Ed. S.-T.Yau, International Press, Hong Kong (1992).

\end{thebibliography}
\end{document}